\documentclass[12pt]{amsart}

\usepackage{amsmath}
\usepackage{tabu}
\usepackage{tikz}
\usepackage[margin=1in]{geometry}
\usepackage[matrix,arrow,curve,frame]{xy}
 \usepackage{hyperref}
\usepackage{amsthm}
\usepackage{amsfonts}
\usepackage{amssymb}
\usepackage{wasysym}
\usepackage{mathrsfs}
\usepackage{mathtools}

\definecolor{my-linkcolor}{rgb}{0.75,0,0}
\definecolor{my-citecolor}{rgb}{0.1,0.57,0}
\definecolor{my-urlcolor}{rgb}{0,0,0.75}
\hypersetup{
	colorlinks,
	linkcolor={my-linkcolor},
	citecolor={my-citecolor},
	urlcolor={my-urlcolor}
}

\title[Virasoro tensor categories]{Structure of Virasoro tensor categories at central charge $13-6p-6p^{-1}$ for integers $p > 1$}
 \author{Robert McRae and Jinwei Yang}
\date{}

 \address{Yau Mathematical Sciences Center, Tsinghua University, Beijing 100084, China}
  \email{rhmcrae@tsinghua.edu.cn}

 \address{School of Mathematical Sciences, Shanghai Jiao Tong University, Shanghai 200240, China}

 \email{jinwei2@sjtu.edu.cn}

 \subjclass{Primary 17B68, 17B69, 18M15, 81R10}

\newtheorem{thm}{Theorem}[section]
\newtheorem{cor}[thm]{Corollary}
\newtheorem{lem}[thm]{Lemma}
\newtheorem{prop}[thm]{Proposition}
\theoremstyle{definition}\newtheorem{defi}[thm]{Definition}
\theoremstyle{definition}\newtheorem{rem}[thm]{Remark}
\theoremstyle{definition}
\theoremstyle{definition}

\newcommand{\cE}{\mathcal{E}}
\newcommand{\cY}{\mathcal{Y}}
\newcommand{\cU}{\mathcal{U}}
\newcommand{\cV}{\mathcal{V}}
\newcommand{\cA}{\mathcal{A}}
\newcommand{\cR}{\mathcal{R}}
\newcommand{\cM}{\mathcal{M}}

\newcommand{\cX}{\mathcal{X}}
\newcommand{\cF}{\mathcal{F}}
\newcommand{\cI}{\mathcal{I}}
\newcommand{\cJ}{\mathcal{J}}
\newcommand{\cL}{\mathcal{L}}
\newcommand{\cO}{\mathcal{O}}
\newcommand{\cP}{\mathcal{P}}
\newcommand{\cC}{\mathcal{C}}
\newcommand{\cG}{\mathcal{G}}
\newcommand{\cW}{\mathcal{W}}
\newcommand{\cZ}{\mathcal{Z}}
\newcommand{\til}{\widetilde}

\newcommand{\CC}{\mathbb{C}}
\newcommand{\ZZ}{\mathbb{Z}}
\newcommand{\NN}{\mathbb{N}}
\newcommand{\RR}{\mathbb{R}}
\newcommand{\QQ}{\mathbb{Q}}

\newcommand{\Id}{\mathrm{Id}}

\newcommand{\tens}{\boxtimes}
\newcommand{\vac}{\mathbf{1}}
\newcommand{\ind}{\mathrm{Ind}}
 \DeclareMathOperator{\im}{Im}
 \DeclareMathOperator{\coker}{Coker}
 
 \DeclareMathOperator{\tr}{Tr}
 \DeclareMathOperator{\rep}{Rep}
 \let\ker\relax
 \let\hom\relax
 \DeclareMathOperator{\ker}{Ker}
 \DeclareMathOperator{\hom}{Hom}

\begin{document}
\bibliographystyle{alpha}

\numberwithin{equation}{section}

 \begin{abstract}
 Let $\cO_c$ be the category of finite-length central-charge-$c$ modules for the Virasoro Lie algebra whose composition factors are irreducible quotients of reducible Verma modules. Recently, it has been shown that $\cO_c$ admits vertex algebraic tensor category structure for any $c\in\CC$. Here, we determine the structure of this tensor category when $c=13-6p-6p^{-1}$ for an integer $p>1$. For such $c$, we prove that $\cO_{c}$ is rigid, and we construct projective covers of irreducible modules in a natural tensor subcategory $\cO_{c}^0$. We then compute all tensor products involving irreducible modules and their projective covers.
Using these tensor product formulas, we show that $\cO_c$ has a semisimplification which, as an abelian category, is the Deligne product of two tensor subcategories that are tensor equivalent to the Kazhdan-Lusztig categories for affine $\mathfrak{sl}_2$ at levels $-2+p^{\pm 1}$. Next, as a straightforward consequence of the braided tensor category structure on $\cO_c$ together with the theory of vertex operator algebra extensions, we rederive known results for triplet vertex operator algebras $\cW(p)$, including rigidity, fusion rules, and construction of projective covers. Finally, we prove a recent conjecture of Negron that $\cO_c^0$ is braided tensor equivalent to the $PSL(2,\CC)$-equivariantization of the category of $\cW(p)$-modules.
\end{abstract}

\maketitle

\tableofcontents

\section{Introduction}

The Virasoro algebra $\cV ir$ is the unique non-trivial one-dimensional central extension of the Lie algebra of polynomial vector fields on the circle. It is foundational in algebraic approaches to two-dimensional conformal field theory, and it is the source of one of the first-constructed families of vertex operator algebras \cite{FZ1}. As with all Lie algebras, the full category of $\cV ir$-modules is a symmetric tensor category, but for applications in physics, one restricts to categories of $\cV ir$-modules with a fixed central charge: this is the scalar by which the canonical central element of $\cV ir$ acts. The correct tensor product operation on such categories then becomes the fusion product of conformal field theory, which can be defined mathematically in terms of vertex algebraic intertwining operators (see for example \cite{HLZ3}).

At central charge $c = c_{p,q}= 13-6(\frac{p}{q} +\frac{q}{p})$ for $p, q \geq 2$ and $\gcd(p, q) = 1$, the $\cV ir$-module category of primary interest, corresponding to ``minimal models'' in rational conformal field theory \cite{BPZ}, is the representation category of the simple Virasoro vertex operator algebra $V_c$. The algebra $V_c$ is rational \cite{Wa} and $C_2$-cofinite \cite{Zh, DLM}, and thus its representations form a modular tensor category \cite{Hu_Vir_tens, Hu_rigid}. For all other central charges, however, the Virasoro vertex operator algebras are neither rational nor $C_2$-cofinite, and only recently has there been much progress in understanding the tensor structure of their representations.

In \cite{CJORY}, it was shown that for any $c\in\CC$, the category $\cO_c$ of $C_1$-cofinite grading-restricted generalized modules for the universal Virasoro vertex operator algebra of central charge $c$ is the same as the category of finite-length $\cV ir$-modules whose composition factors are irreducible quotients of reducible Verma modules of central charge $c$. As a consequence, it was shown that $\cO_c$ satisfies the conditions of Huang-Lepowsky-Zhang's vertex tensor category theory \cite{HLZ1}-\cite{HLZ8}, and thus $\cO_c$ is a braided tensor category as described in \cite{HLZ8}. Some details of the tensor structure on $\cO_c$ are known for the following $c$:
\begin{enumerate}
 \item For $c=13-6t-6t^{-1}$ with $t\notin\QQ$, it was shown in \cite{CJORY} that $\cO_c$ is a rigid semisimple tensor category, with tensor products of irreducible modules given by the fusion rules calculated previously in \cite{FZ2} using a Zhu algebra approach.

 \item For $c=1$, tensor products of simple modules in $\cO_1$ were determined in \cite{McR} using the fusion rule calculations of \cite{Mi}, and it was shown in \cite[Remark 4.4.6]{CMY2} using results from \cite{McR} that $\cO_1$ is rigid. The full category $\cO_1$ is not semisimple, but its simple objects generate a semisimple tensor subcategory, namely, the category of $C_1$-cofinite unitary modules for the unitary vertex operator algebra $V_1$.

 \item For $c=13-6p-6p^{-1}$ with $p > 1$ an integer and for $c=25$, fusion rules for irreducible modules in $\cO_c$ were calculated in \cite{Lin} and \cite{OH}, respectively. However, since these categories are not semisimple, fusion rules are not enough to identify tensor products of irreducible modules in $\cO_c$. Rigidity for these categories has also remained open.

\end{enumerate}

In this work, we present a comprehensive analysis of the tensor category $\cO_c$ at central charge $c=c_{p,1}=13-6p-6p^{-1}$ for integers $p > 1$; especially, we prove rigidity and compute all tensor products of irreducible modules. The simple Virasoro vertex operator algebras $V_c$ at these central charges occur as subalgebras of many of the best-known vertex operator algebras in logarithmic conformal field theory, including the singlet algebras \cite{Ka, A, AM_log_intw, CF, CMR, CMY2}, triplet algebras \cite{FHST, FGST1, FGST2, GR, AM_trip, AM_log_mods, NT, TW, CGR}, and logarithmic $\mathcal{B}_p$ algebras \cite{CRW, AuCKR, ACGY}. Reflecting the non-semisimplicity of the Virasoro zero-mode $L_0$ in logarithmic conformal field theory (which leads to logarithmic singularities in correlation functions), the Virasoro categories $\cO_{c_{p,1}}$ are neither semisimple nor finite.

Although the singlet and triplet algebra extensions of $V_c$ have been studied fairly extensively by mathematicians, most work on the Virasoro algebra itself at central charge $13-6p-6p^{-1}$ has appeared in the physics literature, in the study of ``logarithmic minimal models'' denoted $\mathcal{LM}(1,p)$. Starting with work of Gaberdiel and Kausch \cite{GaK}, indecomposable modules at these central charges have been constructed and fusion products have been predicted using a variety of methods \cite{PRZ, RP, RS, KyR, BFGT, BGT, Ra, MRR}. Comparison of these works with our results summarized in Theorem \ref{thm:main_thm} below shows that the vertex algebraic tensor category $\mathcal{O}_c$ can be viewed as a rigorous mathematical setting for logarithmic minimal models. For example, the formula in Theorem \ref{thm:main_thm}(3) for the tensor product of irreducible $V_c$-modules agrees with the fusion product conjecture in \cite[Equation 4.1]{GaK}. More precisely, the mathematics of $\mathcal{LM}(1,p)$ is captured by the tensor structure on the subcategory $\mathcal{O}_c^0$ of $\cO_c$ mentioned in Theorem \ref{thm:main_thm}(2), which we introduced in order to obtain projective covers of irreducible modules. This turns out to be the smallest tensor subcategory of $\cO_c$ that contains all irreducible modules.

At central charge $c=c_{p,1}$, the Virasoro category $\cO_{c}$ has simple modules labeled $\cL_{r,s}$ for $r,s\in\ZZ$ such that $r\geq 1$ and $1\leq s\leq p$. Tensor products of these $V_c$-modules are described in the following theorem, which summarizes our main results:
\begin{thm}\label{thm:main_thm}
Let $V_c$ denote the simple Virasoro vertex operator algebra of central charge $c=13-6p-6p^{-1}$ for an integer $p > 1$. Then:
 \begin{enumerate}
  \item The tensor category $\cO_c$ of $C_1$-cofinite grading-restricted generalized $V_c$-modules is rigid and ribbon, with duals given by the contragredient modules of \cite{FHL} and natural twist isomorphism $\theta=e^{2\pi iL_0}$.

  \item Every irreducible module $\cL_{r,s}$ in $\cO_c$ has a projective cover $\cP_{r,s}$ in a natural tensor subcategory $\cO_c^0$ of $\cO_c$.

  \item Tensor products of the irreducible modules in $\cO_c$ are as follows: 
  \begin{equation*}
   \cL_{r,s}\boxtimes \cL_{r',s'}   \cong \bigoplus_{\substack{k = |r-r'|+1\\ k+r+r' \equiv 1\; ({\rm mod}\; 2)}}^{r+r'-1} \bigg(\bigoplus_{\substack{\ell = |s-s'|+1\\ \ell+s+s' \equiv 1\; (\mathrm{mod}\; 2)}}^{\min(s+s'-1, 2p-1-s-s')} \cL_{k, \ell} \oplus \bigoplus_{\substack{\ell = 2p+1-s-s'\\ \ell+s+s' \equiv 1\; (\mathrm{mod}\; 2)}}^{p} \cP_{k, \ell}\bigg)
  \end{equation*}
for $r, r'\geq 1$ and $1\leq s,s'\leq p$, with sums taken to be empty if the lower bound exceeds the upper bound.
 \end{enumerate}
\end{thm}

The proof of Theorem \ref{thm:main_thm} begins in Section \ref{sec:first_fus}, where we largely determine which composition factors of the tensor products $\cL_{1,2}\tens\cL_{r,s}$ show up in the lowest conformal weight spaces of the tensor product modules. To do so, we use the Zhu algebra approach developed in \cite{FZ1, Li, FZ2, HY}, among other references, but our calculations also resemble those done by physicists to compute fusion products using the Nahm-Gaberdiel-Kausch algorithm \cite{Na, GaK}. See \cite{KR} for a comparison of mathematicians' and physicists' approaches to fusion products; note that our work in Section \ref{sec:first_fus} as well as later in Proposition \ref{prop:P1_structure} recovers (in greater generality and more systematically) the results of the sample calculations in \cite[Sections 7 and 8]{KR}.

To fully determine tensor products in $\cO_c$, we use rigidity. To prove that $\cO_{c}$ is rigid, we first prove that $\cL_{1,2}$ is rigid (and self-dual) using explicit formulas for compositions of intertwining operators, obtained from solutions to Belavin-Polyakov-Zamolodchikov equations (Theorem \ref{rigidityofl12}); the method is the same as in \cite{TW} for the triplet algebras and in \cite{CMY2} for the singlet algebras. Next, the modules $\cL_{r,1}$, $r\geq 1$, are the irreducible $V_c$-modules appearing in the decomposition of the doublet abelian intertwining algebra \cite{AM_doub} as a $V_c$-module. As $V_c$ is an $SU(2)$-fixed point subalgebra of the doublet, results in \cite{McR} show that the modules $\cL_{r,1}$ generate a tensor subcategory of $\cO_c$ that is braided tensor equivalent to an abelian $3$-cocycle twist of $\rep SU(2)$ (Theorem \ref{thm:Lr1_fus_rules}). Consequently, these $V_c$-modules are rigid. Once we know that the modules $\cL_{1,2}$ and $\cL_{r,1}$ are rigid, we can compute tensor products involving these modules using the preliminary results of Section \ref{sec:first_fus}. We show that all remaining irreducible modules in $\cO_c$ occur as direct summands in repeated tensor products of the rigid modules $\cL_{1,2}$ and $\cL_{r,1}$, and thus are rigid. Finally, we use \cite[Theorem~4.4.1]{CMY2} to extend rigidity from irreducible modules to all finite-length modules in $\cO_c$.

The modules $\cL_{r,s}$ do not have projective covers in the full category $\cO_{c}$ since their associated Verma modules have infinite length. Thus  to obtain projective covers, it is indeed necessary to introduce the tensor subcategory $\cO_c^0$, which contains all irreducible modules in $\cO_c$. We can define $\cO_c^0$ in several ways: it turns out to be the tensor subcategory of $\cO_c$ (closed under tensor products and subquotients) generated by $\cL_{1,2}$, but it is more useful to define $\cO_c^0$ as the M\"{u}ger centralizer of the semisimple subcategory of $\cO_c$ that has simple objects $\cL_{2n+1,1}$, $n\in\NN$. Equivalently, this is the subcategory of modules in $\cO_c$ that induce to ordinary modules for the triplet vertex operator algebra $\cW(p)$, an infinite-order extension of $V_c$.

 In $\cO_c^0$, the irreducible modules $\cL_{r,p}$ are already projective (Theorem \ref{projoflrp}), and then we construct length-$3$ projective covers $\cP_{1,s}$ from $\cL_{1,p}$ recursively (Theorem \ref{thm:P1s_structure}), using the methods of \cite[Section 5.1]{CMY2}. Finally, we show that $\cP_{r,s}=\cL_{r,1}\tens\cP_{1,s}$ is a length-$4$ projective cover of $\cL_{r,s}$ for $r\geq 2$ (Theorem \ref{thm:Prs_structure}). After constructing all projective covers, we complete the proof of the tensor product formula in Theorem \ref{thm:main_thm}(3), and we also determine all tensor products of the  projective modules with irreducible modules and with each other (see the details in Theorem \ref{generalfusionrules}).

In Section \ref{subsec:ss}, we investigate relations between $\cO_c$ and representations of the affine Lie algebra $\widehat{\mathfrak{sl}}_2$ at levels $-2+p^{\pm1}$ (note that $V_c$ is the $W$-algebra obtained via quantum Drinfeld-Sokolov reduction from the universal affine vertex operator algebras for $\mathfrak{sl}_2$ at both levels \cite{FFr}; see also \cite[Chapter 15]{FB}). First, the tensor product formulas of Theorem \ref{generalfusionrules} show that $\cO_c$ has a semisimplification which is a ribbon category with simple objects $\cL_{r,s}$ for $r \geq 1$ and $1 \leq s \leq p-1$. As an abelian category, the semisimplification is the Deligne product of two subcategories: $\cO_{c}^L$ containing the modules $\cL_{r,1}$ for $r\geq 1$, and $\cO_c^R$ containing the modules $\cL_{1,s}$ for $1\leq s\leq p-1$. We then use \cite{ACGY} to show that $\cO_{c}^L$ is braided tensor equivalent to the Kazhdan-Lusztig category $KL_{-2+1/p}(\mathfrak{sl}_2)$ of $\widehat{\mathfrak{sl}}_2$-modules at level $-2+p^{-1}$, while we use the main theorem of \cite{KW} to show that $\cO_c^R$ is tensor equivalent to the $\widehat{\mathfrak{sl}}_2$-module category $KL_{-2+p}(\mathfrak{sl}_2)$. 

Note that $KL_{-2+p}(\mathfrak{sl}_2)$ is a modular tensor category since the simple affine vertex operator algebra of $\mathfrak{sl}_2$ at level $-2+p$ is rational and $C_2$-cofinite. The corresponding universal affine vertex operator algebra, however, has a non-semisimple $C_1$-cofinite module category; it would be interesting to see if this category bears any relation to the non-semisimple Virasoro category $\cO_c$. There is in fact a Kazhdan-Lusztig-type tensor equivalence conjectured in \cite{BFGT, BGT} between $\cO_c^0$ and a module category for the Lusztig limit of quantum $\mathfrak{sl}_2$ at the root of unity $e^{\pi i/p}$; see also \cite[Conjecture 11.4]{Ne} for a reformulation of this conjecture. After the initial version of the present paper was posted on arXiv, this conjecture was proved in \cite[Theorem 10.1]{GN}; the proof heavily used the tensor structure on $\cO_c^0$ deduced here.

We conclude this paper by applying our results, together with the vertex operator algebra extension theory of \cite{HKL, CKM, CMY1}, to the triplet vertex operator algebra extension $\cW(p)\supseteq V_c$. Using the rigid tensor category structure on $\cO_c$, we can rather quickly derive rigidity of the tensor category $\cC_{\cW(p)}$ of $\cW(p)$-modules, tensor product formulas in $\cC_{\cW(p)}$, and a construction of the projective covers of irreducible $\cW(p)$-modules. The only properties of $\cW(p)$ that we need come from \cite{AM_trip}: the classification of irreducible $\cW(p)$-modules and their decompositions as direct sums of $V_c$-modules, as well as some of the structure of the Zhu algebra of $\cW(p)$. Our results on $\cW(p)$ recover those obtained in \cite{AM_log_mods, NT, TW}. Our tensor-categorical approach especially provides an alternative to the technical construction of projective covers for irreducible $\cW(p)$-modules outlined in \cite{NT}. Note that since every vertex operator algebra has a built-in Virasoro subalgebra, vertex operator algebra extension techniques could be used to study the modules for many other vertex operator algebras. For example, the results on singlet algebras recently obtained in \cite{CMY2} could also be recovered from the structure of $\cO_c$.

Finally, we use our results together with ideas from \cite{McR2} to prove a precise relationship conjectured in \cite[Conjecture 11.6]{Ne} between the tensor categories $\cC_{\cW(p)}$ and $\cO_c^0$. It was shown in \cite{ALM} that the full automorphism group of $\cC_{\cW(p)}$ is $PSL(2,\CC)$, with fixed-point subalgebra $V_c$. Consequently, there is a braided tensor category $(\cC_{\cW(p)})^{PSL(2,\CC)}$, called the equivariantization of $\cC_{\cW(p)}$, whose objects are $\cW(p)$-modules equipped with a suitably compatible $PSL(2,\CC)$-action. Then an easy extension of \cite[Theorem 4.17]{McR2} (which was proved in a finite group setting) shows that there is a braided tensor equivalence from $\cO_c^0$ to $(\cC_{\cW(p)})^{PSL(2,\CC)}$ given by induction. We remark that essentially the same proof shows that if $T^\vee\subseteq PSL(2,\CC)$ is the one-dimensional torus, then the $T^\vee$-equivariantization of $\cC_{\cW(p)}$ is braided tensor equivalent to the category $\cC_{\cM(p)}^0$ of modules for the singlet vertex operator algebra $\cM(p)$ that was studied in \cite{CMY2}. Such a relationship had also been conjectured in \cite[Conjecture 11.6]{Ne}.

We plan to explore the tensor structure of $\cO_c$ for other central charges in future work. The remaining unsolved cases are the universal Virasoro vertex operator algebra at central charge $c_{p,q}$ and the simple Virasoro vertex operator algebra at central charge $c_t=13-6t-6t^{-1}$ for $t=-\frac{p}{q}$ a negative rational number. For $c_{p,q}$, the universal Virasoro vertex operator algebra is neither simple nor self-contragredient and thus the braided tensor category $\cO_{c_{p,q}}$ will be poorly behaved. For example, it will not be rigid because tensor products of non-zero modules in $\cO_{c_{p,q}}$ can be zero. However, we expect $\cO_{c_t}$ for $t=-\frac{p}{q}$ to be rigid and quite interesting, and we expect $V_{c_t}$ to admit large conformal vertex algebra extensions analogous to the triplet $W$-algebras. These categories $\cO_{c_t}$ will be subjects of forthcoming papers.

\medskip

\noindent {\bf Acknowledgments.} We thank Thomas Creutzig for many useful discussions, and we thank the referee for comments and suggestions. JY also thanks Florencia Orosz Hunziker for discussions on the Virasoro algebra.

\section{Preliminaries}

In this section we collect some results on the representation theory of the Virasoro Lie algebra, and on intertwining operators among modules for a vertex operator algebra.

\subsection{The Virasoro algebra}

Let $\cV ir$ denote the Virasoro Lie algebra with basis $\lbrace L_n\,\vert\,n\in\ZZ\rbrace\cup\lbrace\mathbf{c}\rbrace$ with $\mathbf{c}$ central and commutation relations
\begin{equation*}
 [L_m,L_n]=(m-n)L_{m+n}+\frac{m^3-m}{12}\delta_{m+n,0}\mathbf{c}.
\end{equation*}
We will sometimes use the decomposition $\cV ir=\cV ir_-\oplus\cV ir_{\geq 0}$, where
\begin{equation*}
 \cV ir_-=\mathrm{span}\lbrace L_n\,\vert\, n<0\rbrace,\qquad\cV ir_{\geq 0} =\mathrm{span}\lbrace L_n,\mathbf{c}\,\vert\,n\geq 0\rbrace.
\end{equation*}
For any vector space $\cU$ on which $L_0$ and $\mathbf{c}$ act by commuting operators, $\cU$ extends to a $\cV ir_{\geq 0}$-module on which $L_n$ acts by zero for $n>0$, and then we can form the induced module $\ind_{\cV ir_{\geq 0}}^{\cV ir} \cU$. In particular, for any central charge $c\in\CC$ and conformal dimension $h\in\CC$, the one-dimensional $\cV ir_{\geq 0}$-module $\CC_{c,h}$ on which $\mathbf{c}$ acts by $c$ and $L_0$ acts by $h$ induces to the Verma module $V(c,h)=\ind_{\cV ir_{\geq 0}}^{\cV ir} \CC_{c,h}$. Every Verma module $V(c,h)$ has a unique irreducible quotient $L(c,h)$.

For a central charge $c\in\CC$, we define $V_c$ to be the quotient of the Verma module $V(c,0)$ (induced from $\CC_{c,0}=\CC\vac$) by the submodule generated by the singular vector $L_{-1}\vac$. By \cite{FZ1}, $V_c$ is a vertex operator algebra in the sense of \cite{LL}. Moreover, every $\cV ir$-module $\cW$ that is suitably graded by generalized $L_0$-eigenvalues is a grading-restricted generalized $V_c$-module. Specifically, we require a grading $\cW=\bigoplus_{h\in\CC} \cW_{[h]}$ such that:
\begin{enumerate}
 \item $\cW_{[h]}$ is the generalized $L_0$-eigenspace with generalized eigenvalue $h$,
 \item $\dim\cW_{[h]}<\infty$ for all $h\in\CC$, and
 \item For any $h\in\CC$, $\cW_{[h+n]}=0$ for $n\in\ZZ$ sufficiently negative.
\end{enumerate}
The irreducible modules $L(c,h)$ for $h\in\CC$ comprise all irreducible $V_c$-modules. We are interested, however, in the category $\cO_c$ of $C_1$-cofinite grading-restricted generalized $V_c$-modules: by \cite{CJORY} this is the category of finite-length $\cV ir$-modules at central charge $c$ whose composition factors are irreducible quotients of reducible Verma modules. (In particular, irreducible Verma modules are not $C_1$-cofinite.)

Writing the central charge as $c=13-6t-6t^{-1}$ for some $t\in\CC\setminus\lbrace 0\rbrace$, the Feigin-Fuchs criterion for the existence of singular vectors in Verma modules \cite{FF} implies that $\cO_c$ contains all irreducible modules $\cL_{r,s}=L(c,h_{r,s})$ for $r,s\in\ZZ_+$, where
\begin{equation*}
 h_{r,s}:=\frac{r^2-1}{4} t-\frac{rs-1}{2}+\frac{s^2-1}{4} t^{-1} = \frac{(tr-s)^2}{4t}-\frac{(t-1)^2}{4t}.
\end{equation*}
Moreover, every irreducible module in $\cO_c$ is isomorphic to $L(c,h_{r,s})$ for some $r,s\in\ZZ$ (see \cite[Section 5.3]{IK} for a full description of the irreducible modules in $\cO_c$ for general central charges). For any $r,s\in\ZZ$, we use $\cV_{r,s}$ to denote the Verma module $V(c,h_{r,s})$.

It was established in \cite{CJORY} that for any central charge $c$, the category $\cO_c$ of $V_c$-modules admits the vertex algebraic braided tensor category structure of \cite{HLZ1}-\cite{HLZ8}. In this work, we are mainly concerned with central charges $c_{p,1}=13-6p-6p^{-1}$ for integers $p > 1$. At these central charges, we can use the conformal weight symmetries $h_{r,s+p}=h_{r-1,s}$ and $h_{r,s}=h_{-r,-s}$ for $r,s\in\ZZ$ to show that any irreducible module in $\cO_{c_{p,1}}$ is isomorphic to a unique $\cL_{r,s}$ with $r\geq 1$ and $1\leq s\leq p$. Then we have the following embedding diagrams involving the Verma modules $\cV_{r,s}$ (see for example \cite[Section 5.3]{IK}):
\begin{enumerate}
 \item When $1\leq s\leq p-1$, we have the diagram
  \begin{equation*}
   \cV_{1,s} \longleftarrow \cV_{2,p-s} \longleftarrow \cV_{3,s} \longleftarrow \cV_{4,p-s} \longleftarrow \cdots
  \end{equation*}
  In particular, the maximal proper submodule of $\cV_{r,s}$ is $\cV_{r+1,p-s}$ when $r\geq 1$ and $1\leq s\leq p-1$.

  \item When $s=p$, we have the diagram
  \begin{equation*}
   \cV_{i,p} \longleftarrow \cV_{i+2,p} \longleftarrow \cV_{i+4,p} \longleftarrow \cV_{i+6,p} \longleftarrow \cdots
  \end{equation*}
for $i=1,2$. In particular, the maximal proper submodule of $\cV_{r,p}$ is $\cV_{r+2,p}$ when $r\geq 1$.
\end{enumerate}
Note that the maximal proper submodule of $\cV_{1,1}$ is a Verma module generated by a singular vector of degree $1$, so $V_c\cong\cL_{1,1}$ as a $V_c$-module at the central charges we are considering. In particular, $V_c$ is a simple (and self-contragredient) vertex operator algebra.

In addition to Verma modules, we will sometimes need to work with their contragredients $\cV_{r,s}'$. Since irreducible Virasoro modules are self-contragredient, the surjections $\cV_{r,s}\rightarrow\cL_{r,s}$ dualize to injections $\cL_{r,s}\rightarrow\cV_{r,s}'$. In particular, $\cL_{r,s}$ is the $V_c$-submodule of $\cV_{r,s}'$ generated by the lowest conformal weight space.

\subsection{Intertwining operators among modules for a vertex operator algebra}

We recall the definition of (logarithmic) intertwining operator among a triple of modules for a vertex operator algebra $V$ from \cite{HLZ2}:
\begin{defi}
 Suppose $W_1$, $W_2$, and $W_3$ are grading-restricted generalized $V$-modules. An \textit{intertwining operator} of type $\binom{W_3}{W_1\,W_2}$ is a linear map
 \begin{align*}
  \cY: W_1\otimes W_2 & \rightarrow W_3[\log x]\lbrace x\rbrace\nonumber\\
   w_1\otimes w_2 & \mapsto \cY(w_1,x)w_2=\sum_{h\in\CC}\sum_{k\in\NN} (w_1)_{h,k} w_2\,x^{-h-1}(\log x)^k
 \end{align*}
which satisfies the following properties:
\begin{enumerate}
 \item \textit{Lower truncation}: For any $w_1\in W_1$, $w_2\in W_2$, and $h\in\CC$, $(w_1)_{h+n,k} w_2 =0$ for $n\in\ZZ$ sufficiently large, independently of $k$.

 \item The \textit{Jacobi identity}: For $v\in V$ and $w_1\in W_1$,
 \begin{align*}
  x_0^{-1}\delta\left(\frac{x_1-x_2}{x_0}\right) Y_{W_3}(v,x_1)\cY(w_1,x_2) & - x_0^{-1}\left(\frac{-x_2+x_1}{x_0}\right)\cY(w_1,x_2)Y_{W_2}(v,x_1)\nonumber\\
  & = x_1^{-1}\delta\left(\frac{x_2+x_0}{x_1}\right)\cY(Y_{W_1}(v,x_0)w_1,x_2).
 \end{align*}

 \item The \textit{$L_{-1}$-derivative property}: For $w_1\in W_1$,
 \begin{equation*}
  \cY(L_{-1} w_1,x)=\dfrac{d}{dx}\cY(w_1,x).
 \end{equation*}
\end{enumerate}
\end{defi}

We will need two consequences of the Jacobi identity. Extracting the coefficient of $x_0^{-1} x_1^{-n-1}$ in the Jacobi identity yields the \textit{commutator formula}
\begin{equation}\label{eqn:gen_comm_form}
 v_n\cY(w_1,x) = \cY(w_1,x)v_n+\sum_{i\geq 0} \binom{n}{i} x^{n-i}\cY(v_i w_1,x);
\end{equation}
in the special case that $v$ is the conformal vector $\omega$, this means
\begin{equation}\label{eqn:Vir_comm_form}
 L_n\cY(w_1,x) =\cY(w_1,x)L_n+\sum_{i\geq 0}\binom{n+1}{i} x^{n+1-i}\cY(L_{i-1} w_1,x).
\end{equation}
Similarly, extracting the coefficient of $x_0^{-n-1} x_1^{-1}$ yields the \textit{iterate formula}
\begin{align}\label{eqn:gen_it_form}
 \cY(v_n w_1,x) =\sum_{i\geq 0} (-1)^i\binom{n}{i}\left(v_{n-i}\, x^i\cY(w_1,x) -(-1)^n x^{n-i}\cY(w_1,x)v_i\right);
\end{align}
in the special case $v=\omega$ we have
\begin{align}\label{eqn:Vir_it_form}
\cY(L_n w_1,x) =\sum_{i\geq 0} (-1)^i\binom{n+1}{i}\left( L_{n-i}\,x^i\cY(w_1,x)+(-1)^{n} x^{n+1-i}\cY(w_1,x)L_{i-1}\right).
\end{align}

For grading-restricted generalized $V$-modules $W_1$, $W_2$, $W_3$,  we say that an intertwining operator $\cY$ of type $\binom{W_3}{W_1\,W_2}$ is \textit{surjective} if
\begin{equation*}
W_3=\mathrm{span}\lbrace (w_1)_{h,k} w_2\,\vert\,w_1\in W_1, w_2\in W_2, h\in\CC, k\in\NN\rbrace.
\end{equation*}
Actually, we can reduce the spanning set for the image of an intertwining operator somewhat:
\begin{lem}\label{lem:intw_op_surjectivity}
 Let $W_1$, $W_2$, and $W_3$ be grading-restricted generalized $V$-modules. An intertwining operator $\cY$ of type $\binom{W_3}{W_1\,W_2}$ is surjective if and only if
 \begin{equation*}
  W_3 =\mathrm{span}\lbrace (w_1)_{h,0} w_2\,\vert\,w_1\in W_1, w_2\in W_2, h\in\CC\rbrace.
 \end{equation*}
\end{lem}
\begin{proof}
We just need to show that all $(w_1)_{h,k} w_2$ for $k\in\NN$ are contained in the span of the vectors $(w_1)_{h,0} w_2$ for $w_1\in W_1$, $w_2\in W_2$, and $h\in\CC$. Using the $L_{-1}$-derivative property,
 \begin{align*}
  \cY(L_{-1}w_1,x)w_2 & = \frac{d}{dx}\sum_{h\in\CC}\sum_{k\in\NN} (w_1)_{h,k} w_2\,x^{-h-1}(\log x)^k \nonumber\\
  & =\sum_{h\in\CC}\sum_{k\in\NN} (w_1)_{h,k} w_2\,x^{-h-2}\left(k(\log x)^{k-1}-(h+1)(\log x)^k\right).
 \end{align*}
From this we see that
\begin{equation*}
 (w_1)_{h,k+1} w_2 =\frac{1}{k+1}\left((h+1)(w_1)_{h,k} w_2+(L_{-1}w_1)_{h+1,k} w_2\right),
\end{equation*}
so that
\begin{equation*}
 (w_1)_{h,k} w_2\in\mathrm{span}\lbrace (w_1)_{h,0} w_2\,\vert\,w_1\in W_1\,w_2\in W_2, h\in\CC\rbrace
\end{equation*}
for all $k\in\NN$ follows by induction on $k$.
\end{proof}

Associated to any intertwining operator $\cY$ of type $\binom{W_3}{W_1\,W_2}$, we have an \textit{intertwining map}
\begin{equation*}
 I: W_1\otimes W_2\rightarrow\overline{W}_3 =\prod_{h\in\CC} (W_3)_{[h]}
\end{equation*}
defined by
\begin{equation*}
 I(w_1\otimes w_2) =\cY(w_1,1)w_2
\end{equation*}
for $w_1\in W_1$, $w_2\in W_2$, where we realize the substitution $x\mapsto 1$ using the real-valued branch of logarithm $\ln 1=0$. In particular, for generalized $L_0$-eigenvectors $w_1\in W_1$ and $w_2\in W_2$, the coefficients $(w_1)_{h,0} w_2$ are simply the projections of $I(w_1\otimes w_2)$ to the conformal weight spaces of $W_3$. Thus we get the following corollary of Lemma \ref{lem:intw_op_surjectivity}:
\begin{cor}\label{cor:intw_op_surjectivity}
 Let $W_1$, $W_2$, and $W_3$ be grading-restricted generalized $V$-modules. An intertwining operator $\cY$ of type $\binom{W_3}{W_1\,W_2}$ is surjective if and only if $W_3$ is spanned by projections of vectors $\cY(w_1,1)w_2$ for $w_1\in W_1$, $w_2\in W_2$ to the conformal weight spaces of $W_3$.
\end{cor}

In \cite{HLZ3}, tensor products of $V$-modules are defined in terms of intertwining maps; they can be defined equivalently in terms of intertwining operators:
\begin{defi}
 Let $\cC$ be a category of grading-restricted generalized $V$-modules containing $W_1$ and $W_2$. A \textit{tensor product} of $W_1$ and $W_2$ in $\cC$ is a pair $(W_1\tens W_2,\cY_\tens)$, with $W_1\tens W_2$ a module in $\cC$ and $\cY_\tens$ an intertwining operator of type $\binom{W_1\tens W_2}{W_1\,W_2}$, which satisfies the following universal property: For any module $W_3$ in $\cC$ and intertwining operator $\cY$ of type $\binom{W_3}{W_1\,W_2}$, there is a unique $V$-module homomorphism $f: W_1\tens W_2\rightarrow W_3$ such that $\cY=f\circ\cY_\tens$.
\end{defi}

If the tensor product $(W_1\tens W_2, \cY_\tens)$ exists, then the tensor product intertwining operator $\cY_\tens$ is surjective \cite[Proposition 4.23]{HLZ3}. In \cite{HLZ1}-\cite{HLZ8}, it was shown under suitable conditions, such as closure under tensor products, that $V$-module categories $\cC$ have braided tensor category structure. In \cite{CJORY}, it was shown that these conditions are satisfied by the category $\cO_c$ of $C_1$-cofinite grading-restricted generalized modules for the Virasoro vertex operator algebra $V_c$ at any central charge $c$. For a detailed description of the braided tensor category structure on categories such as $\cO_c$, in particular a description of the left and right unit isomorphisms $l$ and $r$, the associativity isomorphisms $\cA$, and the braiding isomorphisms $\cR$, see \cite{HLZ8} or the exposition in \cite[Section 3.3]{CKM}.

\subsection{Zhu algebra construction of intertwining operators}

Let $V$ be a vertex operator algebra with grading-restricted generalized modules $W_1$, $W_2$, and $W_3$. The \textit{fusion rule} $\mathcal{N}^{W_3}_{W_1, W_2}$ is the dimension of the space of intertwining operators of type $\binom{W_3}{W_1\,W_2}$. Here, we recall some general results on constructing intertwining operators and determining fusion rules using the Zhu algebra approach developed in \cite{FZ1, Li, FZ2, HY}, among other references.

To start, consider a grading-restricted generalized $V$-module $W=\bigoplus_{h\in\CC} W_{[h]}$. If we take $I$ to be the set of cosets in $\CC/\ZZ$ such that for $i\in I$, $W_{[h]}\neq 0$ for some $h\in i$, then
\begin{equation}\label{eqn:W_decomp}
 W=\bigoplus_{i\in I} \bigoplus_{n=0}^\infty W_{[h_i+n]}
\end{equation}
with $h_i$ the minimal conformal weight occurring in the coset $i$. Each $W_i=\bigoplus_{n=0}^\infty W_{[h_i+n]}$ is a $V$-submodule of $W$, so that $\vert I\vert =1$ if $W$ is non-zero and indecomposable, and $\vert I\vert$ is finite if $W$ is finitely generated.

The decomposition \eqref{eqn:W_decomp} implies that $W$ has an $\NN$-grading $W=\bigoplus_{n=0}^\infty W(n)$, given by
\begin{equation*}
 W(n)=\bigoplus_{i\in I} W_{[h_i+n]},
\end{equation*}
such that
\begin{equation}\label{eqn:N-grading_cond}
 v_m\cdot W(n)\subseteq W(\deg v+n-m-1)
\end{equation}
for $v\in V$, $m\in\ZZ$, and $n\in\NN$. Although this need not be the unique $\NN$-grading such that \eqref{eqn:N-grading_cond} holds, we shall always use this particular $\NN$-grading for grading-restricted generalized $V$-modules unless specified otherwise. If $W$ is finitely generated, so that $\vert I\vert<\infty$, then each $W(n)$ is the direct sum of finitely many generalized $L_0$-eigenspaces. In this case, we have well-defined projection maps
\begin{equation*}
 \pi_n: \overline{W}=\prod_{h\in\CC} W_{[h]}\rightarrow W(n)
\end{equation*}
for each $n\in\NN$.

Now suppose $W_1$, $W_2$, and $W_3$ are three grading-restricted generalized $V$-modules such that $W_3$ is finitely generated (to guarantee that the projection map $\pi_0: \overline{W}_3\rightarrow W_3(0)$ exists). Let $A(V)$ denote the Zhu algebra of $V$ defined in \cite{Zh} and let $A(W_1)$ denote the $A(V)$-bimodule defined in \cite{FZ1}. The degree-$0$ subspaces $W_2(0)$ and $W_3(0)$ are left $A(V)$-modules \cite{Zh}. Now for any intertwining operator $\cY$ of type $\binom{W_3}{W_1\,W_2}$, the following $A(V)$-module map was first constructed in \cite{FZ1}:
\begin{align*}
 \pi(\cY): A(W_1) \otimes_{A(V)} W_2(0) &\longrightarrow W_3(0)\nonumber\\
  [w_1]\otimes u_2 & \longmapsto \pi_0\left(\cY(w_1,1)u_2\right),
\end{align*}
where $[w_1]$ is the image of $w_1\in W_1$ in $A(W_1)$ and $\cY(\cdot,1)\cdot$ is the intertwining map associated to $\cY$. The next proposition is essentially a version of \cite[Proposition 24]{TW}, where the result is attributed to Nahm \cite{Na}:
\begin{prop}\label{prop:piY_surjective}
 Assume that $W_1$, $W_2$, and $W_3$ are grading-restricted generalized $V$-modules such that $W_2$ is generated by $W_2(0)$ as a $V$-module and $W_3$ is finitely generated. If $\cY$ is a surjective intertwining operator, then $\pi(\cY)$ is surjective.
\end{prop}

\begin{proof}
Since $\cY$ is surjective, Corollary \ref{cor:intw_op_surjectivity} says that $W_3(0)$ is spanned by $\pi_0\left(\cY(w_1,1)w_2\right)$ for $w_1\in W_1$ and $w_2\in W_2$. Thus we need to show that
 \begin{equation*}
  \pi_0\left(\cY(w_1,1)w_2\right)\in \im\pi(\cY)
 \end{equation*}
for any $w_1\in W_1$, $w_2\in W_2$. This holds by definition for $w_2\in W_2(0)$. For $w_2\in\bigoplus_{n\geq 1} W_2(n)$, we note that because $W_2(0)$ generates $W_2$ as a $V$-module, $w_2$ is a linear combination of  vectors $v_n u_2$ for $u_2\in W_2(0)$, homogeneous $v\in V$, and $n\in\ZZ$ such that $\deg v-n-1> 0$ (see \cite[Proposition 4.5.6]{LL}). The commutator formula \eqref{eqn:gen_comm_form} then implies that for any $w_1\in W_1$,
\begin{align*}
 \pi_0\left(\cY(w_1,1)v_n u_2\right) &=\pi_0\bigg(v_n\cY(w_1,1)u_2-\sum_{i\geq 0}\binom{n}{i} \cY(v_i w_1,1)u_2\bigg)\\
 & = -\sum_{i\geq 0}\binom{n}{i}\pi_0(\cY(v_i w_1,1)u_2)\in \im\pi(\cY)
\end{align*}
since $\deg v_n>0$. This proves the proposition.
\end{proof}

Note that $\cY\mapsto\pi(\cY)$ defines a linear map from intertwining operators of type $\binom{W_3}{W_1\,W_2}$ to $\hom_{A(V)}(A(W_1)\otimes_{A(V)} W_2(0), W_3(0))$. The main theorem of \cite{Li} (generalized to logarithmic intertwining operators in \cite{HY}) is that this linear map is an isomorphism under suitable conditions on $W_1$, $W_2$, and $W_3$. For simplicity, we will describe these conditions only when $V$ is a Virasoro vertex operator algebra $V_c$, in which case we have an isomorphism $A(V_c)\cong\CC[x]$ given by $[\omega]\mapsto x$ \cite{FZ1}.

Any $\CC[x]$-module $\cU$ is equivalently an $A(V_c)$-module, which is equivalently a $\cV ir_{\geq 0}$-module on which $L_0$ acts by $x$ and $L_n$ acts by $0$ for $n>0$. We then have the induced generalized Verma module $\cV=\ind^{\cV ir}_{\cV ir_{\geq 0}} \cU$. If $\cU$ is finite dimensional, then we have an $A(V_c)\cong\CC[x]$-module isomorphism $\cU\cong\cU^*$, so that the lowest conformal weight space $\cV'(0)$ of the generalized Verma module  contragredient  is isomorphic to $\cU$. Now the following theorem is the main result of \cite{Li, HY} for Virasoro vertex operator algebras (see also \cite[Lemma 2.19]{FZ2}):
\begin{thm}\label{thm:Zhu_fus_rules}
 Suppose $\cW_1$ is a grading-restricted generalized $V_c$-module generated by $\cW_1(0)$ and $\cU_2$, $\cU_3$ are finite-dimensional $A(V_c)$-modules. Then $\cY\mapsto\pi(\cY)$ defines a linear isomorphism from intertwining operators of type $\binom{\cV_3'}{\cW_1\,\cV_2}$ to $\hom_{A(V_c)}(A(\cW_1)\otimes_{A(V_c)} \cU_2,\cU_3)$, where $\cV_i=\ind^{\cV ir}_{\cV ir_{\geq 0}}\cU_i$ for $i=2,3$. In particular, fusion rules satisfy
 \begin{equation*}
  \mathcal{N}_{\cW_1,\cV_2}^{\cV_3'} =\dim \hom_{A(V_c)}(A(\cW_1)\otimes_{A(V_c)} \cU_2,\cU_3).
 \end{equation*}
\end{thm}

\begin{rem}\label{rem:non_std_N_grad}
 In the preceding theorem, we need to define $\pi(\cY)$ using the $\NN$-grading on $\cV_3'$ such that $\cV_3'(0)=\cU_3^*\cong\cU_3$. This $\NN$-grading will differ slightly from our usual $\NN$-grading convention if $L_0$ has two eigenvalues on $\cU_3$ that differ by a non-zero integer.
\end{rem}

\section{First results on Virasoro fusion}\label{sec:first_fus}

In this section, our goal is to use Proposition \ref{prop:piY_surjective} to obtain upper bounds on tensor products of certain $V_c$-modules in $\cO_c$, and to use Theorem \ref{thm:Zhu_fus_rules} to obtain lower bounds. At first, we consider arbitrary central charges, and then we specialize to the central charge $c_{p,1}$.

\subsection{Results at general central charge}

In this subsection, we assume $c=13-6t-6t^{-1}$ for any $t\in\CC\setminus\lbrace 0\rbrace$. We want to see what Proposition \ref{prop:piY_surjective}, applied to the surjective tensor product intertwining operator $\cY_\tens$, says about the tensor products $\cL_{1,2}\tens\cL_{r,s}$ and $\cL_{2,1}\tens\cL_{r,s}$ for $r,s\in\ZZ_+$. Thus we must first determine the $A(V_c)$-bimodules $A(\cL_{1,2})$ and $A(\cL_{2,1})$. This was done in \cite[Lemmas 2.10 and 2.11]{FZ2} under the assumption $t\notin\QQ$. Here, we review the calculations to confirm that the same results hold for general $t$ (except that when $t=\frac{p}{q}$ for relatively prime $p,q\in\ZZ_{\geq 2}$, the calculations actually compute Zhu bimodules for certain non-simple Verma module quotients).

In this and the following sections, we use $v_{r,s}$ to denote a lowest-conformal-weight vector generating either $\cV_{r,s}$ or one of its quotients, such as $\cL_{r,s}$. We now compute $A(\cL_{1,2})$, noting that $A(\cL_{2,1})$ can be determined almost identically with the substitutions $v_{1,2}\mapsto v_{2,1}$, $h_{1,2}\mapsto h_{2,1}$, and $t^{-1}\mapsto t$. To begin, the isomorphism $A(V_c)\cong\CC[x]$ corresponds to an isomorphism
\begin{align*}
 \CC[x,y] & \rightarrow A(\cV_{1,2})\nonumber\\
 x^m y^n & \mapsto [\omega]^m\cdot[v_{1,2}]\cdot[\omega]^n,
\end{align*}
where the left and right actions of $\CC[x]$ on the bimodule $\CC[x,y]$ are multiplication by $x$ and $y$, respectively, while the left and right actions of $A(V_c)$ on $A(\cV_{1,2})$ are given by
\begin{equation*}
[\omega]\cdot[v] =[(L_0+2L_{-1}+L_{-2})v],\qquad[v]\cdot[\omega]=[(L_{-2}+L_{-1})v]
\end{equation*}
for $v\in\cV_{1,2}$. Under this isomorphism, we can then identify
$$A(\cL_{1,2})\cong\CC[x,y]/(f_{1,2}(x,y)),$$
where $f_{1,2}(x,y)$ is the polynomial corresponding to the singular vector $(L_{-1}^2-\frac{1}{t}L_{-2})v_{1,2}\in\cV_{1,2}$ generating the maximal proper submodule of $\cV_{1,2}$.

To determine $f_{1,2}(x,y)$, we first note that for $v\in\cV_{1,2}$,
\begin{equation}\label{eqn:bimod_reln_2}
 [L_{-2} v]=[v]\cdot[\omega]-[L_{-1}v].
\end{equation}
This together with
\begin{align*}
 [\omega]\cdot[v] =(\mathrm{wt}\,v)[v]+2[L_{-1} v]+[L_{-2} v]
\end{align*}
implies
\begin{equation}\label{eqn:bimod_reln}
 [L_{-1}v]=[\omega]\cdot[v]-[v]\cdot[\omega]-(\mathrm{wt}\,v)[v].
\end{equation}
Consequently,
\begin{align*}
 \bigg[\bigg(L_{-1}^2- & \frac{1}{t}L_{-2}\bigg)v_{1,2}\bigg]  =[\omega]\cdot[L_{-1}v_{1,2}]-[L_{-1}v_{1,2}]\cdot[\omega]-(h_{1,2}+1)[L_{-1} v_{1,2}]\nonumber\\
 &\qquad\qquad\qquad\qquad-\frac{1}{t}([v_{1,2}]\cdot[\omega]-[L_{-1}v_{1,2}])\nonumber\\
 & =[\omega]\cdot\big([\omega]\cdot[v_{1,2}]-[v_{1,2}]\cdot[\omega]-h_{1,2}[v_{1,2}]\big) -\big([\omega]\cdot[v_{1,2}]-[v_{1,2}]\cdot[\omega]-h_{1,2}[v_{1,2}]\big)\cdot[\omega]\nonumber\\
 &\qquad\qquad-(h_{1,2}+1)\big([\omega]\cdot[v_{1,2}]-[v_{1,2}]\cdot[\omega]-h_{1,2}[v_{1,2}]\big)\nonumber\\
 & \qquad\qquad-\frac{1}{t}[v_{1,2}]\cdot[\omega]+\frac{1}{t}\big([\omega]\cdot[v_{1,2}]-[v_{1,2}]\cdot[\omega]-h_{1,2}[v_{1,2}]\big)\nonumber\\
 & =[\omega]^2\cdot[v_{1,2}]-2[\omega]\cdot[v_{1,2}]\cdot[\omega]+[v_{1,2}]\cdot[\omega]^2-\left(2h_{1,2}+1-\frac{1}{t}\right)[\omega]\cdot[v_{1,2}]\nonumber\\
 & \qquad\qquad + \left(2h_{1,2}+1-\frac{2}{t}\right)[v_{1,2}]\cdot[\omega]+h_{1,2}\left(h_{1,2}+1-\frac{1}{t}\right)[v_{1,2}].
\end{align*}
This corresponds to the polynomial
\begin{align*}
 f_{1,2}(x,y) & =x^2-2xy+y^2-\left(2h_{1,2}+1-\frac{1}{t}\right)x+\left(2h_{1,2}+1-\frac{2}{t}\right)y+h_{1,2}\left(h_{1,2}+1-\frac{1}{t}\right)\nonumber\\
 & =\left(x-y-\left(h_{1,2}+1-\frac{1}{t}\right)\right)\left(x-y-h_{1,2}\right)-\frac{1}{t}y.
\end{align*}
We have now determined $A(\cL_{1,2})$; similarly, we can use the singular vector $(L_{-1}^2-t\,L_{-2})v_{2,1}\in\cV_{2,1}$ to show that
\begin{equation*}
 A(\cL_{2,1})\cong\CC[x,y]/(f_{2,1}(x,y))
\end{equation*}
where
\begin{equation*}
 f_{2,1}(x,y)=\left(x-y-\left(h_{2,1}+1-t\right)\right)(x-y-h_{2,1})-t\,y.
\end{equation*}

Now it is easy to determine the $A(V_c)$-modules $\cM_{r,s}=A(\cL_{1,2})\otimes_{A(V_c)}\CC v_{r,s}$ and $\mathcal{N}_{r,s}=A(\cL_{2,1})\otimes_{A(V_c)}\CC v_{r,s}$ for $r,s\in\ZZ_+$, where $\CC v_{r,s}$ is both $\cV_{r,s}(0)$ and $\cL_{r,s}(0)$. We have
\begin{align*}
 \cM_{r,s} & \cong \CC[x]/(f_{1,2}(x,h_{r,s})),\\
 \mathcal{N}_{r,s} & \cong \CC[x]/(f_{2,1}(x,h_{r,s})),
\end{align*}
where
\begin{align*}
 f_{1,2}(x,h_{r,s}) &=\left(x-\left(h_{1,2}+h_{r,s}+1-\frac{1}{t}\right)\right)\left(x-(h_{1,2}+h_{r,s})\right)-\frac{h_{r,s}}{t}\\
 &= (x-h_{r,s-1})(x-h_{r,s+1}),\nonumber\\
 f_{2,1}(x,h_{r,s}) &=\left(x-\left(h_{2,1}+h_{r,s}+1-t\right)\right)\left(x-(h_{2,1}+h_{r,s})\right)-t\,h_{r,s}\\
 &= (x-h_{r-1,s})(x-h_{r+1,s}).
\end{align*}
In other words, $L_0$ has eigenvalue(s) $h_{r,s\pm 1}$ on $\cM_{r,s}$ and eigenvalue(s) $h_{r\pm 1,s}$ on $\mathcal{N}_{r,s}$.

We can now apply Proposition \ref{prop:piY_surjective}:
\begin{prop}\label{prop:conf_wts}
 Let $r,s\in\ZZ_+$ and let $\cW$ be a grading-restricted generalized $V_c$-module in $\cO_c$.
 \begin{enumerate}
  \item If there is a surjective intertwining operator of type $\binom{\cW}{\cL_{1,2}\,\cL_{r,s}}$, then the conformal weights of $\cW$ are contained in $\lbrace h_{r,s-1}+\NN\rbrace\cup\lbrace h_{r,s+1}+\NN\rbrace$. In particular, this conclusion holds for $\cW=\cL_{1,2}\tens\cL_{r,s}$.

  \item If there is a surjective intertwining operator of type $\binom{\cW}{\cL_{2,1}\,\cL_{r,s}}$, then the conformal weights of $\cW$ are contained in $\lbrace h_{r-1,s}+\NN\rbrace\cup\lbrace h_{r+1,s}+\NN\rbrace$. In particular, this conclusion holds for $\cW=\cL_{2,1}\tens\cL_{r,s}$.
 \end{enumerate}
\end{prop}
\begin{proof}
 First note that $\cL_{r,s}$ is generated by $\cL_{r,s}(0)$ and that $\cW$, as a $C_1$-cofinite module in $\cO_c$, is finitely generated. So in the first case, Proposition \ref{prop:piY_surjective} says that $\cW(0)$ is a homomorphic image of $\cM_{r,s}$ as an $A(V_c)$-module. Thus the generalized $L_0$-eigenvalue(s) on $\cW(0)$ are $h_{r,s\pm1}$, and then the first conclusion of the proposition follows from our $\NN$-grading convention. The proof of the second part of the proposition is the same.
\end{proof}

From now on, we will mainly focus on the tensor products $\cL_{1,2}\tens\cL_{r,s}$. We have shown that there is a surjective $A(V_c)$-homomorphism $\pi(\cY_\tens): \cM_{r,s}\rightarrow(\cL_{1,2}\tens\cL_{r,s})(0)$. We may also regard $\pi(\cY_\tens)$ as a $\cV ir_{\geq 0}$-homomorphism, so  if we set $\cW_{r,s}=\ind_{\cV ir_{\geq 0}}^{\cV ir}\cM_{r,s}$, then the universal property of induced modules leads to a $V_c$-module homomorphism
\begin{equation*}
 \Pi_{r,s}: \cW_{r,s}\rightarrow\cL_{1,2}\tens\cL_{r,s}
\end{equation*}
such that $(\cL_{1,2}\tens\cL_{r,s})(0)\subseteq\im\Pi_{r,s}$. We show that $\Pi_{r,s}$ is usually surjective:
\begin{prop}\label{prop:Pi_rs_surjective}
If $h_{r,s-1}-h_{r,s+1}\notin\ZZ\setminus\lbrace 0\rbrace$, then the homomorphism $\Pi_{r,s}$ is surjective.
\end{prop}
\begin{proof}
Set $\cW=(\cL_{1,2}\tens\cL_{r,s})/\im\Pi_{r,s}$ and let $\pi:\cL_{1,2}\tens\cL_{r,s}\rightarrow\cW$ denote the canonical quotient map. The grading-restricted generalized module $\cW$ is in $\cO_c$, and $\pi\circ\cY_\tens$ is a surjective intertwining operator of type $\binom{\cW}{\cL_{1,2}\,\cL_{r,s}}$. Thus from Propositions \ref{prop:piY_surjective} and \ref{prop:conf_wts}(1),
\begin{equation*}
 \cW(0)\subseteq\cW_{[h_{r,s-1}]}+\cW_{[h_{r,s+1}]}.
\end{equation*}
The two sets $\lbrace h_{r,s-1}+\NN\rbrace$ and $\lbrace h_{r,s+1}+\NN\rbrace$ of potential conformal weights of $\cW$ are either disjoint (if $h_{r,s-1}-h_{r,s+1}\notin\ZZ$) or identical (if $h_{r,s-1}=h_{r,s+1}$). Thus
\begin{equation*}
 (\cL_{1,2}\tens\cL_{r,s})_{[h_{r,s-1}]}+(\cL_{1,2}\tens\cL_{r,s})_{[h_{r,s+1}]} = (\cL_{1,2}\tens\cL_{r,s})(0) \subseteq\im\Pi_{r,s},
\end{equation*}
which means $\cW(0)=0$. By our $\NN$-grading convention, $\cW=0$ as well, that is, $\cL_{1,2}\tens\cL_{r,s} =\im\Pi_{r,s}$ and $\Pi_{r,s}$ is surjective.
\end{proof}

\begin{rem}
 The proof of the above proposition fails when, say, $h_{r,s-1}-h_{r,s+1}\in\ZZ_+$, because then it is possible that $(\cL_{1,2}\tens\cL_{r,s})(0)=(\cL_{1,2}\tens\cL_{r,s})_{[h_{r,s+1}]}$ and that
$(\cL_{1,2}\boxtimes\cL_{r,s})/\im\Pi_{r,s}$ has a non-zero space of conformal weight $h_{r,s-1}$.
\end{rem}

Note that if $h_{r,s-1}\neq h_{r,s+1}$, then $\cW_{r,s}\cong\cV_{r,s-1}\oplus\cV_{r,s+1}$. In these cases, we can determine the images of $v_{r,s\pm1}\in\cV_{r,s\pm1}$ in $\cL_{1,2}\tens\cL_{r,s}$ under the homomorphism $\Pi_{r,s}$. In fact, we get the following by determining the $x$-eigenvectors in $\CC[x]/(f_{1,2}(x,h_{r,s})$ and using definitions:
\begin{equation*}
 \begin{array}{rll}
  \CC v_{r,s-1}\oplus\CC v_{r,s+1} & \rightarrow \CC[x]/(f_{1,2}(x,h_{r,s}) & \rightarrow A(\cL_{1,2})\otimes_{A(V_c)}\CC v_{r,s} \\
  v_{r,s\pm1} & \mapsto x-h_{r,s\mp 1}+(f_{1,2}(x,h_{r,s})) & \mapsto ([\omega]-h_{r,s\mp1})\cdot[v_{1,2}]\otimes_{A(V_c)} v_{r,s}
 \end{array}
\end{equation*}
Then \eqref{eqn:bimod_reln} implies
\begin{align*}
 ([\omega]-h_{r,s\mp1})\cdot[v_{1,2}]\otimes_{A(V_c)} v_{r,s} & =[v_{1,2}] \cdot([\omega]+h_{1,2}-h_{r,s\mp1})\otimes_{A(V_c)} v_{r,s}\nonumber\\
 &\hspace{5em}+[L_{-1}v_{1,2}]\otimes_{A(V_c)} v_{r,s} \nonumber\\
 & =(h_{1,2}+h_{r,s}-h_{r,s\mp1})[v_{1,2}]\otimes_{A(V_c)} v_{r,s}+[L_{-1} v_{1,2}]\otimes_{A(V_c)} v_{r,s},
\end{align*}
which $\pi(\cY_\tens)$ maps to
\begin{equation*}
 \left(-\frac{1\pm r}{2}+\frac{1\pm s}{2} t^{-1}\right)\pi_0(v_{1,2}\tens v_{r,s})+\pi_0(L_{-1}v_{1,2}\tens v_{r,s});
\end{equation*}
here $\tens$ denotes the tensor product intertwining map $\cY_\tens(\cdot,1)\cdot$. Rescaling these vectors a little, we may conclude:
\begin{prop}\label{prop:top_level_eigenvectors}
 For $r,s\in\ZZ_+$, the vectors
 \begin{equation*}
  \Pi_{r,s}(v_{r,s\pm1}) =\left(1\pm s-(1\pm r)t\right)\pi_0(v_{1,2}\tens v_{1,2})+2t\,\pi_0(L_{-1}v_{1,2}\tens v_{1,2})\in(\cL_{1,2}\tens\cL_{r,s})(0)
 \end{equation*}
are, if non-zero, $L_0$-eigenvectors with eigenvalues $h_{r,s\pm1}$.
\end{prop}

\subsection{Results at specialized central charge}

In this section, we now assume that $c=13-6p-6p^{-1}$ where $p > 1$ is an integer. In this case, irreducible modules in $\cO_c$ are given by $\cL_{r,s}$ for $r\geq 1$ and $1\leq s\leq p$, and conformal weights satisfy
\begin{equation*}
 h_{r,s-1}-h_{r,s+1} =r-\frac{s}{p}.
\end{equation*}
We see that $h_{r,s-1}=h_{r,s+1}$ only when $(r,s)=(1,p)$, so that the generalized Verma module $\cW_{r,s}=\ind^{\cV ir}_{\cV ir_{\geq 0}} \cM_{r,s}$ is given by
\begin{equation*}
 \cW_{r,s} \cong\left\lbrace\begin{array}{ccc}
                         \cV_{r,s-1}\oplus\cV_{r,s+1} & \text{if} & (r,s)\neq(1,p)\\
                         \cV^{(2)}_{1,p-1} & \text{if} & (r,s)=(1,p)
                        \end{array}
\right. ,
\end{equation*}
where $\cV^{(2)}_{1,p-1}$ is the generalized Verma module induced from the two-dimensional $\cV ir_{\geq 0}$-module on which $L_0$ acts by the matrix $\left[\begin{array}{cc}
                     h_{1,p-1} & 1 \\
                                     0 & h_{1,p-1}\\                                                                                                                        \end{array}\right]$. Moreover, Proposition \ref{prop:Pi_rs_surjective} yields:
\begin{cor}\label{cor:Pi_rs_surjective}
 The homomorphism $\Pi_{r,s}: \cW_{r,s}\rightarrow\cL_{1,2}\tens\cL_{r,s}$ is surjective when $1\leq s\leq p-1$ and when $(r,s)=(1,p)$.
\end{cor}

Corollary \ref{cor:Pi_rs_surjective} gives an upper bound for the tensor product $\cL_{1,2}\tens\cL_{r,s}$ when $1\leq s\leq p-1$ or when $(r,s)=(1,p)$: in the first case, $\cL_{1,2}\tens\cL_{r,s}$ is a quotient of $\cV_{r,s-1}\oplus\cV_{r,s+1}$, and in the second, $\cL_{1,2}\tens\cL_{1,p}$ is a quotient of $\cV_{1,p-1}^{(2)}$. We next use  Theorem \ref{thm:Zhu_fus_rules} to get lower bounds for these tensor products. We start by obtaining some non-zero intertwining operators:
\begin{prop}\label{prop:intwo_op_exist} \hspace{2em}
 \begin{enumerate}
 \item When $r\geq 1$ and $s=1$, there is a non-zero intertwining operator of type $\binom{\cV_{r,2}'}{\cL_{1,2}\,\cL_{r,1}}$.

  \item When $r\geq 1$ and $2\leq s\leq p-1$, or when $(r,s)=(1,p)$, there is an intertwining operator of type $\binom{\cW_{r,s}'}{\cL_{1,2}\,\cL_{r,s}}$ that contains $\cW_{r,s}'(0)\cong\cM_{r,s}$ in its image.

 \end{enumerate}
\end{prop}
\begin{proof}
Note that $\cL_{1,2}$ is generated by $\cL_{1,2}(0)$ and that $\cM_{r,s}$ is finite dimensional. Thus by Theorem \ref{thm:Zhu_fus_rules}, the identity on $\cM_{r,s}$ induces an intertwining operator $\cY$ of type $\binom{\cW_{r,s}'}{\cL_{1,2}\,\cV_{r,s}}$ such that $\pi(\cY)=\Id_{\cM_{r,s}}$. This intertwining operator will induce a non-zero quotient intertwining operator $\overline{\cY}$ of type $\binom{\cW_{r,s}'}{\cL_{1,2}\,\cL_{r,s}}$ if $\cY\vert_{\cL_{1,2}\otimes\cJ_{r,s}}=0$, where $\cJ_{r,s}$ is the maximal proper submodule of $\cV_{r,s}$. To show this, it is enough to show that there are no non-zero intertwining operators of type $\binom{\cW_{r,s}'}{\cL_{1,2}\,\cJ_{r,s}}$. Since $\cJ_{r,s}$ is a Verma module, this is equivalent to
\begin{equation}\label{eqn:desc_cond}
 \dim\hom_{A(V_c)}(A(\cL_{1,2})\otimes_{A(V_c)}\cJ_{r,s}(0), \cM_{r,s}) =0,
\end{equation}
by Theorem \ref{thm:Zhu_fus_rules}.

For $1\leq s\leq p-1$, $\cJ_{r,s}=\cV_{r+1,p-s}$, so the $L_0$-eigenvalues on $A(\cL_{1,2})\otimes_{A(V_c)}\cJ_{r,s}(0)=\cM_{r+1,p-s}$ are $h_{r+1,p-s\pm1}$. For $2\leq s\leq p-1$, these never equal the $L_0$-eigenvalues $h_{r,s\pm1}$ on $\cM_{r,s}$, proving the second assertion in the proposition for $s<p$. But when $s=1$, we have
\begin{equation*}
 h_{r+1,p-1+1}=h_{r+1,p}=h_{r,0}=h_{r,1-1},
\end{equation*}
so \eqref{eqn:desc_cond} fails. However, since
\begin{equation*}
\dim \hom_{A(V_c)}(\cM_{r+1,p-1},\CC v_{r,2}) =0,
\end{equation*}
we do get a non-zero intertwining operator of type $\binom{\cV_{r,2}'}{\cL_{1,2}\,\cL_{r,1}}$ induced by a non-zero homomorphism $A(\cL_{1,2})\otimes_{A(V_c)}\CC v_{r,1}\rightarrow\CC v_{r,2}$. This proves the first assertion of the proposition.

For $(r,s)=(1,p)$, $\cJ_{1,p}=\cV_{3,p}$, and the eigenvalues of $L_0$ on $\cM_{3,p}$ are
\begin{equation*}
 h_{3,p\pm1}=h_{3,p-1},h_{2,1}.
\end{equation*}
Neither equals the generalized eigenvalue $h_{1,p-1}$ of $L_0$ on $\cM_{1,p}$, so \eqref{eqn:desc_cond} holds, proving the second assertion of the proposition for $(r,s)=(1,p)$.
\end{proof}

\begin{rem}
 For $r\geq 2$ and $s=p$, there is also a non-zero intertwining operator $\cY$ of type $\binom{\cW_{r,p}'}{\cL_{1,2}\,\cL_{r,p}}$ induced by the identity on $\cM_{r,p}$, but we cannot conclude that its image includes $\cM_{r,p}\cong\CC v_{r,p-1}\oplus\CC v_{r-1,1}$, even though $\im\pi(\cY)=\cM_{r,p}$. The reason is that $\pi(\cY)$ is defined using the non-standard $\NN$-grading of Remark \ref{rem:non_std_N_grad} for $\cW_{r,p}'$. In particular, the projection $\pi_0$ does not quite correspond to projection onto conformal weight spaces, which means that we cannot conclude that $\im\pi(\cY)$ is contained in $\im\cY$.
\end{rem}

Using the intertwining operators we have obtained, we can prove:

\begin{prop}\label{prop:Pi_rs_nontrivial} \hspace{2em}
 \begin{enumerate}
  \item For $r\geq 1$ and $s=1$, there is a surjective $V_c$-module map $\cL_{1,2}\tens\cL_{r,1} \rightarrow \cL_{r,2}$.

  \item For $r\geq 1$ and $2\leq s\leq p-1$, there is a surjective $V_c$-module map $\cL_{1,2}\tens\cL_{r,s}\rightarrow\cL_{r,s-1}\oplus\cL_{r,s+1}$.

  \item For $(r,s)=(1,p)$, $(\cL_{1,2}\tens\cL_{1,p})(0)\cong\cM_{1,p}$ as $A(V_c)$-modules.
 \end{enumerate}
\end{prop}
\begin{proof}
 For the cases of $(r,s)$ that we are considering, we have shown that
 \begin{equation*}
  (\cL_{1,2}\tens\cL_{r,s})(0)=(\cL_{1,2}\tens\cL_{r,s})_{[h_{r,s-1}]}+(\cL_{1,2}\tens\cL_{r,s})_{[h_{r,s+1}]}
 \end{equation*}
and that $\Pi_{r,s}: \cW_{r,s}\rightarrow\cL_{1,2}\tens\cL_{r,s}$ is surjective.

When $s=1$, the image of any non-zero intertwining operator of type $\binom{\cV_{r,2}'}{\cL_{1,2}\,\cL_{r,1}}$ is a $C_1$-cofinite module in $\cO_c$ by \cite[Key Theorem]{Miy}. Thus the universal property of the tensor product induces a non-zero map $f:\cL_{1,2}\tens\cL_{r,1}\rightarrow\cV_{r,2}'$, whose image must contain the unique minimal non-zero submodule $\cL_{r,2}$. Moreover, $\im f$ is a quotient of $\cW_{r,1}\cong\cV_{r,0}\oplus\cV_{r,2}$ because $\Pi_{r,1}$ is surjective. As $\cL_{r,2}$ is the only non-zero quotient of $\cV_{r,0}\oplus\cV_{r,2}$ that is also a submodule of $\cV_{r,2}'$, it follows that  $\im f=\cL_{r,2}$, that is, we have a surjective map $\cL_{1,2}\tens\cL_{r,1}\rightarrow\cL_{r,2}$.

Similarly, for $2\leq s\leq p-1$ or $(r,s)=(1,p)$, Proposition \ref{prop:intwo_op_exist} and the universal property of tensor products yield a homomorphism $f:\cL_{1,2}\tens\cL_{r,s}\rightarrow\cW_{r,s}'$ whose image contains
$$\cW_{r,s}'(0)=(\cW_{r,s}')_{[h_{r,s-1}]}+(\cW_{r,s}')_{[h_{r,s+1}]}\cong\cM_{r,s}.$$
This forces $\dim\,(\cL_{1,2}\tens\cL_{r,s})(0)\geq 2$, so $\Pi_{r,s}\vert_{\cM_{r,s}}$ must be injective as well as surjective, proving the proposition in the $(r,s)=(1,p)$ case. When $2\leq s\leq p-1$, surjectivity of $\Pi_{r,s}$ implies that $\im f$ is generated by
\begin{equation*}
 (f\circ\Pi_{r,s})(v_{r,s\pm1})\in\cM_{r,s}\subseteq\cW_{r,s}'\cong\cV_{r,s-1}'\oplus\cV_{r,s+1}'.
\end{equation*}
These vectors generate a submodule isomorphic to $\cL_{r,s-1}\oplus\cL_{r,s+1}$, so we have a surjection $\cL_{1,2}\tens\cL_{r,s}\rightarrow\cL_{r,s-1}\oplus\cL_{r,s+1}$.
\end{proof}

The upper bound of Corollary \ref{cor:Pi_rs_surjective} and the lower bound of Proposition \ref{prop:Pi_rs_nontrivial} already provide strong constraints on the tensor product $\cL_{1,2}\tens\cL_{r,s}$. To fully identify this tensor product, we will need $\cL_{1,2}\tens\cL_{r,s}$ to be a self-contragredient $V_c$-module. This will follow from the rigidity of $\cL_{1,2}$ and $\cL_{r,s}$ in the tensor category $\cO_c$, which we prove for $\cL_{1,2}$ next.

\section{Rigidity, categorical dimensions, and some fusion rules}

In this section, we show that $\cO_c$ is a rigid (and also ribbon) tensor category, and we calculate the categorical dimensions of all simple modules $\cL_{r,s}$. In addition, we determine some tensor products in $\cO_c$ involving $\cL_{1,2}$, and some involving the modules $\cL_{r,1}$ for $r\geq 1$.

\subsection{Rigidity and categorical dimension for \texorpdfstring{$\cL_{1,2}$}{L{1,2}}}

We begin by showing that $\cL_{1,2}$ is rigid and self-dual in  $\cO_c$. Since $V_c = \cL_{1,1}$ is the unit object of $\cO_c$,  we first of all need an evaluation map $e: \cL_{1,2}\tens\cL_{1,2}\rightarrow\cL_{1,1}$ and a coevaluation $i:\cL_{1,1}\rightarrow\cL_{1,2}\tens\cL_{1,2}$.

The evaluation is easy: Since $\cL_{1,2}$ is self-contragredient with lowest conformal weight $h_{1,2}$, symmetries of intertwining operators from \cite{FHL, HLZ2} applied to (possibly a rescaling of) the vertex operator $Y_{\cL_{1,2}}$ yield an intertwining operator $\cE$ of type $\binom{\cL_{1,1}}{\cL_{1,2}\,\cL_{1,2}}$ such that
\begin{equation*}
 \cE(v_{1,2},x)v_{1,2}\in x^{-2 h_{1,2}}\big(\vac + x \cL_{1,1}[[x]]\big).
\end{equation*}
We then define the evaluation $e:\cL_{1,2}\boxtimes\cL_{1,2}\rightarrow\cL_{1,1}$ to be the unique map such that $e\circ\cY_\boxtimes =\cE$.

For the coevaluation, Proposition \ref{prop:top_level_eigenvectors} describes a homomorphism $\cV_{1,1}\rightarrow\cL_{1,2}\boxtimes\cL_{1,2}$. It will descend to a map $i: \cL_{1,1}\rightarrow\cL_{1,2}\boxtimes\cL_{1,2}$ such that
\begin{equation*}
 i(\vac)= -\pi_0(v_{1,2}\boxtimes v_{1,2})+2p\,\pi_0(L_{-1}v_{1,2}\boxtimes v_{1,2})
\end{equation*}
provided that $L_{-1} i(\vac)=0$ (since $L_{-1} v_{1,1}$ generates the maximal proper submodule of $\cV_{1,1}$). To prove this, we use the commutator formula \eqref{eqn:Vir_comm_form}, the iterate formula \eqref{eqn:Vir_it_form}, and the relation $(L_{-1}^2-\frac{1}{p}L_{-2})v_{1,2}=0$ in $\cL_{1,2}$ to compute
\begin{align*}
 L_{-1}\pi_0(L_{-1}v_{1,2}\boxtimes v_{1,2}) & =\pi_1(L_{-1}^2 v_{1,2}\boxtimes v_{1,2})+\pi_1(L_{-1}v_{1,2}\boxtimes L_{-1}v_{1,2})\nonumber\\
 & =\frac{1}{p}\pi_1(L_{-2}v_{1,2}\boxtimes v_{1,2})+L_{-1}\pi_0(v_{1,2}\boxtimes L_{-1}v_{1,2})-\pi_1(v_{1,2}\boxtimes L_{-1}^2 v_{1,2})\nonumber\\
 & = \frac{1}{p}\pi_1(v_{1,2}\boxtimes(L_{-1}+L_0)v_{1,2})-L_{-1}\pi_0(L_{-1}v_{1,2}\boxtimes v_{1,2})\nonumber\\
 &\qquad\qquad-\frac{1}{p}\pi_1(v_{1,2}\boxtimes L_{-2}v_{1,2}).
\end{align*}
We solve for $L_{-1}\pi_0(L_{-1}v_{1,2}\boxtimes v_{1,2})$ and apply the commutator formula \eqref{eqn:Vir_comm_form} to get
\begin{align*}
 L_{-1}\pi_0(L_{-1} v_{1,2}\boxtimes v_{1,2}) & = \frac{h_{1,2}}{2p}\pi_1(v_{1,2}\boxtimes v_{1,2})+\frac{1}{2p}\pi_1(v_{1,2}\boxtimes L_{-1}v_{1,2})\nonumber\\
 &\qquad\qquad+\frac{1}{2p}\pi_1((L_{-1}-L_0)v_{1,2}\boxtimes v_{1,2})\nonumber\\
 & =\frac{1}{2p}\big(\pi_1(L_{-1}v_{1,2}\boxtimes v_{1,2}+\pi_1(v_{1,2}\boxtimes L_{-1}v_{1,2})\big)\nonumber\\
 & = \frac{1}{2p} L_{-1}\pi_0(v_{1,2}\boxtimes v_{1,2}).
\end{align*}
Thus indeed
\begin{equation*}
 L_{-1}\big(-\pi_0(v_{1,2}\boxtimes v_{1,2})+2p\,\pi_0(L_{-1}v_{1,2}\boxtimes v_{1,2})\big) =0,
\end{equation*}
showing that the coevaluation $i$ exists.

We now prove the rigidity of $\cL_{1,2}$:
\begin{thm}\label{rigidityofl12}
 The module $\cL_{1,2}$ is rigid and self-dual in the tensor category $\cO_c$.
\end{thm}
\begin{proof}
We need to show that the compositions
\begin{equation*}
 \cL_{1,2}\xrightarrow{l^{-1}}\cL_{1,1}\boxtimes\cL_{1,2}\xrightarrow{i\boxtimes\Id}(\cL_{1,2}\boxtimes\cL_{1,2})\boxtimes\cL_{1,2}\xrightarrow{\cA^{-1}}\cL_{1,2}\boxtimes(\cL_{1,2}\boxtimes\cL_{1,2})\xrightarrow{\Id\boxtimes e}\cL_{1,2}\boxtimes\cL_{1,1}\xrightarrow{r}\cL_{1,2}
\end{equation*}
and
\begin{equation*}
 \cL_{1,2}\xrightarrow{r^{-1}}\cL_{1,2}\boxtimes\cL_{1,1}\xrightarrow{\Id\boxtimes i}\cL_{1,2}\boxtimes(\cL_{1,2}\boxtimes\cL_{1,2})\xrightarrow{\cA}(\cL_{1,2}\boxtimes\cL_{1,2})\boxtimes\cL_{1,2}\xrightarrow{e\boxtimes\Id}\cL_{1,1}\boxtimes\cL_{1,2}\xrightarrow{l}\cL_{1,2}
\end{equation*}
are identical non-zero multiples of the identity (we can then rescale either $e$ or $i$ to get the identity). By Lemma 4.2.1 and Corollary 4.2.2 of \cite{CMY3}, these two compositions are equal, so it is enough to show that one of them is non-zero. We shall show that the second, which we label $\mathfrak{R}$ for convenience, is non-zero. In particular, we just need to show that $\langle v_{1,2},\mathfrak{R}(v_{1,2})\rangle\neq 0$, where $\langle\cdot,\cdot\rangle$ is the nondegenerate invariant bilinear form on $\cL_{1,2}$ such that $\langle v_{1,2},v_{1,2}\rangle=1$.

To compute $\langle v_{1,2},\mathfrak{R}(v_{1,2})\rangle$, we first use the definition of $r$ (see \cite[Section 12.2]{HLZ8} or \cite[Section 3.3.3]{CKM}) to get
\begin{equation*}
 r^{-1}(v_{1,2})=r^{-1}\left(\pi_0\left(e^{L_{-1}}Y_{\cL_{1,2}}(\vac,-1)v_{1,2}\right)\right)= \pi_0\left(\cY_\tens(v_{1,2},1)\vac\right).
\end{equation*}
Then we observe that $i(\vac)$ is the coefficient of the monomial $x^{-2 h_{1,2}} (\log x)^0$ in
\begin{align*}
 x^{L_0}\big( 2p  \,(L_{-1} x^{-L_0}v_{1,2}\boxtimes x^{-L_0}v_{1,2}) & -x^{-L_0}v_{1,2}\boxtimes x^{-L_0} v_{1,2}\big)\nonumber\\
 & =2p\,x\cY_\tens(L_{-1} v_{1,2},x)v_{1,2}-\cY_\tens(v_{1,2},x)v_{1,2}\nonumber\\
 & =\left(2p\,x\dfrac{d}{dx}-1\right)\cY_\tens(v_{1,2},x)v_{1,2}.
\end{align*}
Thus $\langle v_{1,2},\mathfrak{R}(v_{1,2})\rangle$ is the coefficient of $x^{-2h_{1,2}}(\log x)^0$ in
\begin{align*}
\left(2p\,x\frac{d}{dx}-1\right)\left \langle v_{1,2}, \left[l\circ(e\tens\Id)\circ\cA\circ\cY_\tens\right](v_{1,2},1)\cY_\tens(v_{1,2},x)v_{1,2}\right\rangle.
\end{align*}
This series is the expansion of a multivalued analytic function on the punctured unit disk. Alternatively, it is a single-valued analytic function on the simply-connected region
\begin{equation*}
 U_1=\lbrace z\in\CC\,\vert\,\vert z\vert <1\rbrace\setminus(-1,0],
\end{equation*}
where we choose the single-valued branch corresponding to the branch of logarithm
\begin{equation*}
\log z = \ln \vert z\vert +i\,\arg z
\end{equation*}
with $-\pi<\arg z<\pi$. From the definitions of $\cA$, $e$, and $l$ (again see \cite[Section 12.2]{HLZ8} or \cite[Section 3.3]{CKM}), the analytic continuation of this function to the simply-connected region
\begin{equation*}
 U_2 =\lbrace z\in\CC\,\vert\,\vert z\vert>\vert 1-z\vert>0\rbrace\setminus[1,\infty)=\lbrace z\in\CC\,\vert\,\mathrm{Re}\,z>1/2\rbrace\setminus[1,\infty)
\end{equation*}
is
\begin{align}\label{eqn:rigidity_iterate}
 \left(2p\,x\frac{d}{dx}-1\right) &\,\left \langle v_{1,2}, \left[l\circ(e\tens\Id)\circ\cY_\tens\right]\left(\cY_\tens(v_{1,2},1-x)v_{1,2},x\right)v_{1,2}\right\rangle\nonumber\\
 & = \left( 2p\,x\frac{d}{dx}-1\right)\left\langle v_{1,2}, [l\circ\cY_\tens](\cE(v_{1,2},1-x)v_{1,2},x)v_{1,2},x)v_{1,2}\right\rangle\nonumber\\
 & =\left(2p\,x\frac{d}{dx}-1\right)\left\langle v_{1,2}, Y_{\cL_{1,2}}(\cE(v_{1,2},1-x)v_{1,2},x)v_{1,2}\right\rangle.
\end{align}
This expression should be interpreted as a double series in $1-x$ and $x$, with the branch of logarithm $\log z$ used for both $1-x$ and $x$. Thus to show $\langle v_{1,2},\mathfrak{R}(v_{1,2})\rangle\neq 0$, we need to find the explicit expansion of \eqref{eqn:rigidity_iterate} as a series in $x$ and $\log x$ on $U_1\cap U_2$, and then extract the coefficient of $x^{-2 h_{1,2}}(\log x)^0$.

Compositions of intertwining operators involving $C_1$-cofinite modules for the Virasoro algebra are solutions to Belavin-Polyakov-Zamolodchikov equations \cite{BPZ, Hu_Vir_tens}. When all insertions in the intertwining operators are lowest-conformal-weight vectors $v_{1,2}\in\cL_{1,2}$, the specific differential equation appears in \cite[Equation 8.71]{Fran} (with the parameters of that equation specialized to $t=\frac{1}{p}$ and $h_0=h_1=h_2=h_3=h_{1,2}$); see also \cite[Section 4.2]{TW}. Namely, on $U_1$, the series
\begin{equation}\label{eqn:phi(x)}
 \phi(x) = \left \langle v_{1,2}, \left[l\circ(e\tens\Id)\circ\cA\circ\cY_\tens\right](v_{1,2},1)\cY_\tens(v_{1,2},x)v_{1,2}\right\rangle
\end{equation}
is a solution to the second-order regular-singular-point differential equation
\begin{equation}\label{eqn:BPZ-equation}
 x(1-x)\phi''(x)+\frac{1}{p}(1-2x)\phi'(x)-\frac{h_{1,2}}{p} x^{-1}(1-x)^{-1}\phi(x)=0.
\end{equation}
For a detailed vertex algebraic derivation of this equation, see \cite[Proposition 4.1.2]{CMY2}.

For the reader's convenience, we summarize how \eqref{eqn:BPZ-equation} is derived in \cite{CMY2}. First, we set
\begin{equation*}
\Phi(x_1,x_2) =\langle v_{1,2},\cY_1(v_{1,2},x_1)\cY_2(v_{1,2},x_2)v_{1,2}\rangle
\end{equation*}
where $\cY_1=l\circ(e\tens\Id)\circ\cA\circ\cY_\tens$ and $\cY_2=\cY_\tens$. Then the relation $(L_{-1}^2-\frac{1}{p} L_{-2})v_{1,2} =0$ in $\cL_{1,2}$ and the $L_{-1}$-derivative property of intertwining operators imply that
\begin{align*}
\partial_{x_2}^2\Phi(x_1,x_2) & =\frac{1}{p}\langle v_{1,2},\cY_1(v_{1,2},x_1)\cY_2(L_{-2}v_{1,2},x_2)v_{1,2}\rangle.
\end{align*}
Using the iterate formula \eqref{eqn:Vir_it_form} and commutator formula \eqref{eqn:Vir_comm_form}, as well as the $L_{-1}$-derivative property and the relations $L_0 v_{1,2}=h_{1,2}v_{1,2}$ and $L_n v_{1,2}=0$ for $n>0$, we can express the right side of this equation in terms of $\Phi(x_1,x_2)$ and its first partial derivatives. That is, we obtain a second-order partial differential equation for $\Phi(x_1,x_2)$. However, we want an ordinary differential equation for $\phi(x)=\Phi(1,x)$. For this, we use the $L_0$-conjugation formula for intertwining operators (see for example \cite[Proposition 3.36(b)]{HLZ2}) to write
\begin{equation*}
\Phi(x_1,x_2) = x_1^{-2h_{1,2}}\phi\left(\frac{x_2}{x_1}\right),
\end{equation*}
which implies the relations
\begin{align*}
\partial_{x_1}\Phi(x_1,x_2)\vert_{(x_1,x_2)=(1,x)} & =-2h_{1,2}\phi(x)-x\phi'(x),\nonumber\\
 \partial_{x_2}\Phi(x_1,x_2)\vert_{(x_1,x_2)=(1,x)} & =\phi'(x),\qquad\partial_{x_2}^2\Phi(x_1,x_2)\vert_{(x_1,x_2)=(1,x)} =\phi''(x).
\end{align*}
Plugging these expressions into the partial differential equation for $\Phi(x_1,x_2)$ then yields \eqref{eqn:BPZ-equation}; see \cite{CMY2} for further details.

Now, since the composition of intertwining operators $\phi(x)$ in \eqref{eqn:phi(x)} satisfies the differential equation \eqref{eqn:BPZ-equation},
its analytic continuation
\begin{equation*}
 \psi(x)=\left\langle v_{1,2}, Y_{\cL_{1,2}}(\cE(v_{1,2},1-x)v_{1,2},x)v_{1,2}\right\rangle
\end{equation*}
solves the same differential equation on $U_2$. If we write
\begin{equation}\label{eqn:var_change}
 \psi(x)=x^{1/2p}(1-x)^{1/2p} f(x)
\end{equation}
for some analytic function $f(x)$, then a tedious but straightforward calculation shows that $f(x)$ solves the hypergeometric differential equation
\begin{equation}\label{eqn:hypgeo_diff_eq}
 x(1-x) f''(x) +\frac{2}{p}(1-2x)f'(x)+\frac{1}{p}\left(1-\frac{3}{p}\right)f(x)=0,
\end{equation}
whose solutions are well known (see for example \cite[Section 15.10]{DLMF}).

For $p\geq 3$, \eqref{eqn:hypgeo_diff_eq} has the following basis of solutions on $U_2$ (see \cite[Equations 15.10.13 and 15.10.14]{DLMF}):
\begin{align}\label{eqn:hypgeo_solns}
 f_1(x) & = x^{-1/p}{}_2 F_1\left(\frac{1}{p},1-\frac{1}{p};\frac{2}{p};-\frac{1-x}{x}\right)\nonumber\\
 f_2(x) & =x^{-1/p}(1-x)^{1-2/p}{}_2 F_1\left(\frac{1}{p},1-\frac{1}{p};2-\frac{2}{p};-\frac{1-x}{x}\right).
\end{align}
On the other hand, the $L_0$-conjugation formula and the definition of $\cE$ show that
\begin{align}\label{eqn:psi_series}
 (1-x)^{2h_{1,2}}\psi(x) & = \left(\frac{1-x}{x}\right)^{2h_{1,2}}\left\langle v_{1,2}, Y_{\cL_{1,2}}\left(\cE\left(v_{1,2},\frac{1-x}{x}\right)v_{1,2},1\right)v_{1,2}\right\rangle\nonumber\\
 & =\left(\frac{1-x}{x}\right)^{2 h_{1,2}}\left(\langle v_{1,2},Y_{\cL_{1,2}}(\vac,1)v_{1,2}\rangle\left(\frac{1-x}{x}\right)^{-2h_{1,2}} +\ldots\right)\nonumber\\
 & \in 1+\left(\frac{1-x}{x}\right)\CC\left[\left[\frac{1-x}{x}\right]\right].
\end{align}
By examining the powers of $\frac{1-x}{x}$ in \eqref{eqn:var_change} and \eqref{eqn:hypgeo_solns}, we see that
\begin{equation*}
 \psi(x)= x^{1/2p}(1-x)^{1/2p} f_2(x) = (1-x)^{-2h_{1,2}}\left(1+\frac{1-x}{x}\right)^{1/2p} {}_2 F_1\left(\frac{1}{p},1-\frac{1}{p};2-\frac{2}{p};-\frac{1-x}{x}\right).
\end{equation*}
Now we need to expand $\psi(x)$ in $U_1$ as a series in $x$. By the connection formulas for hypergeometric functions (see for example \cite[Equation 15.10.18]{DLMF}), we have
\begin{align*}
 f_2(x) = \frac{\Gamma\big(1-\frac{2}{p}\big)\Gamma\big(2-\frac{2}{p}\big)}{\Gamma\big(1-\frac{1}{p}\big)\Gamma\big(2-\frac{3}{p}\big)} &  {}_2 F_1\left(\frac{1}{p},\frac{3}{p}-1;\frac{2}{p}; x\right)\nonumber\\ &+\frac{\Gamma\big(\frac{2}{p}-1\big)\Gamma\big(2-\frac{2}{p}\big)}{\Gamma\big(\frac{1}{p}\big)\Gamma\big(1-\frac{1}{p}\big)} x^{1-2/p} {}_2 F_1\left(\frac{1}{p},1-\frac{1}{p};2-\frac{2}{p}; x\right)
\end{align*}
on $U_1\cap U_2$. Only the second term contributes to the coefficient of $x^{-2 h_{1,2}}$ in $(2p\,x\frac{d}{dx}-1)\psi(x)$:
\begin{align*}
 \left(2p\,x\dfrac{d}{dx}-1\right) & x^{-2h_{1,2}}(1-x)^{1/2p} {}_2 F_1\left(\frac{1}{p},1-\frac{1}{p};2-\frac{2}{p}; x\right)\nonumber\\
 & =x^{-2h_{1,2}}(1-x)^{1/2p}\left[\left(-4p\,h_{1,2} -\frac{x}{1-x}-1\right){}_2 F_1\left(\frac{1}{p},1-\frac{1}{p};2-\frac{2}{p}; x\right)\right.\nonumber\\
 &\qquad\qquad\qquad\qquad\qquad\qquad\left.+2p\,x\,{}_2 F_1 '\left(\frac{1}{p},1-\frac{1}{p};2-\frac{2}{p}; x\right) \right]\nonumber\\
 & \in x^{-2 h_{1,2}}\big(2(p-2)+x\CC[[x]]\big).
\end{align*}
We conclude that when $p\geq 3$,
\begin{equation*}
 \langle v_{1,2}, \mathfrak{R}(v_{1,2})\rangle =2(p-2)\frac{\Gamma\big(\frac{2}{p}-1\big)\Gamma\big(2-\frac{2}{p}\big)}{\Gamma\big(\frac{1}{p}\big)\Gamma\big(1-\frac{1}{p}\big)} = -2(p-2)\frac{\sin(\pi/p)}{\sin(2\pi/p)}=-\frac{p-2}{\cos(\pi/p)}\neq 0,
\end{equation*}
using \cite[Equation 5.5.3]{DLMF} in the second equality. This proves $\cL_{1,2}$ is rigid when $p\geq 3$.

For $p = 2$, the equation \eqref{eqn:hypgeo_diff_eq} has logarithmic solutions. On the simply-connected region
\begin{equation*}
 1-U_1 =\lbrace z\in\CC\,\vert\, \vert 1-z\vert<1\rbrace\setminus [1,2),
\end{equation*}
which has non-empty intersection with $U_2$, \eqref{eqn:hypgeo_diff_eq} has the following basis of solutions:
\begin{align*}
 f_1(x) & = {}_2 F_1\left(\frac{1}{2},\frac{1}{2};1;1-x\right)\nonumber\\
 f_2(x) & = f_1(x)\log(1-x)+G(1-x),
\end{align*}
where $G(x)$ is a power series (which we may assume has no constant term). Since \eqref{eqn:psi_series} shows that $(1-x)^{-2h_{1,2}}\psi(x)$ is analytic at $x=1$ with value $1$, we must have
\begin{equation*}
 \psi(x)=x^{1/4}(1-x)^{1/4}f_1(x) = x^{1/4}(1-x)^{1/4}{}_2 F_1\left(\frac{1}{2},\frac{1}{2};1;1-x\right)
\end{equation*}
on $(1-U_1)\cap U_2$. We need to expand $\psi(x)$ on $U_1$ as a series in $x$; to do so, we use \cite[Equation 15.8.10]{DLMF}, which states that
\begin{align*}
 {}_2 F_1\left(\frac{1}{2},\frac{1}{2};1;1-x\right) = -\frac{1}{\Gamma\big(\frac{1}{2}\big)\Gamma\big(\frac{1}{2}\big)}\sum_{n= 0}^\infty \frac{\big(\frac{1}{2}\big)_n \big(\frac{1}{2}\big)_n}{(n!)^2} x^n\cdot\left(\log x +C_n\right),
\end{align*}
for $x\in U_1\cap(1-U_1)$, where the constants $C_n$ can be expressed in terms of the digamma function. Thus on the non-empty open region $U_1\cap(1-U_1)\cap U_2$,
\begin{align*}
 \left(4x\dfrac{d}{dx}-1\right)\psi(x) & =x^{1/4}(1-x)^{1/4}\left[4\cdot\frac{1}{4}-\frac{x}{1-x}-1\right] {}_2 F_1\left(\frac{1}{2},\frac{1}{2};1;1-x\right)\nonumber\\
 &\qquad\qquad  -\frac{4\,x^{1/4}(1-x)^{1/4}}{\Gamma\big(\frac{1}{2}\big)\Gamma\big(\frac{1}{2}\big)}\sum_{n= 0}^\infty \frac{\big(\frac{1}{2}\big)_n \big(\frac{1}{2}\big)_n}{(n!)^2}\cdot\left[n x^n(\log x+C_n)+x^n \right],
\end{align*}
and the coefficient of $x^{1/4}$ is
\begin{equation*}
 -\frac{4}{\Gamma\big(\frac{1}{2}\big)\Gamma\big(\frac{1}{2}\big)} = -\frac{4}{\pi/\sin(\pi/2)} =-\frac{4}{\pi}\neq 0.
\end{equation*}
We conclude that $\langle v_{1,2},\mathfrak{R}(v_{1,2})\rangle\neq 0$ and thus $\cL_{1,2}$ is rigid when $p=2$.
\end{proof}

Our calculations allow us to describe the evaluation and coevaluation for $\cL_{1,2}$ explicitly. If we fix a non-zero lowest-conformal weight vector $v_{1,2}\in\cL_{1,2}$, we take the evaluation to be
\begin{align*}
 e: \cL_{1,2}\tens\cL_{1,2} & \rightarrow \cL_{1,1}\nonumber\\
 \pi_0(v_{1,2}\boxtimes v_{1,2}) & \mapsto \vac.
\end{align*}
The $L_0$-conjugation formula determines $e$ on the other possibly linearly independent lowest-conformal-weight vector:
\begin{align*}
 e\left(\pi_0(L_{-1}v_{1,2}\boxtimes v_{1,2})\right) & =e\left(\pi_0(L_0(v_{1,2}\boxtimes v_{1,2})-L_0v_{1,2}\boxtimes v_{1,2}-v_{1,2}\tens L_0v_{1,2})\right)\nonumber\\
 & = (L_0-2h_{1,2}) e(\pi_0(v_{1,2}\tens v_{1,2}))= -2h_{1,2}\vac.
\end{align*}
With this choice of evaluation, we must take the coevaluation as follows:
\begin{equation*}
 i(\vac) = \left\lbrace\begin{array}{lll}
                        \frac{\cos(\pi/p)}{p-2}\left(\pi_0(v_{1,2}\tens v_{1,2})-2p\,\pi_0(L_{-1} v_{1,2}\tens v_{1,2})\right) & \text{if} & p\geq 3\\
                        \frac{\pi}{4}\left(\pi_0(v_{1,2}\tens v_{1,2})-4\,\pi_0(L_{-1} v_{1,2}\tens v_{1,2})\right) & \text{if} & p=2
                       \end{array}
                       \right. .
\end{equation*}
Using these explicit evaluation and coevaluation, we determine the categorical dimension
\begin{equation*}
 \dim_{\cO_c} \cL_{1,2} = e\circ\cR\circ(\theta\tens\Id)\circ i: \cL_{1,1}\rightarrow\cL_{1,1}
\end{equation*}
of $\cL_{1,2}$ in $\cO_c$, where $\theta=e^{2\pi i L_0}$ is the ribbon twist on $\cO_c$:
\begin{prop}\label{prop:L12_dim}
 In the tensor category $\cO_c$, $\dim_{\cO_c} \cL_{1,2}=-2\cos(\pi/p)\,\Id_{\cL_{1,1}}$.
\end{prop}
\begin{proof}
Since $\cL_{1,1}$ is simple, the dimension is just a scalar multiple of the identity. Using $a_p$ to denote $\frac{\cos(\pi/p)}{p-2}$ or $\frac{\pi}{4}$ according as $p\geq 3$ or $p=2$ (note that $a_2=\lim_{p\to 2} a_p$), we calculate
\begin{align*}
 \dim_{\cO_c} \cL_{1,2} & \, : \vac \mapsto  a_p\,e^{2\pi i h_{1,2}} (e\circ\cR)\left(\pi_0(v_{1,2}\tens v_{1,2})-2p\,\pi_0(L_{-1} v_{1,2}\tens v_{1,2})\right)\nonumber\\
 & =a_p\,e^{2\pi i h_{1,2}} (e\circ\pi_0)\left(e^{L_{-1}}\cY_\tens(v_{1,2},e^{\pi i}) v_{1,2}-2p\,e^{L_{-1}}\cY_\tens(v_{1,2},e^{\pi i})L_{-1}v_{1,2}\right)\nonumber\\
 & =a_p\,e^{2\pi i h_{1,2}} e^{\pi i L_0}(e\circ\pi_0)\left( e^{-\pi i L_0} v_{1,2}\tens e^{-\pi i L_0} v_{1,2}-2p\,(e^{-\pi i L_0} v_{1,2}\tens e^{-\pi i L_0} L_{-1} v_{1,2})\right)\nonumber\\
 & = a_p\,(e\circ\pi_0)\left(v_{1,2}\tens v_{1,2}+2p\,(v_{1,2}\tens L_{-1}v_{1,2})\right)\nonumber\\
 & =a_p\,e\left(\pi_0(v_{1,2}\tens v_{1,2})-2p\,\pi_0(L_{-1}v_{1,2}\tens v_{1,2})\right)\nonumber\\
 & =a_p\,(1+4ph_{1,2})\vac\nonumber\\
 & =2a_p(2-p)\vac = -2\cos(\pi/p)\vac
\end{align*}
as required.
\end{proof}
Note that the dimension formula is valid for all $p\geq 2$; in particular, $\dim_{\cO_c} \cL_{1,2} =0$ when $p=2$. Note also that if we ignore the braiding and twist isomorphisms, we still get
\begin{equation}\label{eqn:L12_left_trace}
 e\circ i = -2\cos(\pi/p)\,\Id_{\cL_{1,1}}.
\end{equation}
This quantity is an invariant of the tensor category structure on $\cO_c$ (it depends on the associativity isomorphisms, but not on the braiding or ribbon twist).

\subsection{Rigidity of \texorpdfstring{$\cO_c$}{Oc} and some fusion rules}\label{sec:fus_and_rig}

In this section, we determine the tensor products of $\cL_{1,2}$ with the irreducible modules in $\cO_c$, and we prove that $\cO_c$ is rigid. But first, we establish rigidity and fusion products of the modules $\cL_{r,1}$:
\begin{thm}\label{thm:Lr1_fus_rules}
 The irreducible $V_c$-modules $\cL_{r,1}$ are rigid for $r\geq 1$, and
 \begin{equation}\label{eqn:Lr1_fus_rules}
  \cL_{r,1}\tens\cL_{r',1} \cong \bigoplus_{\substack{k = |r-r'|+1\\ k+r+r' \equiv 1\; ({\rm mod}\; 2)}}^{r+r'-1} \cL_{k,1}
 \end{equation}
for $r,r'\geq 1$.
\end{thm}
\begin{proof}
  We use a realization of $V_c$ as the fixed-point subalgebra of a compact automorphism group of an abelian intertwining algebra. The triplet vertex operator algebra $\cW(p)$ is a $C_2$-cofinite vertex operator algebra extension of $V_c$; its automorphism group is $PSL(2,\CC)$ \cite{ALM} and $V_c$ is the fixed-point subalgebra. In particular, $V_c$ is the fixed-point subalgebra of the compact automorphism group $SO(3,\RR)$ acting on $\cW(p)$.

  The triplet $\cW(p)$ admits a simple current extension $\cA(p)$ called the doublet \cite{AM_doub}; it is an abelian intertwining algebra. The Lie algebra $\mathfrak{sl}_2$ acts by derivations on $\cA(p)$ \cite[Remark 2]{ACGY}, and this action exponentiates to an action of $SL(2,\CC)$ by automorphisms. In particular, $V_c$ is the fixed-point subalgebra of the compact automorphism group $SU(2)$ acting continuously on $\cA(p)$. As an $SU(2)\times V_c$-module,
\begin{equation*}
 \cA(p)\cong\bigoplus_{r\geq 1} M_r\otimes \cL_{r,1}
\end{equation*}
where $M_r$ is the $r$-dimensional irreducible $SU(2)$-module (again see \cite[Remark 2]{ACGY}).

Now by the main theorems of \cite{McR}, the modules $\cL_{r,1}$ are the simple objects of a semisimple tensor subcategory of $\cO_c$ that is braided tensor equivalent to $\rep SU(2)$ (twisted by an abelian $3$-cocycle of $\ZZ/2\ZZ$). In particular, the modules $\cL_{r,1}$ are rigid (since finite-dimensional $SU(2)$-modules are rigid) and the fusion rules \eqref{eqn:Lr1_fus_rules} hold.
\end{proof}

Now we can determine the tensor products of $\cL_{1,2}$ with most irreducible modules in $\cO_c$:
\begin{thm}\label{thm:L12_fus_rules}
 For $r\geq 1$ and $1\leq s\leq p$, the irreducible $V_c$-module $\cL_{r,s}$ is rigid. Moreover,
 \begin{equation}\label{eqn:most_L12_fusion}
  \cL_{1,2}\boxtimes\cL_{r,s}\cong\left\lbrace\begin{array}{lll}
                                               \cL_{r,2} & \text{if} & s=1 \\
                                               \cL_{r,s-1}\oplus\cL_{r,s+1} & \text{if} & 2\leq s\leq p-1
                                              \end{array} \right.
 \end{equation}
 for all $r\geq 1$.
\end{thm}
\begin{proof}
We prove the theorem by induction on $s$. For $s=1$, Theorem \ref{thm:Lr1_fus_rules} shows that $\cL_{r,1}$ is rigid, but we still need to determine $\cL_{1,2}\tens\cL_{r,1}$. We will prove that this tensor product is $\cL_{r,2}$ by induction on $r$, with the base case $\cL_{1,2}\tens\cL_{1,1}\cong\cL_{1,2}$ clear because $\cL_{1,1}$ is the unit object of $\cO_c$.

Now assume that we know $\cL_{1,2}\tens\cL_{r,1}\cong\cL_{r,2}$ for some $r\geq 1$, and consider $\cL_{r+1,1}$. Because $\cL_{1,2}$ is rigid, the tensoring functor $\cL_{1,2}\tens\bullet$ is exact, so by \eqref{eqn:Lr1_fus_rules} and the inductive hypothesis, we have an injection
\begin{equation*}
 \cL_{1,2}\tens\cL_{r+1,1}\rightarrow\cL_{1,2}\tens(\cL_{2,1}\tens\cL_{r,1})\cong\cL_{2,1}\tens\cL_{r,2}.
\end{equation*}
Now on the one hand, Proposition \ref{prop:conf_wts}(1) says that the conformal weights of $\cL_{1,2}\tens\cL_{r+1,1}$ are contained in $\lbrace h_{r+1,0}+\NN\rbrace\cup\lbrace h_{r+1,2}+\NN\rbrace$, while on the other hand, Proposition \ref{prop:conf_wts}(2) says that the weights are contained in $\lbrace h_{r-1,2}+\NN\rbrace\cup\lbrace h_{r+1,2}+\NN\rbrace$.
Since
\begin{equation*}
 h_{r+1,0}-h_{r\pm1,2} =(r\mp r)\frac{p}{2}+r\pm1-p^{-1}\notin\ZZ,
\end{equation*}
we have $h_{r+1,0}\notin\lbrace h_{r-1,2}+\NN\rbrace\cup\lbrace h_{r+1,2}+\NN\rbrace$. Thus $v_{r+1,0}$ is in the kernel of the surjection
\begin{equation*}
 \Pi_{r+1,1}: \cV_{r+1,0}\oplus\cV_{r+1,2}\rightarrow\cL_{1,2}\tens\cL_{r+1,1}
\end{equation*}
from Section \ref{sec:first_fus}, and so there is a surjective map $\cV_{r+1,2}\rightarrow\cL_{1,2}\tens\cL_{r+1,1}$. But now because $\cL_{1,2}$ and $\cL_{r+1,1}$ are rigid and self-dual, their tensor product is also rigid and we have isomorphisms
\begin{equation*}
 \cL_{1,2}\tens\cL_{r+1,1}\cong\cL_{r+1,1}\tens\cL_{1,2}\cong\cL_{r+1,1}'\tens\cL_{1,2}'\cong(\cL_{1,2}\tens\cL_{r+1,1})'.
\end{equation*}
As $\cL_{r+1,2}$ is the only quotient of $\cV_{r+1,2}$ that is self-contragredient, we conclude that $\cL_{1,2}\tens\cL_{r+1,1}\cong\cL_{r+1,2}$. This proves the $s=1$ case of the theorem.

Now assume by induction that for all $r\geq 1$ and some $s\in\lbrace 1,\ldots, p-1\rbrace$, $\cL_{r,s}$ is rigid and \eqref{eqn:most_L12_fusion} holds. Then for all $r\geq 1$, $\cL_{r,s+1}$ is also rigid, since it is a direct summand of the tensor product of rigid objects. If $s\leq p-2$, we still need to compute the fusion products $\cL_{1,2}\tens\cL_{r,s+1}$.
By Corollary \ref{cor:Pi_rs_surjective} and  Proposition \ref{prop:Pi_rs_nontrivial}, this tensor product is a homomorphic image of $\cV_{r,s}\oplus\cV_{r,s+2}$ that has $\cL_{r,s}\oplus\cL_{r,s+2}$ as a quotient. Also, since $\cL_{1,2}$ and $\cL_{r,s+1}$ are both rigid and self-dual, their tensor product is also rigid and self-dual. Thus $\cL_{1,2}\tens\cL_{r,s+1}$ also contains $\cL_{r,s}\oplus\cL_{r,s+2}$ as a submodule. As the only such homomorphic image of $\cV_{r,s}\oplus\cV_{r,s+2}$ is $\cL_{r,s}\oplus\cL_{r,s+2}$ itself, this proves the fusion rules of the theorem in the $s+1$ case.
\end{proof}

We shall describe the fusion products $\cL_{1,2}\tens\cL_{r,p}$ soon, but first note that we have now proved that all simple modules in $\cO_c$ are rigid. This means we can use \cite[Theorem 4.4.1]{CMY2} to extend rigidity to general finite-length modules in $\cO_c$:
\begin{thm}\label{thm:rigidity}
For $c=13-6p-6p^{-1}$ with $p > 1$ an integer, the tensor category $\cO_c$ of $C_1$-cofinite grading-restricted generalized $V_c$-modules is rigid. Moreover, it is a braided ribbon tensor category with natural twist isomorphism $\theta=e^{2\pi i L_0}$.
\end{thm}

As another consequence of Theorem \ref{thm:L12_fus_rules}, we can derive some more fusion rules in $\cO_c$:
\begin{thm}\label{thm:Lr_Ls_fusion}
 For $r\geq 1$ and $1\leq s\leq p$,
 \begin{equation*}
  \cL_{r,1}\tens\cL_{1,s}\cong \cL_{r,s}.
 \end{equation*}
\end{thm}
\begin{proof}
 The $s=1$ case is clear and the $s=2$ case was proved in Theorem \ref{thm:L12_fus_rules}. We can prove the general case by induction on $s$. In particular, for $2\leq s\leq p-1$, Theorem \ref{thm:L12_fus_rules} shows that we have an exact sequence
 \begin{equation*}
  0\longrightarrow\cL_{1,s-1}\longrightarrow\cL_{1,2}\tens\cL_{1,s}\longrightarrow\cL_{1,s+1}\longrightarrow 0.
 \end{equation*}
Since $\cL_{r,1}$ is rigid, the tensoring functor $\cL_{r,1}\tens\bullet$ is exact, and the inductive hypothesis implies that there is an exact sequence
\begin{equation*}
 0\longrightarrow\cL_{r,s-1}\longrightarrow\cL_{1,2}\tens\cL_{r,s}\longrightarrow\cL_{r,1}\tens\cL_{1,s+1}\longrightarrow 0.
\end{equation*}
Since $\cL_{1,2}\tens\cL_{r,s}\cong \cL_{r,s-1}\oplus\cL_{r,s+1}$ by Theorem \ref{thm:L12_fus_rules}, it follows that $\cL_{r,1}\tens\cL_{1,s+1}\cong \cL_{r,s+1}$.
\end{proof}

We now turn to the fusion products $\cL_{1,2}\tens\cL_{r,p}$. In the next section, we will show that these modules are projective covers of $\cL_{r,p-1}$ in a certain tensor subcategory of $\cO_c$, so we will use the notation $\cP_{r,p-1}=\cL_{1,2}\tens\cL_{r,p}$. First we handle $r=1$:
\begin{prop}\label{prop:P1_structure}
 The tensor product $\cP_{1,p-1}$ is a self-dual indecomposable length-$3$ module with subquotients as indicated in the diagram
 \begin{equation*}
  \xymatrix{
  \cL_{1,p-1} \ar[r] \ar[rd] & (\cV_{1,p-1}/\cV_{3,p-1})' \ar[r] \ar[d] & \cL_{2,1} \ar[d] \\
  & \cP_{1,p-1} \ar[r] \ar[rd] & \cV_{1,p-1}/\cV_{3,p-1} \ar[d] \\
  & & \cL_{1,p-1} \\
  }
 \end{equation*}
and Loewy diagram
\begin{equation*}
\xymatrixrowsep{1pc}
\xymatrixcolsep{.75pc}
 \xymatrix{
 & & \cL_{1,p-1} \\
 \cP_{1,p-1}: & \cL_{2,1} \ar[ru] & \\
& & \cL_{1,p-1} \ar[lu]\\
 }
\end{equation*}
\end{prop}
\begin{proof}
 First, $\cP_{1,p-1}$ is self-dual because $\cL_{1,2}$ and $\cL_{r,p}$ are self-dual and because the tensor product is commutative. Also, since $\Pi_{1,p-1}$ is surjective by Corollary \ref{cor:Pi_rs_surjective}, $\cP_{1,p-1}$ is a quotient of the generalized Verma module $\cV_{1,p-1}^{(2)}$. As this generalized Verma module has a unique maximal proper submodule (the sum of all proper submodules is proper because any proper submodule is graded and must intersect $\cV_{1,p-1}^{(2)}(0)$ in its $L_0$-eigenspace), $\cP_{1,p-1}$ has unique irreducible quotient $\cL_{1,p-1}$. Then because $\cP_{1,p-1}$ is self-dual, it also contains $\cL_{1,p-1}$ as unique irreducible submodule. Since $\Pi_{1,p-1}$ is an isomorphism on degree-$0$ spaces by Proposition \ref{prop:Pi_rs_nontrivial}(3), the submodule $\cL_{1,p-1}$ is generated by the image under $\Pi_{1,p-1}$ of an $L_0$-eigenvector in $\cV_{1,p-1}^{(2)}(0)$. This means that $\ker\Pi_{1,p-1}$ contains the maximal proper submodule of the Verma submodule $\cV_{1,p-1}\subseteq\cV_{1,p-1}^{(2)}$.

 So far, we have shown that there is an exact sequence
 \begin{equation*}
  0\longrightarrow\cL_{1,p-1}\longrightarrow\cP_{1,p-1}\longrightarrow \cV_{1,p-1}/\cJ\longrightarrow 0,
 \end{equation*}
where the submodule $\cJ\subseteq\cV_{1,p-1}$ is a Verma module occurring in the embedding diagram
\begin{equation*}
 \cV_{1,p-1}\longleftarrow \cV_{2,1}\longleftarrow\cV_{3,p-1}\longleftarrow\cV_{4,1}\longleftarrow\cdots
\end{equation*}
Let $\cL_{r,s}$ denote the unique irreducible submodule of $\cV_{1,p-1}/\cJ$ (that is, $\cJ=\cV_{r+1,p-s}$). We have $r\geq 2$ because $\cL_{1,p-1}$ does not admit non-split self-extensions at central charge $c_{p,1}$ \cite[Section 5.4]{GK}. Now let $\cZ_{1,p-1}\subseteq\cP_{1,p-1}$ denote the inverse image of $\cL_{r,s}$ under the surjection $\cP_{1,p-1}\rightarrow\cV_{1,p-1}/\cJ$; thus we have an exact sequence
\begin{equation*}
 0\longrightarrow \cL_{1,p-1}\longrightarrow\cZ_{1,p-1}\longrightarrow\cL_{r,s}\longrightarrow 0.
\end{equation*}
This sequence does not split because $\cL_{1,p-1}$ is the unique irreducible submodule of $\cP_{1,p-1}$, and $r\geq 2$. Applying the exact contragredient functor, we get the non-split sequence
\begin{equation*}
 0\longrightarrow\cL_{r,s}\longrightarrow\cZ_{1,p-1}'\longrightarrow\cL_{1,p-1}\longrightarrow 0.
\end{equation*}
Since $h_{r,s}-h_{1,p-1}\in\ZZ_+$, $\cZ_{1,p-1}'$ contains a singular vector of weight $h_{1,p-1}$, and therefore there is a non-zero homomorphism $\cV_{1,p-1}\rightarrow\cZ_{1,p-1}'$. The image has length at least $2$ (since $\cZ_{1,p-1}'$ does not contain $\cL_{1,p-1}$ as a submodule), and thus $\cZ_{1,p-1}'$ is a homomorphic image of $\cV_{1,p-1}$. The only length-$2$ quotient of $\cV_{1,p-1}$ is $\cV_{1,p-1}/\cV_{3,p-1}$, so $\cZ_{1,p-1}\cong(\cV_{1,p-1}/\cV_{3,p-1})'$ and therefore $(r,s)=(2,1)$.

This verifies the top row in the subquotient diagram for $\cP_{1,p-1}$, and also $\cP_{1,p-1}/\cL_{1,p-1}\cong\cV_{1,p-1}/\cJ$ with $\cJ=\cV_{3,p-1}$. This finishes the proof that $\cP_{1,p-1}$ has the subquotients indicated in the diagram. Now the Loewy diagram is easy: the socle of $\cP_{1,p-1}$ is $\cL_{1,p-1}$ since this is the unique irreducible submodule, and then the socle of $\cP_{1,p-1}/\cL_{1,p-1}\cong\cV_{1,p-1}/\cV_{3,p-1}$ is $\cL_{2,1}$. Moreover, the two extensions $(\cV_{1,p-1}/\cV_{3,p-1})'$ and $\cV_{1,p-1}/\cV_{3,p-1}$ of irreducible subquotients of $\cP_{1,p-1}$ are both indecomposable. Finally, $\cP_{1,p-1}$ itself is indecomposable since the intersection of any two non-zero submodules must contain the unique irreducible submodule $\cL_{1,p-1}$.
\end{proof}

\begin{rem}
 Note that $\cP_{1,p-1}$ is a logarithmic $V_c$-module, with maximum Jordan block size $2$ for $L_0$ beginning in degree $0$.
\end{rem}

Now we handle $r\geq 2$:
\begin{prop}\label{prop:Pr_structure}
 For $r\geq 2$, the tensor product $\cP_{r,p-1}$ is a self-dual indecomposable length-$4$ module with subquotients as indicated in the diagram
 \begin{equation*}
  \xymatrix{
  \cL_{r,p-1} \ar[r] \ar[d] \ar[rd] & (\cV_{r,p-1}/\cV_{r+2,p-1})' \ar[d] \ar[r] & \cL_{r+1,1} \ar[d] \\
  \cV_{r-1,1}/\cV_{r+1,1} \ar[r] \ar[d]  & \cP_{r,p-1} \ar[r] \ar[d] \ar[rd] & \cV_{r,p-1}/\cV_{r+2,p-1} \ar[d] \\
  \cL_{r-1,1} \ar[r] & (\cV_{r-1,1}/\cV_{r+1,1})' \ar[r] & \cL_{r,p-1} \\
  }
 \end{equation*}
and Loewy diagram
\begin{equation*}
 \xymatrixrowsep{1pc}
\xymatrixcolsep{.75pc}
 \xymatrix{
 & & \cL_{r,p-1} & \\
 \cP_{r,p-1}: & \cL_{r-1,1} \ar[ru] & & \cL_{r+1,1} \ar[lu] \\
& & \cL_{r,p-1} \ar[lu] \ar[ru] & \\
 }
\end{equation*}

\end{prop}
\begin{proof}
First, $\cP_{r,p-1}$ is self-dual exactly as in the $r=1$ case. Then from Theorem \ref{thm:Lr_Ls_fusion},
 \begin{equation}\label{eqn:Prp-1}
  \cP_{r,p-1} =\cL_{1,2}\tens\cL_{r,p}\cong\cL_{r,1}\tens(\cL_{1,2}\tens\cL_{1,p})=\cL_{r,1}\tens\cP_{1,p-1}.
 \end{equation}
Thus because $\cL_{r,1}\tens\bullet$ is exact (since $\cL_{r,1}$ is rigid), $\cP_{r,p-1}$ contains submodules $\cL_{r,p-1}\cong\cL_{r,1}\tens\cL_{1,p-1}$ and $\cZ_{r,p-1}\cong\cL_{r,1}\tens(\cV_{1,p-1}/\cV_{3,p-1})'$, and using \eqref{eqn:Lr1_fus_rules}, we have an exact sequence
\begin{equation*}
 0\longrightarrow\cL_{r,p-1}\longrightarrow\cZ_{r,p-1}\longrightarrow\cL_{r-1,1}\oplus\cL_{r+1,1}\longrightarrow 0.
\end{equation*}
Moreover, $\cZ_{r,p-1}$ is a maximal proper submodule of $\cP_{r,p-1}$ because we have an exact sequence
\begin{equation*}
 0\longrightarrow\cZ_{r,p-1}\longrightarrow\cP_{r,p-1}\longrightarrow\cL_{r,p-1}\longrightarrow 0.
\end{equation*}
So $\cL_{r,p-1}$ is both a submodule and quotient of $\cP_{r,p-1}$.

We claim that $\cL_{r\pm1,1}$ are neither submodules nor quotients of $\cP_{r,p-1}$. Indeed, using rigidity,
\begin{align*}
 \hom_{V_c}(\cL_{r\pm1,1},\cP_{r,p-1}) & \cong\hom_{V_c}(\cL_{r\pm1,1},\cL_{r,1}\tens\cP_{1,p-1})\nonumber\\
 &\cong\hom_{V_c}(\cL_{r\pm1,1}\tens\cL_{r,1},\cP_{1,p-1})=0,
\end{align*}
since $\cL_{r\pm1,1}\tens\cL_{r,1}$ is a direct sum of submodules $\cL_{r',1}$ that does not include $\cL_{1,1}$ (by \eqref{eqn:Lr1_fus_rules}) and
since $\cL_{1,p-1}$ is the only irreducible submodule of $\cP_{1,p-1}$. Then since $\cP_{r,p-1}$ is self-dual,
\begin{equation*}
 \hom_{V_c}(\cP_{r,p-1},\cL_{r\pm1,1})=0
\end{equation*}
as well. So if we use $\cX_{r\pm1,1}\subseteq\cZ_{r,p-1}$ to denote the inverse images of $\cL_{r\pm1,1}$ under the surjection $\cZ_{r,p-1}\rightarrow\cL_{r-1,1}\oplus\cL_{r+1,1}$, the exact sequences
\begin{equation*}
 0\longrightarrow\cL_{r,p-1}\longrightarrow\cX_{r\pm1,1}\longrightarrow\cL_{r\pm1,1}\longrightarrow 0
\end{equation*}
do not split. Then using conformal weight considerations as in the $r=1$ case, $\cX_{r+1,1}'$ is a quotient of $\cV_{r,p-1}$ while $\cX_{r-1,1}$ is a quotient of $\cV_{r-1,1}$. Specifically,
\begin{equation*}
 \cX_{r+1,1}\cong(\cV_{r,p-1}/\cV_{r+2,p-1})'
\end{equation*}
and
\begin{equation*}
 \cX_{r-1,1}\cong\cV_{r-1,1}/\cV_{r+1,1},
\end{equation*}
verifying the upper left half of the subquotient diagram for $\cP_{r,p-1}$.

We still need to determine $\cP_{r,p-1}/\cX_{r\pm1,1}$. These quotients appear in the exact sequences
\begin{equation*}
 0\longrightarrow\cZ_{r,p-1}/\cX_{r\pm1,1}\longrightarrow\cP_{r,p-1}/\cX_{r\pm1,1}\longrightarrow\cL_{r,p-1}\longrightarrow 0,
\end{equation*}
with $\cZ_{r,p-1}/\cX_{r\pm1,1}\cong\cL_{r\mp1,1}$. These sequences do not split because $\cL_{r\pm1,1}$ are not quotients of $\cP_{r,p-1}$, so conformal weight considerations as before show that
\begin{equation*}
 \cP_{r,p-1}/\cX_{r+1,1}\cong(\cV_{r-1,1}/\cV_{r+1,1})'
\end{equation*}
and
\begin{equation*}
 \cP_{r,p-1}/\cX_{r-1,1}\cong\cV_{r,p-1}/\cV_{r+2,p-1}.
\end{equation*}
This verifies all subquotients in the diagram for $\cP_{r,p-1}$, and the Loewy diagram also follows easily. In particular, $\mathrm{Soc}(\cP_{r,p-1})\cong\cL_{r,p-1}$ because $\cL_{r\pm1,1}$ are not submodules and $\cL_{r,p-1}$ occurs as a submodule only once (otherwise $\cL_{r\pm1,1}$ would be quotients), and then $\mathrm{Soc}(\cP_{r,p-1}/\cL_{r,p-1})\cong\cL_{r-1,1}\oplus\cL_{r+1,1}$ because again $\cL_{r\pm1,1}$ are not quotients of $\cP_{r,p-1}$. Finally, as in the $r=1$ case, $\cP_{r,p-1}$ is indecomposable because the intersection of any two non-zero submodules must contain the irreducible socle $\cL_{r,p-1}$.
\end{proof}
\begin{rem}
 Proposition \ref{prop:Pr_structure} shows that for $r\geq 2$, the homomorphism $\Pi_{r,p}: \cV_{r,p-1}\oplus\cV_{r,p+1}\rightarrow\cL_{1,2}\tens\cL_{r,p}$ is not surjective: its image is the Verma module quotient $\cV_{r-1,1}/\cV_{r+1,1}$ (note that $\cV_{r-1,1}=\cV_{r,p+1}$).
\end{rem}

We summarize the fusion rules of this section in the following theorem:
\begin{thm}\label{thm:basic_fusion_rules}
 The following fusion rules hold in $\cO_c$:
 \begin{enumerate}
  \item For $r,r'\geq 1$ and $1\leq s\leq p$,
  \begin{equation}\label{fr1}
   \cL_{r',1}\tens\cL_{r,s}\cong\bigoplus_{\substack{k = |r-r'|+1\\ k+r+r' \equiv 1\; ({\rm mod}\; 2)}}^{r+r'-1} \cL_{k,s}.
  \end{equation}

  \item For $r\geq1$ and $1\leq s\leq p$,
  \begin{equation}\label{fr2}
   \cL_{1, 2}\tens\cL_{r, s}\cong\left\lbrace\begin{array}{lll}
                                            \cL_{r,2} & \text{if} & s=1\\
                                            \cL_{r,s-1}\oplus\cL_{r,s+1} & \text{if} & 2\leq s\leq p-1\\
                                            \cP_{r,p-1} & \text{if} & s=p
                                           \end{array}
\right. ,
  \end{equation}
where $\cP_{r,p-1}$ is the indecomposable module described in Propositions \ref{prop:P1_structure} and \ref{prop:Pr_structure}.
 \end{enumerate}
\end{thm}
We will use these formulas to compute all fusion products of irreducible modules later, but we will first need to construct additional indecomposable modules $\cP_{r,s}$ that will appear in the fusion products.

\subsection{Categorical dimensions in \texorpdfstring{$\cO_c$}{Oc}}

Now we can use Proposition \ref{prop:L12_dim} and Theorem \ref{thm:basic_fusion_rules} to compute the categorical dimensions of all irreducible modules in $\cO_c$:
\begin{thm}\label{thm:cat_dim}
 In the ribbon tensor category $\cO_c$,
 \begin{equation}\label{eqn:rs_cat_dim}
  \dim_{\cO_c} \cL_{r,s} =(-1)^{(p+1)(r+1)+s+1}\, r\cdot \frac{\sin(\pi s/p)}{\sin(\pi/p)}
 \end{equation}
for all $r\geq 1$ and $1\leq s\leq p$.
\end{thm}
\begin{proof}
  We have $\dim_{\cO_c}\cL_{1,1}=1$, and Proposition \ref{prop:L12_dim} shows that
 \begin{equation*}
  \dim_{\cO_c}\cL_{1,2}=-2\cos(\pi/p)=-\frac{\sin(2\pi/p)}{\sin(\pi/p)} =-\frac{q^2-q^{-2}}{q-q^{-1}}
 \end{equation*}
where $q=e^{\pi i/p}$. We can now prove by induction on $s$ that $\dim_{\cO_c}\cL_{1,s}=(-1)^{s+1}\frac{\sin(s\pi/p)}{\sin(\pi/p)}$ for $1\leq s\leq p$. Indeed, if this formula holds for $s$, then using the fusion rules \eqref{fr2} and the fact that categorical dimension respects tensor products, we get
\begin{align*}
 \dim_{\cO_c}\cL_{1,s+1} & =(\dim_{\cO_c}\cL_{1,2})(\dim_{\cO_c}\cL_{1,s})-\dim_{\cO_c}\cL_{1,s-1}\nonumber\\
 & =(-1)^{s+2}\frac{(q^2-q^{-2})(q^s-q^{-s})}{(q-q^{-1})^2}-(-1)^{s}\frac{q^{s-1}-q^{-s+1}}{q-q^{-1}}\nonumber\\
 & =\frac{(-1)^{s+2}}{q-q^{-1}}\left((q+q^{-1})(q^s-q^{-s})-q^{s-1}+q^{-s+1}\right) = (-1)^{s+2}\frac{q^{s+1}-q^{-s-1}}{q-q^{-1}},
\end{align*}
as required. From this dimension formula, we can see that $\dim_{\cO_c}\cL_{1,p}=0$.

Next we consider $\cL_{2,1}$. Since this is a composition factor of $\cP_{1,p-1}=\cL_{1,2}\tens\cL_{1,p}$ and since categorical dimension respects extensions,
\begin{align*}
 \dim_{\cO_c}\cL_{2,1} & =\dim_{\cO_{c}}\cP_{1,p-1}-2\dim_{\cO_c}\cL_{1,p-1}\nonumber\\
 & =(\dim_{\cO_c}\cL_{1,2})(\dim_{\cO_c}\cL_{1,p})-2(-1)^p\frac{\sin((p-1)\pi/p)}{\sin(\pi/p)}\nonumber\\
 & =0+(-1)^{p+1}\, 2\cdot\frac{\sin(\pi-\pi/p)}{\sin(\pi/p)} =(-1)^{p+1}\, 2.
\end{align*}
From this, the $\mathfrak{sl}_2$-type fusion rules \eqref{fr1} and induction on $r$ show that
\begin{equation*}
 \dim_{\cO_c}\cL_{r,1} =(-1)^{(p+1)(r+1)}\, r
\end{equation*}
for all $r\geq 1$. Then \eqref{eqn:rs_cat_dim} for general $r$ and $s$ follows from $\cL_{r,s}\cong\cL_{r,1}\tens\cL_{1,s}$.
\end{proof}

\section{Projective modules}\label{sec:proj}

The category $\cO_c$ is quite wild: for example, since all Verma modules $\cV_{r,s}$ have infinite length, each irreducible module $\cL_{r,s}$ has non-split extensions in $\cO_c$ of arbitrary length. This means that no irreducible module $\cL_{r,s}$ has a projective cover in $\cO_c$, and consequently, there is probably no hope of any reasonable classification or description of the indecomposable objects in $\cO_c$. We can remedy this situation somewhat by restricting attention to a tamer tensor subcategory, which we introduce next.

\subsection{The tensor subcategory \texorpdfstring{$\cO_c^0$}{Oc0}}

Recall from Theorem \ref{thm:Lr1_fus_rules} that the modules $\cL_{r,1}$ for $r\geq 1$ are the simple objects of a semisimple tensor subcategory of $\cO_c$ that is braided tensor equivalent to an abelian $3$-cocycle twist of $\rep SU(2)$. Moreover, the modules $\cL_{2n+1,1}$ for $n\in\NN$ are the simple objects of a semisimple symmetric tensor subcategory that is equivalent to $\rep SO(3,\RR)$. These are the irreducible $V_c$-modules that appear in the decomposition of the triplet vertex operator algebra $\cW(p)$ as a $V_c$-module: specifically,
\begin{equation*}
 \cW(p)\cong\bigoplus_{n=0}^\infty (2n+1)\cdot\cL_{2n+1,0}.
\end{equation*}
Because the subcategory $\rep SO(3,\RR)$ of $\cO_c$ is symmetric, monodromies satisfy
\begin{equation*}
 \cR_{\cL_{2n'+1,1},\cL_{2n+1,1}}\circ\cR_{\cL_{2n+1,1},\cL_{2n'+1,1}} =\Id_{\cL_{2n+1,1}\tens\cL_{2n'+1,1}}
\end{equation*}
for all $n,n'\in\NN$, that is, the modules $\cL_{2n+1,1}$ and $\cL_{2n'+1,1}$ centralize each other. We define the subcategory $\cO_c^0\subseteq\cO_c$ to consist of all modules that centralize the $\cL_{2n+1,1}$:
\begin{defi}\label{def:oc0}
 The category $\cO_c^0$ is the M\"{u}ger centralizer of $\rep SO(3,\RR)$ in $\cO_c$, that is, $\cO_c^0\subseteq\cO_c$ is the full subcategory whose objects $\cW$ satisfy
 \begin{equation*}
  \cR_{\cL_{2n+1,1},\cW}\circ\cR_{\cW,\cL_{2n+1,1}} =\Id_{\cW\tens\cL_{2n+1,1}}
 \end{equation*}
for all $n\in\NN$.
\end{defi}

The next result establishes the fundamental properties of $\cO_c^0$:
\begin{prop}
 The category $\cO_c^0$ is a ribbon tensor subcategory of $\cO_c$ that contains all irreducible $V_c$-modules $\cL_{r,s}$ for $r\geq1$, $1\leq s\leq p$.
\end{prop}
\begin{proof}
 To show that $\cO_c^0$ is a monoidal subcategory of $\cO_c$, we just need to show that if $\cW_1$ and $\cW_2$ are modules in $\cO_c^0$, then so is $\cW_1\tens\cW_2$, that is, $$\cR^2_{\cW_1\tens\cW_2,\cL_{2n+1,1}}=\Id_{(\cW_1\tens\cW_2)\tens\cL_{2n+1,1}}$$
 for all $n\in\NN$. But this is straightforward from the hexagon axiom for the braiding $\cR$. Then to show that $\cO_c^0$ is abelian and thus a tensor subcategory of $\cO_c$, it is enough to show that $\cO_c^0$ is closed under submodules and quotient modules. This follows from the rigidity of $\cL_{2n+1,1}$ and corresponding exactness of $\cL_{2n+1,1}\tens\bullet$, as well as the naturality of the braiding in $\cO_c$.

 To show that $\cO_c^0$ is rigid and thus a ribbon subcategory of $\cO_c$, we just need to show closure under contragredients, that is, if $\cR^2_{\cW,\cL_{2n+1,1}}$ is the identity for each $n\in\NN$, then so is $\cR^2_{\cW',\cL_{2n+1,1}}$. Since any such $\cW$ is rigid (in $\cO_c$) by Theorem \ref{thm:rigidity}, we can use \cite[Lemma 8.9.1]{EGNO}, which states that $\cR_{\cW',\cL_{2n+1,1}}$ agrees with the composition
 \begin{align*}
  \cW'\tens & \cL_{2n+1,1} \xrightarrow{r^{-1}} (\cW'\tens\cL_{2n+1,1})\tens V_c\xrightarrow{\Id\tens i_\cW} (\cW'\tens\cL_{2n+1,1})\tens(\cW\tens\cW')\nonumber\\
  & \xrightarrow{assoc.} \cW'\tens((\cL_{2n+1,1}\tens\cW)\tens\cW') \xrightarrow{\Id\tens(\cR_{\cW,\cL_{2n+1,1}}^{-1}\tens\Id)} \cW'\tens((\cW\tens\cL_{2n+1,1})\tens\cW')\nonumber\\
  & \xrightarrow{assoc.} (\cW'\tens\cW)\tens(\cL_{2n+1,1}\tens\cW')\xrightarrow{e_\cW\tens\Id} V_c\tens(\cL_{2n+1,1}\tens\cW')\xrightarrow{l} \cL_{2n+1,1}\tens\cW',
 \end{align*}
where the arrows marked $assoc.$ represent compositions of associativity isomorphisms in $\cO_c$. Using the opposite braiding, $\cR_{\cL_{2n+1,1},\cW'}^{-1}$ is the identical composition, except that $\cR_{\cW,\cL_{2n+1,1}}^{-1}$ is replaced with $\cR_{\cL_{2n+1,1},\cW}$. But $\cR_{\cW,\cL_{2n+1,1}}^{-1}=\cR_{\cL_{2n+1,1},\cW}$ since $\cW$ is an object of $\cO_c^0$, so the compositions giving $\cR_{\cW',\cL_{2n+1,1}}$ and $\cR_{\cL_{2n+1,1},\cW'}^{-1}$ are the same. Therefore the monodromy of $\cW'$ with each $\cL_{2n+1,1}$ is the identity.

Finally, to show that each $\cL_{r,s}$ is an object of $\cO_c^0$, we use the balancing equation for monodromies:
\begin{align*}
 \cR^2_{\cL_{r,s},\cL_{2n+1,1}} =\theta_{\cL_{r,s}\tens\cL_{2n+1,1}}\circ(\theta_{\cL_{r,s}}^{-1}\tens\theta_{\cL_{2n+1,1}}^{-1}).
\end{align*}
Recall that $\theta=e^{2\pi i L_0}$ and that $$\cL_{r,s}\tens\cL_{2n+1,1}\cong\bigoplus_{\substack{k = |r-2n-1|+1\\ k+r+2n \equiv 0\; ({\rm mod}\; 2)}}^{r+2n} \cL_{k,s} =\bigoplus_{k=1}^{\min(r,2n+1)} \cL_{r+2(n-k+1),s}$$
(from Theorem \ref{thm:basic_fusion_rules}). Thus on the $\cL_{r+2(n-k+1),s}$ summand of $\cL_{r,s}\tens\cL_{2n+1,1}$, the monodromy is given by the scalar
\begin{equation*}
 e^{2\pi i(h_{r+2(n-k+1),s}-h_{r,s}-h_{2n+1,1})} =e^{2\pi i\left[(pr-s)(n-k+1)+(k-1)^2p-(2k-1)np+n\right]} =1.
\end{equation*}
Thus $\cR^2_{\cL_{r,s},\cL_{2n+1,1}}=\Id_{\cL_{r,s}\tens\cL_{2n+1,1}}$ for all $n\in\NN$ as required.
\end{proof}

\begin{rem}
 Although $\cO_c^0$ is closed under submodules, quotients, and contragredients, it need not be closed under arbitrary (non-split) extensions. Thus it is possible for a module to be projective in the subcategory $\cO_c^0$ even if it is not projective in $\cO_c$. In fact, we will show that every irreducible module $\cL_{r,s}$ has a projective cover in $\cO_c^0$, although not in $\cO_c$.
\end{rem}

We now begin to obtain projective objects in $\cO_c^0$:
\begin{thm}\label{projoflrp}
 For all $r\geq 1$, the module $\cL_{r,p}$ is both projective and injective in $\cO_c^0$.
\end{thm}
\begin{proof}
 Since $\cL_{r,p}$ is self-dual, injectivity of $\cL_{r,p}$ will follow from projectivity. Moreover, it is enough to show that $\cL_{1,p}$ is projective because $\cL_{r,p}\cong\cL_{r,1}\tens\cL_{1,p}$ from Theorem \ref{thm:Lr_Ls_fusion} (recall that projective objects form a tensor ideal in any rigid tensor category).

 Now because $\cL_{1,p}$ is simple, it will be projective in $\cO_c^0$ if all surjections $\cW\twoheadrightarrow\cL_{1,p}$ with $\cW$ an object of $\cO_c^0$ split. In fact, because all modules in $\cO_c^0$ have finite length, we may assume that $\cW$ has length $2$. (If all length-$2$ extensions of $\cL_{1,p}$ split, then so do all finite-length extensions by induction on length.) Thus we are reduced to considering extensions
 \begin{equation}\label{eqn:proj_exact_seq}
  0\longrightarrow\cL_{r,s}\longrightarrow\cW\longrightarrow\cL_{1,p}\longrightarrow 0.
 \end{equation}
It is easy to see that $h_{1,p}$ is the minimum of all conformal weights $h_{r,s}$ at central charge $c_{p,1}$, so because $\cL_{1,p}$ does not admit non-split self-extensions (see \cite[Section 5.4]{GK}), we may assume $h_{r,s}>h_{1,p}$. This means that $\cW$ contains a singular vector of conformal weight $h_{1,p}$, and thus $\cW$ contains a homomorphic image of the Verma module $\cV_{1,p}$.

If the image of $\cV_{1,p}$ in $\cW$ has length $1$, the exact sequence \eqref{eqn:proj_exact_seq} splits, so we may assume the length is $2$. In this case, the structure of $\cV_{1,p}$ as a $\cV ir$-module shows that $\cW\cong \cV_{1,p}/\cV_{5,p}$ and $(r,s)=(3,p)$. Thus we just need to show that $\cV_{1,p}/\cV_{5,p}$ is not an object of $\cO_c^0$, and for this it is sufficient to show that the monodromy $\cR_{\cL_{3,1},\cV_{1,p}/\cV_{5,p}}^2$ is non-trivial. From the balancing equation
\begin{equation*}
 \cR_{\cL_{3,1},\cV_{1,p}/\cV_{5,p}}^2=\theta_{\cL_{3,1}\tens(\cV_{1,p}/\cV_{5,p})}\circ(\theta_{\cL_{3,1}}^{-1}\tens\theta_{\cV_{1,p}/\cV_{5,p}}^{-1}) = e^{2\pi i(L_0-h_{3,1}-h_{1,p})},
\end{equation*}
it is enough to show that $\cL_{3,1}\tens(\cV_{1,p}/\cV_{5,p})$ is a logarithmic $V_c$-module, that is, $L_0$ acts non-semisimply on the tensor product. To show this, we prove that $\cL_{3,1}\tens(\cV_{1,p}/\cV_{5,p})$ surjects onto a logarithmic self-extension of $\cL_{3,p}$.

First, the exactness of $\cL_{3,1}\tens\bullet$ and the fusion rules of Theorem \ref{thm:basic_fusion_rules} imply there is an exact sequence
\begin{equation*}
 0\longrightarrow\cL_{1,p}\oplus\cL_{3,p}\oplus\cL_{5,p}\longrightarrow\cL_{3,1}\tens(\cV_{1,p}/\cV_{5,p})\longrightarrow\cL_{3,p}\longrightarrow 0.
\end{equation*}
We quotient out the submodule $\cL_{1,p}\oplus\cL_{5,p}$ from the tensor product to get a surjection
\begin{equation*}
 f:\cL_{3,1}\tens(\cV_{1,p}/\cV_{5,p})\longrightarrow\cL_{3,p}^{(2)},
\end{equation*}
where $\cL_{3,p}^{(2)}$ is some self-extension of $\cL_{3,p}$. We want to show that $L_0$ acts non-semisimply on $\cL_{3,p}^{(2)}(0)$; this $A(V_c)$-module is $2$-dimensional and $h_{3,p}$ is its only $L_0$-eigenvalue.

The intertwining operator $\cY=f\circ\cY_\tens$ of type $\binom{\cL_{3,p}^{(2)}}{\cL_{3,1}\,\cV_{1,p}/\cV_{5,p}}$ is surjective because $f$ and $\cY_\tens$ are surjective. Then Proposition \ref{prop:piY_surjective} implies that
\begin{equation*}
 \pi(\cY): A(\cL_{3,1})\otimes_{A(V_c)} (\cV_{1,p}/\cV_{5,p})(0)\longrightarrow\cL_{3,p}^{(2)}(0)
\end{equation*}
is surjective, so $\cL_{3,p}^{(2)}(0)$ is a homomorphic image of $A(\cL_{3,1})\otimes_{A(V_c)}\CC v_{1,p}$. This latter $A(V_c)$-module was determined in \cite{FZ2} (under the unnecessary assumption that $p\notin\QQ$): we now review the computation.

 The computation of the $A(V_c)\cong\CC[x]$-bimodule $A(\cL_{3,1})$ is similar to the computation of $A(\cL_{1,2})$ from Section \ref{sec:first_fus}. Recall there is an isomorphism
\begin{align*}
 \varphi: A(\cV_{3,1}) & \rightarrow \CC[x,y]\\
 [\omega]^m\cdot[v_{3,1}]\cdot[\omega]^n & \mapsto x^m y^n,
\end{align*}
and that
\begin{equation*}
 A(\cL_{3,1})\cong\CC[x,y]/(f_{3,1}(x,y))
\end{equation*}
where $f_{3,1}(x,y)=\varphi([\til{v}])$ for a singular vector $\til{v}\in\cV_{3,1}$ generating the maximal proper submodule. We can take
\begin{equation*}
\widetilde{v}=\left( L_{-1}^3-4p L_{-1}L_{-2}+2p(2p+1)L_{-3}\right)v_{3,1}.
\end{equation*}
Then to compute $\varphi([\til{v}])$, first note that \eqref{eqn:bimod_reln} implies
\begin{equation*}
 \varphi([L_{-1} v])=(x-y-\mathrm{wt}\,v)\varphi([v])
\end{equation*}
for $v\in\cV_{3,1}$, while \eqref{eqn:bimod_reln_2} implies
\begin{equation*}
 \varphi([L_{-2}v])=y\varphi([v])-\varphi([L_{-1}v]) =(2y-x+\mathrm{wt}\,v)\varphi([v]).
\end{equation*}
Then the relation
\begin{equation*}
 [L_{-n} v] =(-1)^n[(n-1)L_{-2}v+(n-2)L_{-1}v]
\end{equation*}
in $A(\cV_{3,1})$ specialized to $n=3$ (see the proof of \cite[Lemma 2.11]{FZ2}) implies
\begin{align*}
 \varphi([L_{-3}v]) & =-2\varphi([L_{-2}v])-\varphi([L_{-1} v])=(x-3y-\mathrm{wt}\,v)\varphi([v]).
\end{align*}
Using these formulas, we get
\begin{align*}
 f_{3,1}(x,y)&= (x-y-h_{3,1}-2)(x-y-h_{3,1}-1)(x-y-h_{3,1})\nonumber\\
 &\hspace{3em}-4p(x-y-h_{3,1}-2)(2y-x+h_{3,1})+2p(2p+1)(x-3y-h_{3,1})\nonumber\\
 & =(x-y)\left((x-y-2p+1)(x-y-1)-4p\,y\right)
\end{align*}
(see \cite[Example 2.12]{FZ2}).

We now have
\begin{align*}
 A(\cL_{3,1})\otimes_{A(V_c)} \CC v_{1,p} &\cong\CC[x]/(f_{3,1}(x,h_{1,p})),
\end{align*}
and it turns out that
\begin{align*}
 f_{3,1}(x,h_{1,p})&=(x-h_{1,p})(x-h_{3,p})^2.
\end{align*}
Thus $L_0$ acts non-semisimply on the only $2$-dimensional quotient of $A(\cL_{3,1})\otimes_{A(V_c)}\CC v_{1,p}$ whose only $L_0$-eigenvalue is $h_{3,p}$. So $\cL_{3,p}^{(2)}$ is a logarithmic module in $\cO_c$, proving that $\cV_{1,p}/\cV_{5,p}$ is not an object of $\cO_c^0$. This completes the proof that $\cL_{1,p}$ is projective in $\cO_c^0$.
\end{proof}

As the modules $\cL_{r,p}$ are irreducible, they are their own projective covers in $\cO_c^0$. For this reason, we will sometimes use the alternate notation $\cL_{r,p}=\cP_{r,p}$ for $r\geq 1$. The irreducible modules $\cL_{r,p-1}$ also have projective covers in $\cO_c^0$:
\begin{prop}\label{prop:Prp-1_proj_cover}
For $r\geq 1$, the module $\cP_{r,p-1}$ is a projective cover of $\cL_{r,p-1}$ in $\cO_c^0$.
\end{prop}
\begin{proof}
 The module $\cP_{r,p-1}$ is projective in $\cO_c^0$ because it is by definition the tensor product of a rigid with a projective module. From Propositions \ref{prop:P1_structure} and \ref{prop:Pr_structure}, there is also a surjective homomorphism $q: \cP_{r,p-1}\rightarrow\cL_{r,p-1}$.

 Now let $\cP$ be any projective module in $\cO_c^0$ with surjective homomorphism $\til{q}: \cP\rightarrow\cL_{r,p-1}$. Because both $\cP$ and $\cP_{r,p-1}$ are projective, there are homomorphisms $f: \cP\rightarrow\cP_{r,p-1}$ and $g:\cP_{r,p-1}\rightarrow\cP$ such that the diagrams
 \begin{equation*}
  \xymatrix{
  & \cP \ar[ld]_{f} \ar[d]^{\til{q}} \\
  \cP_{r,p-1} \ar[r]_q & \cL_{r,p-1} \\
  } \qquad \qquad
  \xymatrix{
  & \cP_{r,p-1} \ar[ld]_{g} \ar[d]^{q} \\
  \cP \ar[r]_(.4){\til{q}} & \cL_{r,p-1} \\
  }
 \end{equation*}
commute; we need to show that $f$ is surjective. Indeed, $f\circ g$, as an endomorphism of a finite-length indecomposable module, is either nilpotent or an isomorphism by Fitting's Lemma, and it cannot be nilpotent because for all $N\in\NN$,
\begin{equation*}
 q\circ(f\circ g)^N = q\neq 0.
\end{equation*}
Therefore $f\circ g$ is an isomorphism, which means $f$ is surjective (and $g$ is injective).
\end{proof}

\subsection{The remaining projective covers}\label{subsec:more_proj_covers}

For $p=2$, we have shown that every irreducible module has a projective cover in $\cO_c^0$. For $p\geq 3$, we now construct projective covers of the remaining irreducible modules $\cL_{r,s}$, $s\leq p-2$, using the method of \cite[Section 5.1]{CMY2}. In fact, many of the arguments from \cite{CMY2} go through almost verbatim in this context.

\subsubsection{The case \texorpdfstring{$r = 1$}{r=1}}
From Proposition \ref{prop:P1_structure}, the maximal submodule $\cZ_{1,p-1}$ of the projective module $\cP_{1,p-1}$ is isomorphic to $(\cV_{1,p-1}/\cV_{3,p-1})'$, and there is an exact sequence
\begin{equation}\label{z1p-1}
0 \longrightarrow \cL_{1,p-1} \longrightarrow \cZ_{1,p-1} \longrightarrow \cL_{2,1} \longrightarrow 0.
\end{equation}
Since $\cL_{1,2}$ is rigid, the functor $\cL_{1,2}\boxtimes \bullet$ is exact. Applying $\cL_{1,2} \boxtimes \bullet$ to \eqref{z1p-1} and using the fusion rules \eqref{fr2}, we get an exact sequence
\begin{equation*}
0 \longrightarrow \cL_{1, p-2}\oplus \cL_{1,p} \longrightarrow \cL_{1,2} \boxtimes \cZ_{1,p-1} \longrightarrow \cL_{2,2} \longrightarrow 0.
\end{equation*}
Because $\cL_{1,p}$ is injective in $\cO_{c}^0$, it is a direct summand of $\cL_{1,2} \boxtimes \cZ_{1,p-1}$. Let $\cZ_{1, p-2}$ be a submodule complement of $\cL_{1,p}$ in $\cL_{1,2} \boxtimes \cZ_{1,p-1}$, that is,
\begin{equation}\label{decz1p-1}
\cL_{1,2} \boxtimes \cZ_{1,p-1} = \cL_{1,p} \oplus \cZ_{1, p-2}.
\end{equation}
It is easy to see that there is an exact sequence
\begin{equation}\label{z1p-2}
0 \longrightarrow \cL_{1,p-2} \longrightarrow \cZ_{1, p-2} \longrightarrow \cL_{2,2} \longrightarrow 0.
\end{equation}
We claim that this exact sequence does not split. Indeed, the rigidity of $\cL_{1,2}$, the fusion rules \eqref{fr2}, and the Loewy diagram of $\cZ_{1,p-1}$ imply
\begin{align*}
\hom (\cL_{2,2}, \cL_{1,2} \boxtimes \cZ_{1,p-1}) &\cong \hom(\cL_{1,2} \boxtimes \cL_{2,2}, \cZ_{1,p-1})\\
&\cong \hom(\cL_{2,1}\oplus \cL_{2,3}, \cZ_{1,p-1}) = 0.
\end{align*}
So $\cL_{2,2}$ cannot be a submodule of $\cZ_{1,p-2}\subseteq\cL_{1,2}\tens\cZ_{1,p-1}$. Note that the non-splitting of \eqref{z1p-2} together with conformal weight considerations show that $\cZ_{1,p-2}\cong(\cV_{1,p-2}/\cV_{3,p-2})'$, just as in the proof of Proposition \ref{prop:P1_structure}.

Now we apply $\cL_{1,2}\boxtimes \bullet$ to the exact sequence
\[
0 \longrightarrow \cZ_{1,p-1} \longrightarrow \cP_{1,p-1} \longrightarrow \cL_{1,p-1} \longrightarrow 0.
\]
Using \eqref{fr2} and the decomposition \eqref{decz1p-1}, we get the exact sequence
\begin{equation*}
0 \longrightarrow \cL_{1,p} \oplus \cZ_{1,p-2} \longrightarrow \cL_{1,2}\boxtimes \cP_{1,p-1} \longrightarrow \cL_{1,p-2}\oplus \cL_{1,p} \longrightarrow 0.
\end{equation*}
Because $\cL_{1,p}$ is both projective and injective in $\cO_c^0$, $2\cdot \cL_{1,p}$ is a direct summand of $\cL_{1,2}\boxtimes \cP_{1,p-1}$. Defining $\cP_{1,p-2}$ to be a direct summand of $\cL_{1,2} \boxtimes \cP_{1,p-1}$ complementary to $2\cdot \cL_{1,p}$, we get an exact sequence
\begin{equation}\label{p1p-2}
0 \longrightarrow \cZ_{1,p-2} \longrightarrow \cP_{1,p-2} \longrightarrow \cL_{1,p-2} \longrightarrow 0.
\end{equation}
We claim that ${\rm Soc}(\cP_{1,p-2}) = \cL_{1, p-2}$. Indeed, \eqref{z1p-2} and \eqref{p1p-2} show that the composition factors of $\cP_{1,p-2}$ are $\cL_{1,p-2}$, $\cL_{1,p-2}$, and $\cL_{2,2}$. We have already seen that $\cL_{2,2}$ is not a submodule of $\cP_{1,p-2}$, while
\begin{align*}
\dim\hom(\cL_{1,p-2},\cP_{1,p-2}) & =\dim\hom(\cL_{1,p-2},\cL_{1,2}\tens\cP_{1,p-1})\nonumber\\
&=\dim\hom(\cL_{1,2}\tens\cL_{1,p-2},\cP_{1,p-1})\nonumber\\ &=\dim\hom(\cL_{1,p-3}\oplus\cL_{1,p-1},\cP_{1,p-1}) =1,
\end{align*}
proving the claim.

Next, the exact sequences \eqref{p1p-2} and \eqref{z1p-2} give
\begin{equation}\label{qp1p-2}
0 \longrightarrow \cL_{2,2} \longrightarrow \cP_{1,p-2}/\cL_{1,p-2} \longrightarrow \cL_{1,p-2} \longrightarrow 0.
\end{equation}
We claim this sequence does not split and thus ${\rm Soc}(\cP_{1,p-2}/\cL_{1,p-2}) = \cL_{2,2}$. Otherwise, we would have $\cP_{1,p-2}/\cL_{1,p-2} \cong \cL_{1,p-2} \oplus \cL_{2,2}$; using the rigidity of $\cL_{1,2}$, this would imply
\begin{align*}
\hom(\cP_{1,p-1}/\cL_{1,p-1}, \cL_{1,2}\boxtimes \cL_{2,2}) &\cong \hom(\cL_{1,2}\boxtimes(\cP_{1,p-1}/\cL_{1,p-1}), \cL_{2,2})\\
& \cong \hom((\cL_{1,2}\boxtimes \cP_{1,p-1})/(\cL_{1,2}\boxtimes\cL_{1,p-1}), \cL_{2,2})\\
& \cong \hom((\cP_{1,p-2}/\cL_{1,p-2})\oplus \cL_{1,p}, \cL_{2,2}) \neq 0,
\end{align*}
whereas in fact the Loewy diagram of $\cP_{1,p-1}$ shows
\begin{align*}
\hom(\cP_{1,p-1}/\cL_{1,p-1}, \cL_{1,2}\boxtimes \cL_{2,2}) &\cong \hom(\cV_{1,p-1}/\cV_{3,p-1}, \cL_{2,1}\oplus \cL_{2,3}) = 0.
\end{align*}
This proves the claim; note that because \eqref{qp1p-2} does not split, $\cP_{1,p-2}/\cL_{1,p-2}\cong\cV_{1,p-2}/\cV_{3,p-2}$, just as in the proof of Proposition \ref{prop:P1_structure}.

We have now derived the Loewy diagram for $\cP_{1,p-2}$ stated in the next proposition. Moreover, $\cP_{1,p-2}$ is indecomposable for the same reasons as $\cP_{1,p-1}$, and $\cP_{1,p-2}$ is projective in $\cO_c^0$ because it is a direct summand of the projective tensor product $\cL_{1,2}\tens\cP_{1,p-1}$. Thus the argument of Proposition \ref{prop:Prp-1_proj_cover} shows that $\cP_{1,p-2}$ is a projective cover of $\cL_{1,p-2}$:
\begin{prop}
The module $\cP_{1,p-2}$ is indecomposable and a projective cover of $\cL_{1,p-2}$ in $\cO_c^0$. It has Loewy diagram
\begin{equation*}
\xymatrixrowsep{1pc}
\xymatrixcolsep{.75pc}
 \xymatrix{
 & & \cL_{1,p-2} \\
 \cP_{1,p-2}: & \cL_{2,2} \ar[ru] & \\
& & \cL_{1,p-2} \ar[lu]\\
 }
\end{equation*}
\end{prop}

Now that we have projective covers $\cP_{1, p-2}$, $\cP_{1,p-1}$, and $\cP_{1,p}$, we proceed to construct modules $\cP_{1,s}$ for $1 \leq s \leq p-3$ recursively (assuming now that $p\geq 4$). Fix $s \in \{1,2, \dots, p-3\}$ and assume we have $\cP_{1,\sigma}$ for all $s+1 \leq \sigma \leq p-1$ such that:
\begin{itemize}
\item The module $\cP_{1, \sigma}$ is a projective cover of $\cL_{1,\sigma}$ in $\cO_c^0$.
\item The Loewy diagram of $\cP_{1, \sigma}$ is
\begin{equation*}
\xymatrixrowsep{1pc}
\xymatrixcolsep{.75pc}
 \xymatrix{
 & & \cL_{1,\sigma} \\
 \cP_{1,\sigma}: & \cL_{2,p-\sigma} \ar[ru] & \\
& & \cL_{1,\sigma} \ar[lu]\\
 }
\end{equation*}
\end{itemize}
We now define $\cP_{1,s}$ as follows. We have a surjection
\[
\cL_{1,2}\boxtimes \cP_{1,s+1} \longrightarrow \cL_{1,2}\boxtimes \cL_{1,s+1} \cong \cL_{1,s}\oplus \cL_{1,s+2} \longrightarrow \cL_{1,s+2}.
\]
Because $\cL_{1,2}$ is rigid and $\cP_{1,s+1}$ is projective, $\cL_{1,2}\boxtimes \cP_{1,s+1}$ is also projective. So because $\cP_{1,s+2}$ is the projective cover of $\cL_{1,s+2}$, we get a surjective map
\[
\cL_{1,2}\boxtimes \cP_{1,s+1} \longrightarrow \cP_{1,s+2}.
\]
Since $\cP_{1,s+2}$ is projective, this surjection splits and $\cP_{1,s+2}$ is a direct summand of $\cL_{1,2}\boxtimes \cP_{1,s+1}$. Define $\cP_{1,s}$ to be a complement of $\cP_{1,s+2}$:
\[
\cL_{1,2}\boxtimes \cP_{1,s+1} = \cP_{1,s}\oplus \cP_{1,s+2}.
\]
The module $\cP_{1,s}$ is in $\cO_c^0$ because this category is a tensor subcategory of $\cO_c$, and it is projective in $\cO_c^0$ because it is a summand of a projective module. We can now prove:
\begin{thm}\label{thm:P1s_structure}
The module $\cP_{1,s}$ is a projective cover of $\cL_{1,s}$ in $\cO_c^0$ with Loewy diagram
\begin{equation*}
\xymatrixrowsep{1pc}
\xymatrixcolsep{.75pc}
 \xymatrix{
 & & \cL_{1,s} \\
 \cP_{1,s}: & \cL_{2,p-s} \ar[ru] & \\
& & \cL_{1,s} \ar[lu]\\
 }
\end{equation*}
\end{thm}
\begin{proof}
From its Loewy diagram, $\cP_{1,s+1}$ has a maximal proper submodule $\cZ_{1,s+1}$ with non-split exact sequence
\begin{equation}\label{z1s}
0 \longrightarrow \cL_{1,s+1} \longrightarrow \cZ_{1,s+1} \longrightarrow \cL_{2,p-s-1} \longrightarrow 0.
\end{equation}
We apply $\cL_{1,2}\boxtimes \bullet$ to \eqref{z1s} and use the fusion rules \eqref{fr2} to get an exact sequence
\begin{equation*}
0 \longrightarrow \cL_{1,s}\oplus \cL_{1,s+2} \longrightarrow \cL_{1,2}\boxtimes \cZ_{1,s+1} \longrightarrow \cL_{2,p-s}\oplus \cL_{2, p-s-2} \longrightarrow 0.
\end{equation*}
This shows that the conformal weights of $\cL_{1,2}\boxtimes \cZ_{1,s+1}$ are contained in the two distinct cosets $h_{1,s}+\ZZ$ and $h_{1,s+2}+\ZZ$, and thus $\cL_{1,2}\boxtimes \cZ_{1,s+1}$ decomposes as a direct sum of two modules, say $\cZ_{1,s}$ and $\widetilde{\cZ}_{1,s+2}$, with exact sequences
\begin{equation}\label{z1s-1}
0 \longrightarrow \cL_{1,s} \longrightarrow \cZ_{1,s} \longrightarrow \cL_{2, p-s} \longrightarrow 0
\end{equation}
and
\[
0 \longrightarrow \cL_{1,s+2} \longrightarrow \widetilde{\cZ}_{1,s+2} \longrightarrow \cL_{2, p-s-2} \longrightarrow 0.
\]
We claim that \eqref{z1s-1} does not split. Otherwise, $\cL_{2,p-s}$ is a submodule of $\cZ_{1,s}$, and thus also a submodule of $\cL_{1,2}\boxtimes \cZ_{1,s+1}$. Rigidity of $\cL_{1,2}$ would then imply
\[
\hom(\cL_{1,2}\boxtimes \cL_{2,p-s}, \cZ_{1,s+1}) \cong \hom(\cL_{2,p-s}, \cL_{1,2}\boxtimes \cZ_{1,s+1}) \neq 0.
\]
However by the fusion rules \eqref{fr2} and the non-split exact sequence \eqref{z1s} for $\cZ_{1,s+1}$, there is no non-zero homomorphism
\[
\cL_{1,2}\boxtimes \cL_{2,p-s} \cong \cL_{2,p-s-1} \oplus \cL_{2,p-s+1} \longrightarrow \cZ_{1,s+1}.
\]
As a result, ${\rm Soc}(\cZ_{1,s}) = \cL_{1,s}$.

Now we apply $\cL_{1,2}\boxtimes \bullet$ to the exact sequence
\[
0 \longrightarrow \cZ_{1,s+1} \longrightarrow \cP_{1,s+1} \longrightarrow \cL_{1,s+1} \longrightarrow 0
\]
and get an exact sequence
\[
0 \longrightarrow \cZ_{1,s} \oplus \widetilde{\cZ}_{1,s+2} \longrightarrow \cL_{1,2}\boxtimes \cP_{1,s+1} \longrightarrow \cL_{1,s} \oplus \cL_{1,s+2} \longrightarrow 0.
\]
Conformal weight considerations again show that $\cP_{1,s}$ satisfies the exact sequence
\begin{equation}\label{p1s-1}
0 \longrightarrow \cZ_{1,s} \longrightarrow \cP_{1,s} \longrightarrow \cL_{1,s} \longrightarrow 0.
\end{equation}
We claim that ${\rm Soc}(\cP_{1,s}) = \cL_{1,s}$. Otherwise, since \eqref{z1s-1} is non-split, we would have ${\rm Soc}(\cP_{1,s}) = 2\cdot\cL_{1,s}$, and rigidity of $\cL_{1,2}$ would imply
\[
\dim \hom(\cL_{1,2}\boxtimes \cL_{1,s}, \cP_{1,s+1}) = \dim \hom(\cL_{1,s}, \cL_{1,2}\boxtimes \cP_{1,s+1}) = 2.
\]
However, this would conflict with
\[
\dim \hom(\cL_{1,2}\boxtimes \cL_{1,s}, \cP_{1,s+1}) = \dim \hom([\cL_{1,s-1} \oplus] \cL_{1,s+1}, \cP_{1,s+1}) = 1,
\]
where the summand in the brackets occurs for $s \geq 2$.

The exact sequences  \eqref{z1s-1} and \eqref{p1s-1} give an exact sequence
\begin{equation}\label{qp1s-1}
0 \longrightarrow \cL_{2,p-s} \longrightarrow \cP_{1,s}/\cL_{1,s} \longrightarrow \cL_{1,s} \longrightarrow 0.
\end{equation}
We claim that \eqref{qp1s-1} does not split and thus ${\rm Soc}(\cP_{1,s}/\cL_{1,s}) = \cL_{2,p-s}$. Otherwise, we would have $\cP_{1,s}/\cL_{1,s} = \cL_{2,p-s} \oplus \cL_{1,s}$,
and rigidity of $\cL_{1,2}$ would imply
\begin{align*}
\hom(\cP_{1,s+1}/\cL_{1,s+1}, \cL_{1,2}\boxtimes \cL_{2,p-s}) &\cong \hom((\cL_{1,2}\boxtimes \cP_{1,s+1})/(\cL_{1,2}\boxtimes \cL_{1,s+1}), \cL_{2,p-s})\\
&\cong \hom((\cP_{1,s+2}/\cL_{1,s+2}) \oplus (\cP_{1,s}/\cL_{1,s}), \cL_{2,p-s}) \neq 0.
\end{align*}
However, in fact
\begin{align*}
\hom(\cP_{1,s+1}/\cL_{1,s+1}, \cL_{1,2}\boxtimes \cL_{2,p-s}) & \cong \hom(\cP_{1,s+1}/\cL_{1,s+1}, \cL_{2,p-s-1}\oplus \cL_{2,p-s+1}) = 0.
\end{align*}
Thus $\cP_{1,s}/\cL_{1,s}$ is indecomposable, and we have verified the Loewy diagram for $\cP_{1,s}$.

Finally, $\cP_{1,s}$ is a projective cover of $\cL_{1,s}$ in $\cO_c^0$ by the argument of Proposition \ref{prop:Prp-1_proj_cover}.
\end{proof}

\begin{rem}
 As in the proof of Proposition \ref{prop:P1_structure}, we have $\cP_{1,s}/\cL_{1,s}\cong \cV_{1,s}/\cV_{3,s}$ and $\cZ_{1,s}\cong(\cV_{1,s}/\cV_{3,s})'$.
\end{rem}

\subsubsection{The case \texorpdfstring{$r \geq 2$}{r>=2}} We can construct the projective cover $\cP_{r,s}$ for $r\geq 2$, $1 \leq s \leq p-2$ exactly as in \cite[Section~5]{CMY2}. Alternatively, we can simply define $\cP_{r,s}=\cL_{r,1}\tens\cP_{1,s}$.
\begin{thm}\label{thm:Prs_structure}
 For $r\geq 2$ and $1\leq s\leq p-2$, the module $\cP_{r,s}=\cL_{r,1}\tens\cP_{1,s}$ is a projective cover of $\cL_{r,s}$ in $\cO_c^0$. It has Loewy diagram
 \begin{equation*}
 \xymatrixrowsep{1pc}
\xymatrixcolsep{.75pc}
 \xymatrix{
 & & \cL_{r,s} & \\
 \cP_{r,s}: & \cL_{r-1,p-s} \ar[ru] & & \cL_{r+1,p-s} \ar[lu] \\
& & \cL_{r,s} \ar[lu] \ar[ru] & \\
 }
\end{equation*}
\end{thm}
\begin{proof}
 The Loewy diagram for $\cP_{r,s}$ follows from that for $\cP_{1,s}$ exactly as in the $s=p-1$ case of Proposition \ref{prop:Pr_structure}. Also as in Proposition \ref{prop:Pr_structure}, $\cP_{r,s}$ is indecomposable and there is a surjection $\cP_{r,s}\rightarrow\cL_{r,s}$. Moreover, $\cP_{r,s}$ is projective in $\cO_c^0$ since it is the tensor product of a rigid with a projective module. Thus the argument of Proposition \ref{prop:Prp-1_proj_cover} shows that $\cP_{r,s}$ is a projective cover of $\cL_{r,s}$.
\end{proof}

\section{Tensor product formulas and semisimplification}

We now compute all tensor products involving irreducible modules $\cL_{r,s}$ and their projective covers $\cP_{r,s}$. As a consequence, we show that there is a semisimple subquotient category of $\cO_c$ which is a product of two $\mathfrak{sl}_2$-type tensor subcategories.

\subsection{General fusion rules}
We first show how the irreducible modules $\cL_{r',1}$ and $\cL_{1,2}$ tensor with the projective covers; recall that $\cP_{r,p}=\cL_{r,p}$ for $r\geq 1$:
\begin{thm}\label{thm:prs_fusion}
\begin{itemize}
\item[(1)]For $r,r'\geq 1$ and $1\leq s \leq p$,
\begin{equation}\label{moreprs}
\cL_{r',1}\boxtimes \cP_{r,s} \cong\bigoplus_{\substack{k = |r-r'|+1\\ k+r+r' \equiv 1\; ({\rm mod}\; 2)}}^{r+r'-1} \cP_{k,s}
\end{equation}
\item[(2)]For $p\geq 3$ and $r \geq 1$, $1\leq s \leq p-1$,
\begin{equation}\label{more1prs}
\cL_{1,2}\boxtimes \cP_{r,s} \cong
\begin{cases}
\cP_{1,2}\oplus \cP_{2,p} \;\;\; &\mbox{if}\;\;\; r=s = 1\\
\cP_{r,2}\oplus \cP_{r-1,p}\oplus \cP_{r+1, p} \;\;\; &\mbox{if}\;\;\; s = 1, \; r \geq 2\\
\cP_{r,s-1}\oplus \cP_{r,s+1}\;\;\; &\mbox{if}\;\;\; 2\leq s\leq p-2\\
\cP_{r,p-2}\oplus 2\cdot \cP_{r,p}, \;\;\; &\mbox{if}\;\;\; s = p-1
\end{cases}
\end{equation}
\item[(3)] For $p=2$ and $r\geq 1$,
\begin{equation}\label{more2prs}
 \cL_{1,2}\boxtimes\cP_{r,1} \cong
\begin{cases}
 2\cdot\cP_{1,2}\oplus \cP_{2,2}\;\;\; &\mbox{if}\;\;\; r = 1 \\
\cP_{r-1,2}\oplus 2\cdot\cP_{r,2}\oplus\cP_{r+1,2}\;\;\; &\mbox{if}\;\;\; r \geq 2
\end{cases}
\end{equation}
\end{itemize}
\end{thm}
\begin{proof}
The $s=p$ case of \eqref{moreprs} is just the $s=p$ case of \eqref{fr1}. For $1\leq s\leq p-1$ and $r=1$, \eqref{moreprs} is just \eqref{eqn:Prp-1} and the definition of $\cP_{r',s}$ in Theorem \ref{thm:Prs_structure}. For $r\geq 2$, we simply calculate
\begin{align*}
 \cL_{r',1}\tens\cP_{r,s} & =\cL_{r',1}\tens(\cL_{r,1}\tens\cP_{1,s})\cong(\cL_{r',1}\tens\cL_{r,1})\tens\cP_{1,s}\nonumber\\
 & \cong \bigoplus_{\substack{k = |r-r'|+1\\ k+r+r' \equiv 1\; ({\rm mod}\; 2)}}^{r+r'-1} \cL_{k,1}\tens\cP_{1,s}\cong\bigoplus_{\substack{k = |r-r'|+1\\ k+r+r' \equiv 1\; ({\rm mod}\; 2)}}^{r+r'-1} \cP_{k,s},
\end{align*}
using the $r=1$ case and \eqref{fr1}.

Next, note that the $r=1$, $2\leq s\leq p-1$ cases of \eqref{more1prs} are immediate from our construction of the modules $\cP_{1,s}$ in Section \ref{subsec:more_proj_covers}. For $r\geq 2$, we can then use \eqref{moreprs} and the $r=1$ case. It remains to prove the $s=1$ cases of \eqref{more1prs}.

Taking $s=1$ now, the maximal proper submodule $\cZ_{r,1}$ of $\cP_{r,1}$ satisfies the exact sequence
\begin{equation*}
0 \longrightarrow \cL_{r,1} \longrightarrow \cZ_{r,1} \longrightarrow [\cL_{r-1, p-1}\oplus] \cL_{r+1, p-1} \longrightarrow 0,
\end{equation*}
where from now on, terms in brackets vanish if $r=1$.
 Applying $\cL_{1,2}\boxtimes \bullet$ and using the fusion rules \eqref{fr2}, we have
\begin{equation}\label{seq:more1prs_proof}
0 \longrightarrow \cL_{r,2} \longrightarrow \cL_{1,2}\boxtimes \cZ_{r,1} \longrightarrow [\cL_{r-1, p-2}\oplus \cL_{r-1,p}] \oplus \cL_{r+1, p-2}\oplus \cL_{r+1,p} \longrightarrow 0.
\end{equation}
Since both of $\cL_{r\pm1,p}$ are projective, $[\cL_{r-1,p}\oplus] \cL_{r+1,p}$ is a direct summand of $\cL_{1,2}\boxtimes \cZ_{r,1}$. Then the complement $\widetilde{\cZ}_{r, 2}$ of $[\cL_{r-1,p}\oplus] \cL_{r+1,p}$ satisfies the exact sequence
\begin{equation}\label{seq:z_til}
0 \longrightarrow \cL_{r,2} \longrightarrow \widetilde{\cZ}_{r, 2} \longrightarrow [\cL_{r-1, p-2}\oplus] \cL_{r+1, p-2} \longrightarrow 0.
\end{equation}

Now consider the exact sequence
\begin{equation*}
0 \longrightarrow \cZ_{r,1} \longrightarrow \cP_{r,1} \longrightarrow \cL_{r,1} \longrightarrow 0.
\end{equation*}
Applying $\cL_{1,2}\boxtimes \bullet$ and using the fusion rules \eqref{fr2}, we have
\begin{equation*}
0 \longrightarrow \widetilde{\cZ}_{r,2}\oplus [\cL_{r-1,p}\oplus] \cL_{r+1,p} \longrightarrow \cL_{1,2}\boxtimes \cP_{r,1} \longrightarrow \cL_{r,2} \longrightarrow 0.
\end{equation*}
Since both of $\cL_{r\pm1,p}$ are injective, there exists a direct summand $\widetilde{\cP}_{r,2}$ of $\cL_{1,2}\boxtimes \cP_{r,1}$ complementary to $[\cL_{r-1,p}\oplus] \cL_{r+1,p}$ satisfying the exact sequence
\begin{equation}\label{seq:p_til}
0 \longrightarrow \widetilde{\cZ}_{r,2} \longrightarrow \widetilde{\cP}_{r,2} \longrightarrow \cL_{r,2} \longrightarrow 0.
\end{equation}
The module $\widetilde{P}_{r,2}$ is projective in $\cO_c^0$ since it is a summand of a projective module. Since $\cP_{r,2}$ is a projective cover of $\cL_{r,2}$, there is thus a surjection $\widetilde{P}_{r,2}\longrightarrow\cP_{r,2}$; but since \eqref{seq:z_til} and \eqref{seq:p_til} show that these two modules have the same length, we get $\widetilde{P}_{r,2}\cong\cP_{r,2}$. Therefore
\begin{equation*}
\cL_{1,2}\boxtimes \cP_{r,1} \cong \cP_{r,2} \oplus  \cL_{r+1,p} [\oplus \cL_{r-1,p}],
\end{equation*}
proving \eqref{more1prs} for $s = 1$.

Now when $p=2$, we need to replace the exact sequence \eqref{seq:more1prs_proof} with
\begin{equation*}
0\longrightarrow\cL_{r,2}\longrightarrow\cL_{1,2}\boxtimes\cZ_{r,1}\longrightarrow [\cL_{r-1,2}\oplus] \cL_{r+1,2}\longrightarrow 0.
\end{equation*}
Since both $\cL_{r\pm1,2}=\cP_{r\pm1,2}$ are projective, this exact sequence splits. The exact sequence
\begin{equation*}
 0\longrightarrow \cL_{1,2}\boxtimes\cZ_{r,1}\longrightarrow\cL_{1,2}\boxtimes\cP_{r,1}\longrightarrow \cL_{1,2}\boxtimes\cL_{r,1}\longrightarrow 0
\end{equation*}
also splits because $\cL_{1,2}\boxtimes\cL_{r,1}\cong\cL_{r,2}$ is projective. Then these two split exact sequences together imply \eqref{more2prs}.
\end{proof}

Finally, here are all fusion rules involving the simple modules $\cL_{r,s}$ and their projective covers in $\cO_c^0$:
\begin{thm}\label{generalfusionrules}
All tensor products in $\cO_c$ of the $V_c$-modules $\cL_{r,s}$ and $\cP_{r,s}$ are as follows, with sums taken to be empty if the lower bound exceeds the upper bound:
\begin{itemize}
\item[(1)] For $r, r' \geq 1$ and $1 \leq s, s' \leq p$,
\begin{align}\label{caser}
& \cL_{r,s}\boxtimes \cL_{r',s'}   \cong \bigoplus_{\substack{k = |r-r'|+1\\ k+r+r' \equiv 1\; ({\rm mod}\; 2)}}^{r+r'-1} \bigg(\bigoplus_{\substack{\ell = |s-s'|+1\\ \ell+s+s' \equiv 1\; (\mathrm{mod}\; 2)}}^{\min(s+s'-1, 2p-1-s-s')} \cL_{k, \ell} \oplus \bigoplus_{\substack{\ell = 2p+1-s-s'\\ \ell+s+s' \equiv 1\; (\mathrm{mod}\; 2)}}^{p} \cP_{k, \ell}\bigg).
\end{align}
\item[(2)] For $r, r' \geq 1$, $1 \leq s \leq p$, and $1 \leq s' \leq p-1$,
\begin{align}\label{casePM}
\cL_{r,s}\boxtimes \cP_{r', s'} & \cong \bigoplus_{\substack{k = |r-r'|+1\\ k+r+r' \equiv 1\; ({\rm mod}\; 2)}}^{r+r'-1}\bigg(\bigoplus_{\substack{\ell = |s-s'|+1\\ \ell+s+s' \equiv 1\; ({\rm mod}\; 2)}}^{\min(s+s'-1, p)}\cP_{k, \ell}\oplus \bigoplus_{\substack{\ell = 2p+1-s-s'\\ \ell+s+s' \equiv 1\; ({\rm mod}\; 2)}}^{p}\cP_{k, \ell}\bigg)\nonumber\\
&\qquad\oplus\bigoplus_{\substack{\ell = p-s+s'+1\\ \ell+p+s+s' \equiv 1\; ({\rm mod}\; 2)}}^{p}\bigg(\bigoplus_{\substack{k = \max(|r-r'|,1)\\ k+r+r' \equiv 0\; ({\rm mod}\; 2)}}^{r+r'-2} \cP_{k,\ell}\oplus\bigoplus_{\substack{k = |r-r'|+2\\ k+r+r' \equiv 0\; ({\rm mod}\; 2)}}^{r+r'}\cP_{k,\ell}\bigg).
\end{align}
\item[(3)]For $r,r'\geq 1$ and $1\leq s,s' \leq p-1$,
\begin{align}\label{casePP}
 \cP_{r,s}\tens\cP_{r',s'} & \cong2\cdot\bigoplus_{\substack{k = |r-r'|+1\\ k+r+r' \equiv 1\; ({\rm mod}\; 2)}}^{r+r'-1}\bigg(\bigoplus_{\substack{\ell = |s-s'|+1\\ \ell+s+s' \equiv 1\; ({\rm mod}\; 2)}}^{\min(s+s'-1, p)}\cP_{k, \ell}\oplus\bigoplus_{\substack{\ell = 2p+1-s-s'\\ \ell+s+s' \equiv 1\; ({\rm mod}\; 2)}}^{p}\cP_{k,\ell}\bigg)\nonumber\\
 &\qquad \oplus\bigoplus_{\substack{\ell = s+s'+1\\ \ell+s+s' \equiv 1\; ({\rm mod}\; 2)}}^{p}\bigg( \bigoplus_{\substack{k = \max(|r-r'|-1,1)\\ k+r+r' \equiv 1\; ({\rm mod}\; 2)}}^{r+r'-3}\cP_{k,\ell}\oplus \bigoplus_{\substack{k = \max(|r-r'|+1,2)\\ k+r+r' \equiv 1\; ({\rm mod}\; 2)}}^{r+r'-1}\cP_{k,\ell}\nonumber\\
&\qquad\qquad\qquad\qquad\qquad\qquad\oplus\bigoplus_{\substack{k = |r-r'|+1\\ k+r+r' \equiv 1\; ({\rm mod}\; 2)}}^{r+r'-1}\cP_{k,\ell}\oplus\bigoplus_{\substack{k = |r-r'|+3\\ k+r+r' \equiv 1\; ({\rm mod}\; 2)}}^{r+r'+1}\cP_{k,\ell} \bigg)\nonumber\\
 &\qquad\oplus\bigoplus_{\substack{k = \max(|r-r'|,1)\\ k+r+r' \equiv 0\; ({\rm mod}\; 2)}}^{r+r'-2}\bigg(\bigoplus_{\substack{\ell = |p-s-s'|+1\\ \ell+p+s+s' \equiv 1\; ({\rm mod}\; 2)}}^{p}\cP_{k, \ell}\oplus\bigoplus_{\substack{\ell = p-\vert s-s'\vert+1\\ \ell+p+s+s' \equiv 1\; ({\rm mod}\; 2)}}^{p}\cP_{k, \ell}\bigg)\nonumber\\
 &\qquad\oplus\bigoplus_{\substack{k = |r-r'|+2\\ k+r+r' \equiv 0\; ({\rm mod}\; 2)}}^{r+r'}\bigg(\bigoplus_{\substack{\ell = |p-s-s'|+1\\ \ell+p+s+s' \equiv 1\; ({\rm mod}\; 2)}}^{p}\cP_{k, \ell}\oplus\bigoplus_{\substack{\ell = p-\vert s-s'\vert+1\\ \ell+p+s+s' \equiv 1\; ({\rm mod}\; 2)}}^{p}\cP_{k, \ell}\bigg)
\end{align}
 \end{itemize}
\end{thm}
\begin{proof}
The proof of the $r=r'=1$ case of \eqref{caser} is exactly the same as the corresponding proof in \cite[Theorem 5.2.1]{CMY2}, so we omit the details:
\begin{equation*}\label{caser_1}
\cL_{1,s}\boxtimes \cL_{1,s'} \cong \bigoplus_{\substack{\ell = |s-s'|+1\\ \ell+s+s' \equiv 1\; (\mathrm{mod}\; 2)}}^{{\rm min}\{s+s'-1, 2p-1-s-s'\}}\cL_{1, \ell} \oplus \bigoplus_{\substack{\ell = 2p+1-s-s'\\ \ell+s+s' \equiv 1\; (\mathrm{mod}\; 2)}}^{p}\cP_{1, \ell}.
\end{equation*}
The general case then follows from the commutativity and associativity of tensor products in $\cO_c$ and the fusion rules \eqref{fr1} and \eqref{moreprs}:
\begin{align*}
 \cL_{r,s}\tens\cL_{r',s'} & \cong(\cL_{r,1}\tens\cL_{r',1})\tens(\cL_{1,s}\tens\cL_{1,s'})\nonumber\\
 & \cong\bigg(\bigoplus_{\substack{k = |r-r'|+1\\ k+r+r' \equiv 1\; (\mathrm{mod}\; 2)}}^{r+r'-1}\cL_{k, 1}\bigg)\tens\bigg(\bigoplus_{\substack{\ell = |s-s'|+1\\ \ell+s+s' \equiv 1\; (\mathrm{mod}\; 2)}}^{{\rm min}\{s+s'-1, 2p-1-s-s'\}}\cL_{1, \ell} \oplus \bigoplus_{\substack{\ell = 2p+1-s-s'\\ \ell+s+s' \equiv 1\; (\mathrm{mod}\; 2)}}^{p}\cP_{1, \ell}\bigg)\nonumber\\
 &\cong \bigoplus_{\substack{k = |r-r'|+1\\ k+r+r' \equiv 1\; (\mathrm{mod}\; 2)}}^{r+r'-1}\bigg(\bigoplus_{\substack{\ell = |s-s'|+1\\ \ell+s+s' \equiv 1\; (\mathrm{mod}\; 2)}}^{{\rm min}\{s+s'-1, 2p-1-s-s'\}}\cL_{k, \ell} \oplus \bigoplus_{\substack{\ell = 2p+1-s-s'\\ \ell+s+s' \equiv 1\; (\mathrm{mod}\; 2)}}^{p}\cP_{k, \ell}\bigg).
\end{align*}

Let us now consider the $r=r'=1$ case of \eqref{casePM}:
\begin{align}\label{eqn:LPr=r'=1}
\cL_{1,s}\tens\cP_{1,s'} & \cong \bigoplus_{\substack{\ell = |s-s'|+1\\ \ell+s+s' \equiv 1\; ({\rm mod}\; 2)}}^{\min(s+s'-1, p)}\cP_{1, \ell}\oplus \bigoplus_{\substack{\ell = 2p+1-s-s'\\ \ell+s+s' \equiv 1\; ({\rm mod}\; 2)}}^{p}\cP_{1, \ell} \oplus \bigoplus_{\substack{\ell = p-s+s'+1\\ \ell+p+s+s' \equiv 1\; ({\rm mod}\; 2)}}^{p}\cP_{2, \ell}.
\end{align}
The case $s=1$ is easy since $\cL_{1,1}$ is the unit object of $\cO_c$ and since only the first sum in \eqref{eqn:LPr=r'=1} is non-empty (because $s'\leq p-1$). Then for $s=2$, \eqref{eqn:LPr=r'=1} in the cases $s'=1$, $2\leq s'\leq p-2$, and $s'=p-1$ yields the corresponding cases of \eqref{more1prs} and \eqref{more2prs}. This proves \eqref{eqn:LPr=r'=1} when $p=2$, and for $p\geq 3$, we can finish the proof using induction on $s$.

Thus assume we have proved \eqref{eqn:LPr=r'=1} for some $s$ such that $2\leq s\leq p-1$, and consider the $s+1$ case. Since
\begin{equation*}
 \cL_{1,2}\tens(\cL_{1,s}\tens\cP_{1,s'}) \cong(\cL_{1,2}\tens\cL_{1,s})\tens\cP_{1,s'}\cong(\cL_{1,s-1}\tens\cP_{1,s'})\oplus(\cL_{1,s+1}\tens\cP_{1,s'})
\end{equation*}
and since all these tensor products have finite length, the Krull-Schmidt Theorem guarantees that we can determine the indecomposable summands of $\cL_{1,s+1}\tens\cP_{1,s'}$ by subtracting the indecomposable summands of $\cL_{1,s-1}\tens\cP_{1,s'}$ from those of $\cL_{1,2}\tens(\cL_{1,s}\tens\cP_{1,s'})$. So we get
\begin{align*}
 \cL_{1,s+1}\tens\cP_{1,s'}\cong \left\lbrace\begin{array}{lll}
                                              (\cL_{1,s}\tens\cP_{1,2})\oplus(\cL_{1,s}\tens\cL_{2,p})\ominus(\cL_{1,s-1}\tens\cP_{1,1}) & \text{if} & s'=1\\
                                              (\cL_{1,s}\tens\cP_{1,s'-1})\oplus(\cL_{1,s}\tens\cP_{1,s'+1})\ominus(\cL_{1,s-1}\tens\cP_{1,s'}) & \text{if} & 2\leq s'\leq p-2\\
                                              (\cL_{1,s}\tens\cP_{1,p-2})\oplus2\cdot(\cL_{1,s}\tens\cL_{1,p})\ominus(\cL_{1,s-1}\tens\cP_{1,p-1}) & \text{if} & s'=p-1\\
                                             \end{array}
                                             \right.
\end{align*}
using the fusion rules \eqref{more1prs}. Analysis of these three formulas using the $s$ and $s-1$ cases of \eqref{eqn:LPr=r'=1} (which hold by induction), as well as $s'=p$ cases of \eqref{caser}, then yields the $s+1$ case of \eqref{eqn:LPr=r'=1}. For $s'=p-1$, it is helpful to divide the analysis into cases $s<p-1$ and $s=p-1$.

Now we prove \eqref{casePM} for general $r,r'$ using the $r=r'=1$ case along with \eqref{fr1} and \eqref{moreprs}:
\begin{align*}
 \cL_{r,s}\tens\cP_{r',s'} & \cong(\cL_{r,1}\tens\cL_{r',1})\tens(\cL_{1,s}\tens\cP_{1,s'})\nonumber\\
 & \cong \bigoplus_{\substack{k = |r-r'|+1\\ k+r+r' \equiv 1\; ({\rm mod}\; 2)}}^{r+r'-1}\bigg(\bigoplus_{\substack{\ell = |s-s'|+1\\ \ell+s+s' \equiv 1\; ({\rm mod}\; 2)}}^{\min(s+s'-1, p)}(\cL_{k,1}\tens\cP_{1, \ell})\oplus \bigoplus_{\substack{\ell = 2p+1-s-s'\\ \ell+s+s' \equiv 1\; ({\rm mod}\; 2)}}^{p}(\cL_{k,1}\tens\cP_{1, \ell})\bigg)\nonumber\\
& \hspace{2em}\oplus \bigoplus_{\substack{k = |r-r'|+1\\ k+r+r' \equiv 1\; ({\rm mod}\; 2)}}^{r+r'-1}\bigoplus_{\substack{\ell = p-s+s'+1\\ \ell+p+s+s' \equiv 1\; ({\rm mod}\; 2)}}^{p}(\cL_{k,1}\tens\cP_{2, \ell})\nonumber\\
& \cong \bigoplus_{\substack{k = |r-r'|+1\\ k+r+r' \equiv 1\; ({\rm mod}\; 2)}}^{r+r'-1}\bigg(\bigoplus_{\substack{\ell = |s-s'|+1\\ \ell+s+s' \equiv 1\; ({\rm mod}\; 2)}}^{\min(s+s'-1, p)}\cP_{k, \ell}\oplus \bigoplus_{\substack{\ell = 2p+1-s-s'\\ \ell+s+s' \equiv 1\; ({\rm mod}\; 2)}}^{p}\cP_{k, \ell}\bigg)\nonumber\\
&\hspace{2em}\oplus\bigoplus_{\substack{\ell = p-s+s'+1\\ \ell+p+s+s' \equiv 1\; ({\rm mod}\; 2)}}^{p}\bigg(\bigoplus_{\substack{k = \max(|r-r'|,1)\\ k+r+r' \equiv 0\; ({\rm mod}\; 2)}}^{r+r'-2} \cP_{k,\ell}\oplus\bigoplus_{\substack{k = |r-r'|+2\\ k+r+r' \equiv 0\; ({\rm mod}\; 2)}}^{r+r'}\cP_{k,\ell}\bigg)
\end{align*}
as required.

To prove \eqref{casePP}, we again take $r=r'=1$ first. The exact sequences
\begin{equation*}
 0\longrightarrow\cZ_{1,s}\tens\cP_{1,s'}\longrightarrow\cP_{1,s}\tens\cP_{1,s'}\longrightarrow\cL_{1,s}\tens\cP_{1,s'}\longrightarrow 0
\end{equation*}
and
\begin{equation*}
 0\longrightarrow\cL_{1,s}\tens\cP_{1,s'}\longrightarrow\cZ_{1,s}\tens\cP_{1,s'}\longrightarrow\cL_{2,p-s}\tens\cP_{1,s'}\longrightarrow 0,
\end{equation*}
both of which split since $\cL_{1,s}\tens\cP_{1,s'}$ and $\cL_{2,p-s}\tens\cP_{1,s'}$ are projective in $\cO_c^0$, imply that
\begin{equation*}
 \cP_{1,s}\boxtimes\cP_{1,s'}\cong 2\cdot(\cL_{1,s}\boxtimes\cP_{1,s'})
 \oplus(\cL_{2,p-s}\boxtimes\cP_{1,s'}).
\end{equation*}
Thus by \eqref{casePM},
\begin{align*}
 \cP_{1,s}\tens\cP_{1,s'}  & \cong2\cdot\bigg(\bigoplus_{\substack{\ell = |s-s'|+1\\ \ell+s+s' \equiv 1\; ({\rm mod}\; 2)}}^{\min(s+s'-1, p)}\cP_{1, \ell}\oplus \bigoplus_{\substack{\ell = 2p+1-s-s'\\ \ell+s+s' \equiv 1\; ({\rm mod}\; 2)}}^{p}\cP_{1, \ell} \oplus \bigoplus_{\substack{\ell = p-s+s'+1\\ \ell+p+s+s' \equiv 1\; ({\rm mod}\; 2)}}^{p}\cP_{2, \ell}\bigg)\nonumber\\
 &\qquad \oplus \bigoplus_{\substack{\ell = |p-s-s'|+1\\ \ell+p+s+s' \equiv 1\; ({\rm mod}\; 2)}}^{\min(p-s+s'-1, p)}\cP_{2, \ell}\oplus \bigoplus_{\substack{\ell = p+s-s'+1\\ \ell+p+s+s' \equiv 1\; ({\rm mod}\; 2)}}^{p}\cP_{2, \ell} \oplus \bigoplus_{\substack{\ell = s+s'+1\\ \ell+s+s' \equiv 1\; ({\rm mod}\; 2)}}^{p}(\cP_{1, \ell}\oplus\cP_{3,\ell})\nonumber\\
 &\cong \bigoplus_{\substack{\ell = |s-s'|+1\\ \ell+s+s' \equiv 1\; ({\rm mod}\; 2)}}^{ p}\cP_{1, \ell}\oplus\bigoplus_{\substack{\ell = |s-s'|+1\\ \ell+s+s' \equiv 1\; ({\rm mod}\; 2)}}^{\min(s+s'-1, p)}\cP_{1, \ell}\oplus\bigoplus_{\substack{\ell = s+s'+1\\ \ell+s+s' \equiv 1\; ({\rm mod}\; 2)}}^{p}\cP_{3,\ell}\nonumber\\
 & \qquad\oplus \bigoplus_{\substack{\ell = 2p+1-s-s'\\ \ell+s+s' \equiv 1\; ({\rm mod}\; 2)}}^{p}(2\cdot\cP_{1, \ell}) \oplus \bigoplus_{\substack{\ell = |p-s-s'|+1\\ \ell+p+s+s' \equiv 1\; ({\rm mod}\; 2)}}^{p}\cP_{2, \ell}\oplus\bigoplus_{\substack{\ell = p-\vert s-s'\vert+1\\ \ell+p+s+s' \equiv 1\; ({\rm mod}\; 2)}}^{p}\cP_{2, \ell}.
\end{align*}
We now get \eqref{casePP} by tensoring this expression with $\cL_{r,1}\tens\cL_{r',1}$ as before.
\end{proof}

\begin{rem}
The fusion rules for irreducible $V_c$-modules in $\cO_c$ follow from the tensor product formula \eqref{caser}: For $r,r',r''\geq 1$ and $1\leq s,s',s''\leq p$,
\begin{equation*}
\dim \hom_{V_c}(\cL_{r,s}\tens\cL_{r',s'}, \cL_{r'',s''})\leq 1,
\end{equation*}
with equality if and only if
$$r''\in\lbrace r+r'-1,r+r'-3,\ldots,\vert r-r'\vert+1\rbrace$$
and
$$s''\in\lbrace s+s'-1,s+s'-3,\ldots,\vert s-s'\vert+1\rbrace$$
(with $s''\leq p$ also). This agrees with \cite[Theorem 2.3]{Lin}, but note that the fusion rule result of \cite{Lin} does not distinguish between $\cL_{r'',s''}$ or $\cP_{r'',s''}$ appearing as a summand of $\cL_{r,s}\tens\cL_{r',s'}$.
\end{rem}

\subsection{Semisimplification}\label{subsec:ss}

Theorem \ref{generalfusionrules} show that the full subcategory $\cO'_c\subseteq\cO_c$ whose objects are finite direct sums of modules $\cL_{r,s}$ and $\cP_{r,s}$ for $r\geq 1$, $1\leq s\leq p$ is an additive monoidal subcategory  of $\cO_c$ (but it is not abelian since it is not closed under submodules and quotients). Since the modules $\cL_{r,s}$ and $\cP_{r,s}$ are all self-dual, $\cO'_c$ is a ribbon category, and thus we can define its semisimplification $\overline{\cO_c'}$ as usual to be the quotient of $\cO_c'$ by the tensor ideal of negligible morphisms. Recall (see for example \cite[Definition 3.3.16]{BK}) that $f: \cW_1\rightarrow\cW_2$ in $\cO'_c$ is negligible if the categorical trace $\tr_{\cO'_c}f\circ g$ vanishes for all morphisms $g:\cW_2\rightarrow\cW_1$. Moreover, an object $\cW$ in $\cO'_c$ is negligible if $\Id_\cW$ is negligible; such objects are isomorphic to $0$ in the semisimplification $\overline{\cO_c'}$.

\begin{lem}
 An irreducible module $\cL_{r,s}$ is negligible in $\cO'_c$ if and only if $s=p$. Moreover, all projective modules $\cP_{r,s}$ are negligible.
\end{lem}
\begin{proof}
 Since $\cL_{r,s}$ is irreducible, $\mathrm{End}_{\cO'_c}(\cL_{r,s})=\CC\cdot \Id_{\cL_{r,s}}$ and thus $\cL_{r,s}$ is negligible if and only if its categorical dimension $\tr_{\cO'_c} \Id_{\cL_{r,s}}$ vanishes. Then \eqref{eqn:rs_cat_dim} shows that $\dim_{\cO_c}\cL_{r,s}=0$ if and only if $s=p$. For the projective modules, the definitions and constructions in Sections \ref{sec:fus_and_rig} and \ref{subsec:more_proj_covers} show that every $\cP_{r,s}$ is in the tensor ideal generated by the modules $\cL_{r,p}$. Since negligible morphisms are a tensor ideal containing all $\Id_{\cL_{r,p}}$, each $\cP_{r,s}$ is negligible.
\end{proof}

\begin{cor}
 The category $\overline{\cO_c'}$ is a semisimple abelian category with simple objects $\cL_{r,s}$ for $r\geq 1$ and $1\leq s\leq p-1$.
\end{cor}

Since negligible morphisms form a tensor ideal, the semisimplification $\overline{\cO_c'}$ is also a (ribbon) tensor category. Tensor products of simple objects in $\overline{\cO_c'}$ follow from Theorem \ref{generalfusionrules}(1):
\begin{prop}\label{prop:ss_fus_rules}
 Simple objects in $\overline{\cO_c'}$ have the following tensor products:
 \begin{equation*}
  \cL_{r,s}\boxtimes \cL_{r',s'}   \cong \bigoplus_{\substack{k = |r-r'|+1\\ k+r+r' \equiv 1\; ({\rm mod}\; 2)}}^{r+r'-1} \bigoplus_{\substack{\ell = |s-s'|+1\\ \ell+s+s' \equiv 1\; (\mathrm{mod}\; 2)}}^{\min(s+s'-1, 2p-1-s-s')} \cL_{k, \ell}
 \end{equation*}
for $r, r'\geq 1$ and $1\leq s,s'\leq p-1$.
\end{prop}

From this proposition, we see that as an abelian category, $\overline{\cO_c'}$ decomposes as the Deligne product of two tensor subcategories. First, the modules $\cL_{r,1}$ are the simple objects of a tensor subcategory which we denote by $\cO_c^L$.  As discussed in the proof of Theorem \ref{thm:Lr1_fus_rules} (see also \cite[Corollary~14]{ACGY}), $\cO_c^L$ is braided tensor equivalent to an abelian $3$-cocycle twist of $\mathrm{Rep}\, SU(2)$ (or $\mathrm{Rep}\,\mathfrak{sl}_2$). This same cocycle twist of $\rep\mathfrak{sl}_2$ is also braided tensor equivalent to the Kazhdan-Lusztig category $KL_{-2+1/p}(\mathfrak{sl}_2)$ of modules for the simple affine vertex operator algebra $L_{-2+1/p}(\mathfrak{sl}_2)$ at level $-2+\frac{1}{p}$ \cite[Corollary~9]{ACGY}. Thus $\cO_c^L$ is braided tensor equivalent to $KL_{-2+1/p}(\mathfrak{sl}_2)$, although they have different ribbon twists because the conformal weights of $\cL_{2,1}$ differ from those of the corresponding simple $L_{-2+1/p}(\mathfrak{sl}_2)$-module.

Secondly, although the modules $\cL_{1,s}$ for $1\leq s\leq p-1$ do not form the simple objects of a tensor subcategory of $\cO_c$, they do in the semisimple subquotient $\overline{\cO_c'}$. We denote by $\cO_c^R$ the subcategory generated by $\cL_{1,s}$ for $1\leq s\leq p-1$. Then the $r=r'=1$ case of Proposition \ref{prop:ss_fus_rules} yields precisely the Frenkel-Zhu fusion rules \cite{FZ1} for the simple affine vertex operator algebra $L_{-2+p}(\mathfrak{sl}_2)$, under the identification of $\cL_{1,s}$ with the simple $L_{-2+p}(\mathfrak{sl}_2)$-module induced from the $s$-dimensional simple $\mathfrak{sl}_2$-module. We can actually prove a stronger relationship:
\begin{prop}
The subcategory $\cO_c^R$ is tensor equivalent to the category $KL_{-2+p}(\mathfrak{sl}_2)$ of modules for the simple affine vertex operator algebra $L_{-2+p}(\mathfrak{sl}_2)$.
\end{prop}
\begin{proof}
From \cite{F}, the category $KL_{-2+p}(\mathfrak{sl}_2)$  is equivalent (as modular tensor categories) to the semisimplification of the category of tilting modules for the Lusztig quantum group $U_q(\mathfrak{sl}_2)$ at $q = e^{\pi i/p}$ \cite{AnP}. We denote this category by $\cC(q, \mathfrak{sl}_2)$.

Proposition \ref{prop:ss_fus_rules} shows that the Grothendieck rings of the categories $\cO_c^R$ and $\cC(q,\mathfrak{sl}_2)$ are isomorphic under the map $[\cL_{1,s}] \rightarrow [V_{s-1}]$, with $V_{s-1}$ the $s$-dimensional irreducible representation of $U_{q}(\mathfrak{sl}_2)$. Then by \cite[Theorem $A_\ell$]{KW}, $\cO_c^R$ is tensor equivalent to $\cC(\tilde{q}, \mathfrak{sl}_2)^{\tau}$, where $\tilde{q}^2$ is a primitive root of unity of order $p$ (unique up to $\tilde{q}^2\rightarrow\tilde{q}^{-2}$) and $\tau$ denotes modification of the associativity isomorphisms in $\cC(\tilde{q}, \mathfrak{sl}_2)$ by a $3$-cocycle on $\ZZ/2\ZZ$. Up to coboundaries, there is only one non-trivial $3$-cocycle $\tau$ on $\ZZ/2\ZZ$: it modifies the usual associativity isomorphism $V_1\otimes(V_1\otimes V_1)\rightarrow(V_1\otimes V_1)\otimes V_1$ in $\cC(\tilde{q},\mathfrak{sl}_2)$ by a sign.

The tensor categories $\cC(\tilde{q},\mathfrak{sl}_2)$ for various $2p$th roots of unity can be distinguished using the evaluation $e_{V_1}: V_1^*\otimes V_1\rightarrow\CC$ and coevaluation $i_{V_1}: \CC\rightarrow V_1\otimes V_1^*$ (see for example \cite[Exercise 8.18.8]{EGNO}). Specifically, if we identify $V_1=V_1^*=V_1^{**}$, then $e_{V_1}\circ i_{V_1}\in\CC$ is an invariant of the tensor category structure on $\cC(\tilde{q},\mathfrak{sl}_2)$, and in fact
\begin{equation}\label{eqn:quantum_left_tr}
 e_{V_1}\circ i_{V_1} = -\tilde{q}-\tilde{q}^{-1}.
\end{equation}
If $\tau$ is the non-trivial $3$-cocycle on $\ZZ/2\ZZ$, then  $e_{V_1}\circ i_{V_1}= \tilde{q}+\tilde{q}^{-1}$ in $\cC(\tilde{q},\mathfrak{sl}_2)^\tau$, since modification of $\cA_{V_1,V_1,V_1}$ by a sign means that either $e_{V_1}$ or $i_{V_1}$ should be modified by a sign to get rigidity.

For our tensor category $\cO_c^R$, we showed in \eqref{eqn:L12_left_trace} that
\begin{equation*}
 e_{\cL_{1,2}}\circ i_{\cL_{1,2}} =-2\cos(\pi/p) =-\frac{\sin(2\pi/p)}{\sin(\pi/p)}= -\frac{e^{2\pi i/p}-e^{-2\pi i/p}}{e^{\pi i/p}-e^{-\pi i/p}} =-e^{\pi i/p}-e^{-\pi i/p}.
\end{equation*}
Comparing with \eqref{eqn:quantum_left_tr}, we see that $\cO_c^R$ must be tensor equivalent to either $\cC(q,\mathfrak{sl}_2)$ or $\cC(-q,\mathfrak{sl}_2)^\tau$. But these two quantum group categories are equivalent to each other: Since $\pm q$ square to the same primitive $p$th root of unity, \cite{KW} implies that $\cC(q,\mathfrak{sl}_2)$ is tensor equivalent to a $3$-cocycle twist of $\cC(-q,\mathfrak{sl}_2)$, and this cocycle has to be the non-trivial one because $\cC(q,\mathfrak{sl}_2)$ and $\cC(-q,\mathfrak{sl}_2)$ are not tensor equivalent. We conclude that $\cO_c^R$ is tensor equivalent to $\cC(q,\mathfrak{sl}_2)$, and thus also to the tensor category of $L_{-2+p}(\mathfrak{sl}_2)$-modules.
\end{proof}

\begin{rem}
The appearance of affine $\mathfrak{sl}_2$ tensor categories in the semisimplification of $\cO_c$ is not surprising because the Virasoro algebra at central charge $13-6p-6p^{-1}$ is the quantum Drinfeld-Sokolov reduction \cite{FFr} of both universal affine vertex operator algebras $V_{-2+1/p}(\mathfrak{sl}_2)$ and $V_{-2+p}(\mathfrak{sl}_2)$  (see also \cite[Chapter 15]{FB}).
\end{rem}

\begin{rem}
As a braided tensor category, $\overline{\cO_c'}$ is not quite the Deligne product of $\cO_c^L$ and $\cO_c^R$, since these two subcategories do not quite centralize each other. Indeed, the balancing equation for monodromies implies
\begin{equation*}
 \cR_{\cL_{r,1},\cL_{1,s}}^2 =\theta_{\cL_{r,s}}\circ(\theta_{\cL_{r,1}}^{-1}\tens\theta_{\cL_{1,s}}^{-1}) = e^{2\pi i(h_{r,s}-h_{r,1}-h_{1,s})}=e^{\pi i (r+s-rs-1)},
\end{equation*}
which is not trivial if $r,s\in2\ZZ$.
\end{rem}

\section{Connections between Virasoro and triplet vertex operator algebras}

In this section, we show how to obtain basic results in the representation theory of triplet vertex operator algebras $\cW(p)$ using extension theory \cite{HKL, CKM, CMY1} applied to the Virasoro category $\cO_c^0$. Then, we show that the Virasoro category $\cO_c^0$ is braided tensor equivalent to the $PSL(2,\CC)$-equivariantization of the category of grading-restricted generalized $\cW(p)$-modules.

\subsection{Representation theory of triplet vertex operator algebras}

We have already used the vertex operator algebra embedding $V_c \subseteq \cW(p)$ in Section \ref{sec:fus_and_rig}, where $c=13-6p-6p^{-1}$ for $p > 1$ an integer. The triplet algebra $\cW(p)$ is $C_2$-cofinite \cite{AM_trip}, so by \cite{Hu_C2}, every grading-restricted generalized $\cW(p)$ module has finite length, the category $\cC_{\cW(p)}$ of grading-restricted generalized $\cW(p)$-modules has the vertex algebraic braided tensor category structure of \cite{HLZ8}, and every irreducible $\cW(p)$-module has a projective cover in $\cC_{\cW(p)}$. Two of these projective covers were constructed explicitly in \cite{AM_log_mods}, and the remaining ones were obtained in \cite{NT}. Fusion rules and rigidity of $\cC_{\cW(p)}$ were established in \cite{TW}. We now rederive these results as a straightforward consequence of the braided tensor category structure on $\cO_c^0$; we would especially like to emphasize that our tensor-categorical approach provides an alternative to the technical projective cover constructions in \cite{NT}.

To begin, we recall from \cite{AM_trip} that $\cW(p)$ has $2p$ distinct irreducible modules, which we label $\cW_{r,s}$ for $r=1,2$ and $1\leq s\leq p$, with $\cW_{1,1}$ isomorphic to $\cW(p)$ itself. As $V_c$-modules,
\begin{equation}\label{dec:triplet}
 \cW_{r,s}\cong\bigoplus_{n=0}^\infty (2n+r)\cdot\cL_{2n+r,s}.
\end{equation}
This means that every irreducible $\cW(p)$-module is an object in the direct limit completion $\ind(\cO_c)$, which consists of all generalized $V_c$-modules that are unions of their $C_1$-cofinite submodules. As shown in \cite{CMY1}, $\ind(\cO_c)$ has the vertex algebraic braided tensor category structure of \cite{HLZ8}, and thus $\cW(p)$ is a commutative algebra object in $\ind(\cO_c)$ \cite{HKL}. We can then define $\rep^0\cW(p)$ to be the category of generalized $\cW(p)$-modules which, as $V_c$-modules, are objects of $\ind(\cO_c)$. This category also has the vertex algebraic braided tensor category structure of \cite{HLZ8} (see \cite[Theorem 3.65]{CKM} and \cite[Theorem 7.7]{CMY1}). From Proposition 3.1.3 and Remark 3.1.4 of \cite{CMY2}, $\cC_{\cW(p)}$ is a subcategory of $\rep^0\cW(p)$; since $\cC_{\cW(p)}$ also has braided tensor category structure, it is a tensor subcategory of $\rep^0\cW(p)$.

We also have the category $\rep\cW(p)$ of not-necessarily-local $\cW(p)$-modules which, as $V_c$-modules, are objects of $\ind(\cO_c)$. The restriction functor $\cG:\rep\cW(p)\rightarrow\ind(\cO_c)$ which forgets the $\cW(p)$-action on objects of $\rep\cW(p)$ (but remembers the underlying $\cV ir$-module structure) has a left adjoint by \cite[Theorem 1.6(2)]{KO}; see also \cite[Lemma 7.8.12]{EGNO} or \cite[Lemma 2.61]{CKM}. The left adjoint $\cF$ is called induction, and it is a tensor functor; it is given on objects and morphisms of $\ind(\cO_c)$ by
\begin{align*}
 \cF(\cW) = \cW(p)\tens\cW,\qquad \cF(f) = \Id_{\cW(p)}\tens f.
\end{align*}
As $\cF$ and $\cG$ form an adjoint pair of functors, the Frobenius reciprocity relation
\begin{equation*}
 \hom_{\cW(p)}(\cF(\cW),\cX)\cong\hom_{V_c}(\cW,\cG(\cX))
\end{equation*}
holds for objects $\cW$ in $\ind(\cO_c)$ and $\cX$ in $\rep\cW(p)$. Moreover,
since the modules $\cL_{2n+1,1}$ appearing in the decomposition of $\cW(p)$ as a $V_c$-module are rigid, induction is an exact functor (see the similar \cite[Proposition 3.2.2]{CMY2} and its proof).

\begin{lem}\label{lem:ind_functor}
Induction restricts to a functor  $\cF:\cO_c^0 \rightarrow \cC_{\cW(p)}$.
\end{lem}
\begin{proof}
From the $r=s=1$ case of \eqref{dec:triplet} and naturality of the braiding,
\begin{align*}
\cR_{\cW, \cW(p)}\circ \cR_{\cW(p), \cW} & = \bigoplus_{n =0}^{\infty} (2n+1)\cdot\cR_{\cW, \cL_{2n+1}} \circ \cR_{\cL_{2n+1}, \cW}\nonumber\\
&= \bigoplus_{n =0}^{\infty}(2n+1)\cdot \Id_{\cL_{2n+1}\boxtimes \cW}
 =\Id_{\cW(p)\boxtimes \cW}
\end{align*}
if $\cW$ is in $\cO_c^0$. Then \cite[Theorem~2.65]{CKM} implies $\cF(\cW)$ is an object of $\rep^0\cW(p)$. Also, finite-length modules in $\cO_c$ induce to finite-length modules in $\rep\cW(p)$ because induction is exact and because simple modules in $\cO_c$ induce to finite-length $\cW(p)$-modules, as we will compute in Proposition \ref{prop:inductions} below. In particular, modules in $\cO_c^0$ induce to finite-length modules in $\rep^0\cW(p)$, which are necessarily grading restricted and thus are in $\cC_{\cW(p)}$.
\end{proof}

\begin{rem}
 Our definition of $\cO_c^0$ was chosen so that $\cO_c^0$ is precisely the subcategory of modules in $\cO_c$ that induce to local $\cW(p)$-modules (in $\rep^0\cW(p)$).
\end{rem}

We now compute the inductions of simple $V_c$-modules. First we need the following lemma, which is just basic algebra:
\begin{lem}
 Suppose $\cX$ is an object of $\rep\cW(p)$ and $\cW$ is an irreducible $\cW(p)$-module such that $\dim\hom_{\cW(p)}(\cX,\cW)<\infty$. Then there is a surjective $\cW(p)$-homomorphism $\cX\rightarrow\hom_{\cW(p)}(\cX,\cW)^*\otimes\cW$.
\end{lem}
\begin{proof}
 Let $\lbrace f_i\rbrace_{i=1}^I$ be a basis of $\hom_{\cW(p)}(\cX,\cW)$ and let $\lbrace f_i^*\rbrace_{i=1}^I$ be the corresponding dual basis of $\hom_{\cW(p)}(\cX,\cW)^*$. Then we have the $\cW(p)$-homomorphism
 \begin{align*}
  F: \cX & \rightarrow \hom_{\cW(p)}(\cX,\cW)^*\otimes\cW\nonumber\\
  b & \mapsto \sum_{i=1}^I f_i^*\otimes f_i(b).
 \end{align*}
To show that $F$ is surjective, note that the cokernel $\coker F=\hom_{\cW(p)}(\cX,\cW)^*\otimes\cW/\im F$ is isomorphic to a finite direct sum of copies of $\cW$ (since $\cW$ is irreducible), so $F$ is surjective if and only if $\hom_{\cW(p)}(\coker F,\cW)=0$.

Thus suppose $g\in\hom_{\cW(p)}(\coker F,\cW)$; it is enough to show that $g\circ q=0$, where $q: \hom_{\cW(p)}(\cX,\cW)^*\otimes\cW\rightarrow\coker F$ is the natural quotient map. Now, because $\cW$ is irreducible, there is a linear isomorphism
\begin{align*}
 \hom_{\cW(p)}(\cX,\cW) & \rightarrow\hom_{\cW(p)}\big(\hom_{\cW(p)}(\cX,\cW)^*\otimes \cW,\cW\big)\nonumber\\
 f & \mapsto \left[f^*\otimes w\mapsto\langle f^*,f\rangle w\right]
\end{align*}
Thus $g\circ q$ has this form for some $f\in\hom_{\cW(p)}(\cX,\cW)$, and moreover, $g\circ q$ annihilates $\im F$. In other words,
\begin{align*}
 0 = (g\circ q)(F(b)) =\sum_{i=1}^I (g\circ q)\big(f_i^*\otimes f_i(b)\big) =\sum_{i=1}^I \langle f_i^*,f\rangle f_i(b)=f(b)
\end{align*}
for all $b\in\cX$. Thus $f=0$ and therefore $g\circ q=0$ as well, proving $F$ is surjective.
\end{proof}

\begin{prop}\label{prop:inductions}
 For $r\geq 1$ and $1\leq s\leq p$,
 \begin{equation}\label{ind:irred}
 \cF(\cL_{r,s})\cong r\cdot\cW_{\bar{r},s},
 \end{equation}
 where $r\cdot$ denotes the direct sum of $r$ copies and $\bar{r}=1$ or $2$ according as $r$ is even or odd.
\end{prop}
\begin{proof}
 By Frobenius reciprocity and \eqref{dec:triplet},
 \begin{equation*}
\dim\hom_{\cW(p)}(\cF(\cL_{r,s}),\cW_{\bar{r},s}) =\dim\hom_{V_c}(\cL_{r,s},\cG(\cW_{\bar{r},s})) = r,
 \end{equation*}
 so by the preceding lemma, there is a surjective homomorphism $F: \cF(\cL_{r,s})\rightarrow r\cdot\cW_{\bar{r},s}$. To show that $F$ is also injective, it is enough to show that $\cF(\cL_{r,s})$ and $r\cdot\cW_{\bar{r},s}$ are isomorphic as grading-restricted $V_c$-modules, since then they will have the same graded dimension. Indeed, using the fusion rules of Theorems \ref{thm:Lr1_fus_rules} and \ref{thm:Lr_Ls_fusion},
\begin{align*}
 \cG(\cF(\cL_{r,s})) & \cong \bigoplus_{n=0}^\infty (2n+1)\cdot(\cL_{2n+1,1}\tens\cL_{r,s})\nonumber\\
 &\cong\bigoplus_{n=0}^\infty\bigoplus_{k=0}^{\min(r-1,2n)} (2n+1)\cdot\cL_{2n+r-2k,s}.
\end{align*}
For any $m\in\NN$, we need to determine the multiplicity of $\cL_{2m+\bar{r},s}$ in this direct sum: we get $2n+1$ copies of $\cL_{2m+\bar{r}}$ for each $k=n-m+\frac{r-\bar{r}}{2}$ such that
\begin{equation*}
 0\leq n-m+\frac{r-\bar{r}}{2}\leq\min(r-1,2n),
\end{equation*}
that is,
\begin{equation*}
 \left\vert m-\frac{r-\bar{r}}{2}\right\vert\leq n\leq m-1+\frac{r+\bar{r}}{2}.
\end{equation*}
Thus for $m\leq\frac{r-\bar{r}}{2}$, the multiplicity of $\cL_{2m+\bar{r},s}$ is
\begin{align*}
 \sum_{i=0}^{2m+\bar{r}-1} & \left[2\left(-m+\frac{r-\bar{r}}{2}+i\right)+1\right]\nonumber\\  &=(2m+\bar{r})(-2m+r-\bar{r}+1)+2\cdot\frac{(2m+\bar{r}-1)(2m+\bar{r})}{2} =r\cdot(2m+\bar{r}),
\end{align*}
while for $m\geq\frac{r-\bar{r}}{2}$, the multiplicity of $\cL_{2m+\bar{r},s}$ is
\begin{equation*}
 \sum_{i=0}^{r-1} \left[2\left(m-\frac{r-\bar{r}}{2}+i\right)+1)\right] =r\cdot\left(2m-r+\bar{r}+1\right)+2\cdot\frac{(r-1)r}{2} =r\cdot(2m+\bar{r}).
\end{equation*}
We conclude that
\begin{equation*}
 \cG(\cF(\cL_{r,s}))\cong r\cdot\bigoplus_{m=0}^\infty (2m+\bar{r})\cdot\cL_{2m+\bar{r},s}\cong \cG(r\cdot\cW_{\bar{r},s})
\end{equation*}
as required, where the last isomorphism comes from \eqref{dec:triplet}.
\end{proof}

Now we use Proposition \ref{prop:inductions} together with the fusion rules \eqref{fr1} and \eqref{fr2} for irreducible $V_c$-modules to determine fusion rules of irreducible $\cW(p)$-modules, previously proved in \cite{TW}:
\begin{thm}\label{TW}
\begin{itemize}
\item[(1)] The $\cW(p)$-module $\cW_{2,1}$ is a self-dual simple current with
\begin{equation}\label{walgebra: fr2}
\cW_{2,1}\boxtimes \cW_{r,s} \cong \cW_{3-r, s}
\end{equation}
for $r=1,2$ and $1\leq s\leq p$.
\item[(2)] The $\cW(p)$-module $\cW_{1,2}$ has fusion rules
\begin{equation}\label{walgebra: fr3}
\cW_{1,2}\boxtimes \cW_{r,s} \cong
\begin{cases}
\cW_{r, 2} \;\;\; &\mbox{if}\;\;\; s = 1\\
\cW_{r, s-1}\oplus \cW_{r, s+1}\;\;\; &\mbox{if}\;\;\; 2 \leq s \leq p-1
\end{cases}
\end{equation}
for $r=1,2$ and $1\leq s\leq p-1$.
\end{itemize}
\end{thm}
\begin{proof}
We use the fact that induction is a monoidal functor. For \eqref{walgebra: fr2}, we have
\begin{align*}
 2r\cdot(\cW_{2,1}\tens\cW_{r,s}) & \cong \cF(\cL_{2,1})\tens\cF(\cL_{r,s}) \cong\cF(\cL_{2,1}\tens\cL_{r,s})\nonumber\\
 & \cong\left\lbrace\begin{array}{lll}
                     \cF(\cL_{2,s}) & \text{if} & r=1\\
                     \cF(\cL_{1,s})\oplus\cF(\cL_{3,s}) & \text{if} & r=2\\
                    \end{array}\right. \nonumber\\
                    & \cong\left\lbrace\begin{array}{lll}
                     2\cdot\cW_{2,s} & \text{if} & r=1\\
                     (1+3)\cdot\cW_{1,s} & \text{if} & r=2\\
                    \end{array}\right. \cong 2r\cdot\cW_{3-r,s}.
\end{align*}
From this we see that $\cW_{3-r,s}$ can be the only composition factor of $\cW_{2,1}\tens\cW_{r,s}$, occurring with multiplicity $1$. The proof of \eqref{walgebra: fr3}, using \eqref{fr2}, is similar.
\end{proof}

The category $\cC_{\cW(p)}$ also inherits rigidity from $\cO_c^0$:
\begin{thm}\label{thm:CW(p)_rigid}
The category $\cC_{\cW(p)}$ is rigid.
\end{thm}
\begin{proof}
Since $\cC_{\cW(p)}$ is the category of finite-length $\cW(p)$-modules, it is closed under contragredients and \cite[Theorem~4.4.1]{CMY2} implies that it is enough to prove that simple $\cW(p)$-modules are rigid. But this holds because by Proposition \ref{prop:inductions}, every simple $\cW(p)$-module is a summand of the induction of a rigid $V_c$-module (see for example \cite[Lemma 1.16]{KO}, \cite[Exercise 2.10.6]{EGNO}, or \cite[Proposition 2.77]{CKM}).
\end{proof}

Next, we use fusion rules and rigidity in $\cC_{\cW(p)}$  to obtain all projective covers of simple modules in $\cC_{\cW(p)}$; these modules have been constructed previously in \cite{AM_log_mods, NT}. The next proposition was obtained in \cite[Section 5.1]{NT} using results from \cite{AM_trip}, but we will repeat the proof for completeness:
\begin{prop}\label{prop:Wrp_proj}
 The simple $\cW(p)$-modules $\cW_{r,p}$ for $r=1,2$ are projective in $\cC_{\cW(p)}$. In particular, each $\cW_{r,p}$ is its own projective cover.
\end{prop}
\begin{proof}
As in the proof of Theorem \ref{projoflrp}, it is enough to show that all length-$2$ exact sequences
\begin{equation}\label{eqn:Wrp_proj_exact}
 0\longrightarrow \cW_{r',s'}\longrightarrow\cX\longrightarrow\cW_{r,p}\longrightarrow 0
\end{equation}
in $\cC_{\cW(p)}$ split. We first claim that $L_0$ acts semisimply on $\cX$ if $(r',s')\neq(r,p)$. This is because the nilpotent part $L_0^{nil}$ of $L_0$ is a $\cW(p)$-module endomorphism of $\cX$ such that $\mathrm{Im}\, L_0^{nil}\subseteq\cW_{r',s'} \subseteq\ker L_0^{nil}$. Thus $L_0^{nil}\neq 0$ would imply 
\begin{equation*}
 \cW_{r',s'}\cong\mathrm{Im}\,L_0^{nil}\cong\cX/\ker L_0^{nil}\cong\cW_{r,p},
\end{equation*}
which contradicts the assumption that $\cW_{r',s'}$ and $\cW_{r,p}$ are non-isomorphic. Now because $\cX$ is a non-logarithmic module and because the irreducible modules $\cW_{r',s'}$, $\cW_{r,p}$ lie in different Virasoro blocks by \eqref{dec:triplet}, the block decomposition of the category of non-logarithmic $\cW(p)$-modules proved in \cite[Theorem 4.4]{AM_trip} implies that \eqref{eqn:Wrp_proj_exact} splits.

It remains to consider the $(r',s')=(r,p)$ case of \eqref{eqn:Wrp_proj_exact}.  Let $A(\cW(p))$ be the Zhu algebra of $\cW(p)$; then the lowest conformal weight space $\cX_{[h_{r,p}]}$ is a two-dimensional self-extension of the irreducible $A(\cW(p))$-module $\cW_{[h_{r,p}]}$. By \cite[Theorem 5.9]{AM_trip}, $A(\cW(p)) \cong I\times M_r(\CC)$ where $I$ is an ideal that acts trivially on $(\cW_{r,p})_{[h_{r,p}]}$ (and any of its self-extensions) and $M_r(\CC)$ is the simple $r\times r$ matrix algebra. Thus $\cX_{[h_{r,p}]}$ is a semisimple $A(\cW(p))$-module that generates $\cX$ as a $\cW(p)$-module. This means that $\cX$ is a homomorphic image of $\overline{F}((\cW_{r,p})_{[h_{r,p}]})\oplus \overline{F}((\cW_{r,p})_{[h_{r,p}]})$, where for a finite-dimensional $A(\cW(p))$-module $M$, $\overline{F}(M)$ denotes the generalized Verma $\cW(p)$-module defined in \cite[Definition 2.7]{Li}. In particular, $\cX$ has to be the length-$2$ quotient $\cW_{r,p}\oplus\cW_{r,p}$ of the direct sum of generalized Verma $\cW(p)$-modules, and thus \eqref{eqn:Wrp_proj_exact} splits in this case as well.
\end{proof}

 To obtain the remaining projective covers, we define $\cR_{1,s}=\cF(\cP_{1,s})$ and then $\cR_{2,s}=\cW_{2,1}\boxtimes\cR_{1,s}$ for $1\leq s\leq p$. With this notation, $\cR_{r,p}\cong\cW_{r,p}$ for $r=1,2$. To show that the modules $\cR_{r,s}$ are projective, we will need their fusion products with $\cW_{2,1}$ and $\cW_{1,2}$ (see \cite[Proposition~38]{TW} where, however, the slightly different formula in the $p=2$ case is omitted):
\begin{prop}\label{prop:Wp_proj_fus}
For $r=1, 2$ and $1\leq s\leq p$,
\begin{equation}\label{walgebra:proj:fr1}
\cW_{2,1}\boxtimes \cR_{r,s} \cong \cR_{3-r,s},
\end{equation}
\begin{equation}\label{walgebra:proj:fr2}
\cW_{1,2}\boxtimes \cR_{r,s} \cong
\begin{cases}
\cR_{r,2} \oplus 2\cdot\cR_{3-r,p} &\mbox{if}\; s = 1\\
\cR_{r,s-1}\oplus \cR_{r,s+1} &\mbox{if}\; 2\leq s \leq p-2\\
\cR_{r,p-2} \oplus 2\cdot \cR_{r,p} &\mbox{if}\; s = p-1\\
\cR_{r,p-1} & \mbox{if}\; s=p
\end{cases} \quad\text{if}\quad p\geq 3,
\end{equation}
\begin{equation}\label{walgebra:proj:fr3}
 \cW_{1,2}\tens\cR_{r,s} \cong 
 \begin{cases}
  2\cdot\cR_{r,2}\oplus 2\cdot\cR_{3-r,2} &\mbox{if}\; s=1\\
  \cR_{r,1} &\mbox{if}\;s=2\\
 \end{cases}
\quad\text{if}\quad p=2.
\end{equation}
\end{prop}
\begin{proof}
The $r=1$ case of \eqref{walgebra:proj:fr1} is the definition of $\cR_{2,s}$, and then the $r=2$ case follows using $\cW_{2,1}\tens\cW_{2,1}\cong\cW_{1,1}$.

The $r=1$ cases of \eqref{walgebra:proj:fr2} and \eqref{walgebra:proj:fr3} follow from
\begin{equation*}
 \cW_{1,2}\tens\cR_{1,s}\cong\cF(\cL_{1,2})\tens\cF(\cP_{1,s})\cong\cF(\cL_{1,2}\tens\cP_{1,s})
\end{equation*}
together with \eqref{fr2}, \eqref{more1prs}, \eqref{more2prs}, and the formula
$$\cF(\cP_{2,s}) \cong \cF(\cL_{2,1}\boxtimes \cP_{1,s}) \cong 2\cdot (\cW_{2,1}\boxtimes \cR_{1,s})\cong 2\cdot \cR_{2,s}.$$
Then the $r=2$ cases follow by tensoring the $r=1$ cases with $\cW_{2,1}$ and applying \eqref{walgebra:proj:fr1}.
\end{proof}

Now we can show that the modules $\cR_{r,s}$ are projective covers:
\begin{thm}
For  $r=1,2$ and $1\leq s \leq p-1$, the $\cW(p)$-module $\cR_{r,s}$ is a projective cover of $\cW_{r,s}$ in $\cC_{\cW(p)}$ with Loewy diagram
\begin{equation*}
 \xymatrixrowsep{1pc}
\xymatrixcolsep{.75pc}
 \xymatrix{
 & & \cW_{r,s} & \\
 \cR_{r,s}: & \cW_{3-r,p-s} \ar[ru] & & \cW_{3-r,p-s} \ar[lu] \\
& & \cW_{r,s} \ar[lu] \ar[ru] & \\
 }
\end{equation*}
\end{thm}
\begin{proof}
We take $r=1$ first. Recall from Theorem \ref{thm:P1s_structure} that $\cP_{1,s}$ has Loewy diagram
\begin{equation*}
\xymatrixrowsep{1pc}
\xymatrixcolsep{.75pc}
 \xymatrix{
 & & \cL_{1,s} \\
 \cP_{1,s}: & \cL_{2,p-s} \ar[ru] & \\
& & \cL_{1,s} \ar[lu]\\
 }
\end{equation*}
Applying the exact functor $\cF$ to $\cP_{1,s}$ and using \eqref{ind:irred},
we see that $\cW_{1, s}$ and $\cW_{2, p-s}$ are the only composition factors of $\cR_{1,s}$, both occurring with multiplicity $2$. We also see that $\cW_{1,s}$ is a submodule of $\cR_{1,s}$, and there is a surjective $\cW(p)$-module map $\cR_{1,s} \rightarrow \cW_{1,s}$.

To determine the Loewy diagram of $\cR_{1,s}$, we first note that the fusion rules \eqref{moreprs} and the decomposition \eqref{dec:triplet} implies that
\[
\cG(\cR_{1,s}) \cong \bigoplus_{n=0}^{\infty}(2n+1)\cdot\cP_{2n+1,s}.
\]
Then by Frobenius reciprocity,
\begin{align*}
\dim\hom_{\cW(p)}(\cW_{1, s}, \cR_{1,s}) & = \dim\hom_{\cW(p)}(\cF(\cL_{1, s}), \cR_{1,s})\nonumber\\
&= \dim\hom_{V_c}\left(\cL_{1, s}, \bigoplus_{n=0}^{\infty}(2n+1)\cdot\cP_{2n+1,s}\right) = 1,
\end{align*}
while
\begin{align*}
2\cdot\dim\hom_{\cW(p)}(\cW_{2, p-s}, \cR_{1,s}) & = \dim\hom_{\cW(p)}(\cF(\cL_{2, p-s}), \cR_{1,s})\nonumber\\
&= \dim\hom_{V_c}\left(\cL_{2, p-s}, \bigoplus_{n=0}^{\infty}(2n+1)\cdot\cP_{2n+1,s}\right) = 0.
\end{align*}
From these, we see that ${\rm Soc}(\cR_{1,s}) = \cW_{1,s}$.

Next, if we apply the exact functor $\cF$ to the exact sequence
\begin{equation*}
 0\longrightarrow\cL_{2,p-s}\longrightarrow\cP_{1,s}/\cL_{1,s}\longrightarrow\cL_{1,s}\longrightarrow 0,
\end{equation*}
we get the exact sequence
\begin{equation*}
  0\longrightarrow2\cdot\cW_{2,p-s}\longrightarrow\cR_{1,s}/\cW_{1,s}\longrightarrow\cW_{1,s}\longrightarrow 0.
\end{equation*}
This sequence does not split because by exactness of induction and Frobenius reciprocity,
\begin{align*}
 \hom_{\cW(p)}(\cR_{1,s}/\cW_{1,s}, \cW_{2,p-s})&\cong\hom_{\cW(p)}(\cF(\cP_{1,s}/\cL_{1,s}),\cW_{2,p-s})\nonumber\\
 & \cong\hom_{V_c}\left(\cP_{1,s}/\cL_{1,s},\bigoplus_{n=0}^\infty (2n+2)\cdot\cL_{2n+2,p-s}\right) =0.
\end{align*}
Consequently, $\mathrm{Soc}(\cR_{1,s}/\cW_{1,s})=2\cdot\cW_{2,p-s}$, and we have verified the row structure of the Loewy diagram for $\cR_{1,s}$. Moreover, all four length-$2$ subquotients indicated in the Loewy diagram for $\cR_{1,s}$ are indecomposable because
\begin{equation*}
 \hom_{\cW(p)}(\cW_{2,p-s},\cR_{1,s})=\hom_{\cW(p)}(\cR_{1,s},\cW_{2,p-s})=0.
\end{equation*}
This completes the verification of the Loewy diagram for $r=1$.

The Loewy diagram for $\cR_{2,s}=\cW_{2,1}\tens\cR_{1,s}$ now follows from that of $\cR_{1,s}$ by \cite[Proposition 2.5]{CKLR} since $\cW_{2,1}$ is a simple current.

Next, the fusion rules \eqref{walgebra:proj:fr2} and \eqref{walgebra:proj:fr3} show that each $\cR_{r,s}$ for $1\leq s\leq p-1$ is a direct summand of $\cW_{1,2}\tens\cR_{r,s+1}$. Since $\cR_{r,p}\cong\cW_{r,p}$ is projective by Proposition \ref{prop:Wrp_proj}, and since the subcategory of projective objects in $\cC_{\cW(p)}$ is closed under direct summands and tensoring with rigid objects, it follows that each $\cR_{r,s}$ is projective. Then the same argument as in the proof of Proposition \ref{prop:Prp-1_proj_cover} shows that it is a projective cover of $\cW_{r,s}$.
\end{proof}

\subsection{The Virasoro category \texorpdfstring{$\cO_c^0$}{Oc0} as an equivariantization} 

Here we prove a relation between the $V_c$-module category $\cO_c^0$ and the $\cW(p)$-module category $\cC_{\cW(p)}$ that was conjectured in \cite[Conjecture 11.6]{Ne}. We recall from \cite[Theorem 2.3]{ALM} that the full automorphism group of $\cW(p)$ is $PSL(2,\CC)$ for any integer $p>1$. Moreover, the action of $PSL(2,\CC)$ on $\cW(p)$ is continuous in the sense that every finite-dimensional conformal weight space of $\cW(p)$ with the Euclidean topology is a continuous $PSL(2,\CC)$-module. The group $PSL(2,\CC)$ also acts on the category $\cC_{\cW(p)}$ of grading-restricted generalized $\cW(p)$-modules by
\begin{equation}\label{eqn:G-action}
 g\cdot(\cX,Y_\cX) = (\cX,Y_\cX(g^{-1}(\cdot),x))
\end{equation}
for $g\in PSL(2,\CC)$. Thus we can form the $PSL(2,\CC)$-equivariantization of the category $\cC_{\cW(p)}$, as defined for example in \cite[Section 2.7]{EGNO}, which consists of $PSL(2,\CC)$-equivariant objects in $\cC_{\cW(p)}$. We will show that $\cO_c^0$ is braided tensor equivalent to the $PSL(2,\CC)$-equivariantization of $\cC_{\cW(p)}$; the proof is a straightforward generalization of \cite[Theorem 4.17]{McR2} to infinite automorphism groups.

First, we recall a slightly variant definition of equivariantization that is more convenient for our purposes. We use \cite[Section 2.3]{McR2} as a reference, but note that there, equivariantizations of categories involving twisted modules for a superalgebra were considered. Here, we only need to consider untwisted modules for a vertex operator algebra, so the situation is simpler. Let $V$ be a vertex operator algebra, $G$ a complex reductive Lie group acting continuously on $V$ by automorphisms, and $\cC$ a braided tensor category of grading-restricted generalized $V$-modules. Assume also that $\cC$ is closed under the action of $G$ given by \eqref{eqn:G-action}.
\begin{defi}
 The \textit{$G$-equivariantization} $\cC^G$ of $\cC$ is the following category:
 \begin{itemize}
  \item Objects of $\cC^G$ are pairs $(\cX,Y_\cX;\varphi_\cX)$ where $(\cX,Y_\cX)$ is an object of $\cC$ and $\varphi_\cX: G\rightarrow GL(\cX)$ is a continuous group representation such that
  \begin{equation}\label{eqn:G-equiv_compat_cond}
   \varphi_\cX(g)\cdot Y_\cX(v,x)=Y_\cX(g\cdot v,x)\varphi_\cX(g)
  \end{equation}
for all $g\in G$.

\item Morphisms from $(\cX_1,Y_{\cX_1};\varphi_{\cX_1})$ to $(\cX_2,Y_{\cX_2};\varphi_{\cX_2})$ in $\cC^G$ consist of all $V\times G$-module homomorphisms $f:\cX_1\rightarrow\cX_2$, that is,
\begin{equation*}
 f\circ Y_{\cX_1}(v,x)=Y_{\cX_2}(v,x)\circ f\qquad\text{and}\qquad f\circ\varphi_{\cX_1}(g)=\varphi_{\cX_2}(g)\circ f
\end{equation*}
for all $v\in V$, $g\in G$.
 \end{itemize}
\end{defi}
\begin{rem}
 The compatibility condition \eqref{eqn:G-equiv_compat_cond} implies that each $\varphi_\cX(g)$ commutes with the action of $L_0$ on $\cX$. Thus the condition that $\varphi_\cX$ be continuous simply means that each finite-dimensional conformal weight space of $\cX$ is a continuous $G$-module.
\end{rem}

As explained for example in \cite[Section 2.3]{McR2}, $\cC^G$ is a braided tensor category in the setting that $G$ is a finite group and that all modules in $\cC$ are objects of a braided tensor category of modules for the $G$-fixed-point subalgebra $V^G\subseteq V$. The same constructions work when $G$ is infinite, but we need to make sure that the action of $G$ on a tensor product $\cX_1\tens\cX_2$ is continuous:
\begin{lem}\label{lem:G-equiv_tens_prod}
 If $V$-modules $\cX_1$ and $\cX_2$ are objects of $\cC^G$, then $\cX_1\tens\cX_2$ is also an object of $\cC^G$ with $G$-action $\varphi_{\cX_1\tens\cX_2}$ characterized by
 \begin{equation}\label{eqn:G-action_tens_prod_def}
  \varphi_{\cX_1\tens\cX_2}(g)\cdot\cY_\tens(b_1,x)b_2 =\cY_\tens(\varphi_{\cX_1}(g) b_1,x)\varphi_{\cX_2}(g) b_2
 \end{equation}
for $g\in G$, $b_1\in\cX_1$, and $b_2\in\cX_2$, where $\cY_\tens$ is the tensor product intertwining operator of type $\binom{\cX_1\tens\cX_2}{\cX_1\,\cX_2}$.
\end{lem}
\begin{proof}
 Using \eqref{eqn:G-equiv_compat_cond} for $\varphi_{\cX_1}$ and $\varphi_{\cX_2}$, it is easy to check that $$\cY_\tens(\varphi_{\cX_1}(g)(\cdot),x)\varphi_{\cX_2}(g): \cX_1\otimes \cX_2\rightarrow(\cX_1\tens \cX_2)[\log x]\lbrace x\rbrace$$
is an intertwining operator of type $\binom{g^{-1}\cdot(\cX_1\tens \cX_2)}{\cX_1\,\,\cX_2}$ for any $g\in G$. Thus the universal property of tensor products induces a unique $V$-module homomorphism
\begin{equation*}
 \varphi_{\cX_1\tens\cX_2}(g): \cX_1\tens\cX_2\rightarrow g^{-1}\cdot(\cX_1\tens\cX_2)
\end{equation*}
such that \eqref{eqn:G-action_tens_prod_def} holds. The definition of the vertex operator for $g^{-1}\cdot(\cX_1\tens\cX_2)$ implies that $\varphi_{\cX_1\tens\cX_2}(g)$ satisfies \eqref{eqn:G-equiv_compat_cond} for all $g\in G$. Moreover, \eqref{eqn:G-action_tens_prod_def} and the fact that $\varphi_{\cX_1}$ and $\varphi_{\cX_2}$ are group homomorphisms implies that $\varphi_{\cX_1\tens\cX_2}$ defines a homomorphism from $G$ to $GL(\cX_1\tens\cX_2)$.

We still need to check that the $G$-action on each finite-dimensional conformal weight space of $\cX_1\tens\cX_2$ is continuous. Recall that for $b_1\in\cX_1$, $b_2\in\cX_2$, $h\in\CC$, and $k\in\NN$, the coefficient of $x^{-h-1}(\log x)^k$ in $\cY_\tens(b_1,x_2)b_2$ is denoted by $(b_1)_{h,k}b_2$. Thus \eqref{eqn:G-action_tens_prod_def} implies that for each $h\in\CC$ and $k\in\NN$,
\begin{align*}
 \psi_{h,k}: \cX_1\otimes\cX_2 &\rightarrow \cX_1\tens\cX_2\nonumber\\
 b_1\otimes b_2 & \mapsto (b_1)_{h,k}b_2 
\end{align*}
is a $G$-module homomorphism. Then because $G$ is a complex reductive Lie group acting continuously on the finite-dimensional conformal weight spaces of $\cX_1$ and $\cX_2$, each of $\cX_1$ and $\cX_2$, and thus $\cX_1\otimes\cX_2$ as well, decomposes as the direct sum of finite-dimensional irreducible continuous $G$-modules. Consequently, the image of each $\psi_{h,k}$ is a direct sum of finite-dimensional irreducible continuous $G$-modules. Finally, because $\cY_\tens$ is a surjective intertwining operator, $\cX_1\tens\cX_2$ is a continuous $G$-module.
\end{proof}

It is now easy to see using \eqref{eqn:G-action_tens_prod_def} and the vertex algebraic tensor category structure on $\cC$ (see \cite{HLZ8} or the exposition in \cite[Section 3.3]{CKM}) that $\cC^G$ is a tensor category. For example, if $f_1:\cX_1\rightarrow\til{\cX}_1$ and $f_2:\cX_2\rightarrow\til{\cX}_2$ are morphisms in $\cC^G$, then the $V$-module homomorphism $f_1\tens f_2: \cX_1\tens\cX_2\rightarrow\til{\cX}_1\tens\til{\cX}_2$ is also a $G$-module homomorphism because
\begin{align*}
 (f_1\tens f_2)\left(\varphi_{\cX_1\tens\cX_2}(g)\cdot\cY_\tens(b_1,x)b_2\right) & =\cY_\tens(f_1(\varphi_{\cX_1}(g)b_1),x)f_2(\varphi_{\cX_2}(g)b_2)\nonumber\\
 & =\cY_\tens(\varphi_{\til{X}_1}(g)f_1(b_1),x)\varphi_{\til{X}_2}(g)f_2(b_2)\nonumber\\
 & =\varphi_{\til{X}_1\tens\til{X}_2}(g)\cdot(f_1\tens f_2)\left(\cY_\tens(b_1,x)b_2\right)
\end{align*}
for $b_1\in\cX_1$, $b_2\in\cX_2$, and $g\in G$. The unit object of $\cC^G$ is $V$ with $G$-action $\varphi_V(g)=g$ for $g\in G$. Then \eqref{eqn:G-equiv_compat_cond}, \eqref{eqn:G-action_tens_prod_def}, and the definitions of the structure isomorphisms in $\cC$ show that the unit, associativity, and braiding isomorphisms in $\cC$ all commute with the $G$-actions on objects of $\cC^G$ and thus define braided tensor category structure on $\cC^G$.

Now take $V=\cW(p)$, $G=PSL(2,\CC)$, and $\cC=\cC_{\cW(p)}$. In this case, recall that Lemma \ref{lem:ind_functor} and the rigidity of $\cO_c^0$ show that induction defines an exact functor $\cF: \cO_c^0\rightarrow\cC_{\cW(p)}$. But just as explained in \cite[Section 2.3]{McR2}, induction actually defines a functor into the $PSL(2,\CC)$-equivariantization:
\begin{lem}\label{lem:ind_BTF}
 Induction defines an exact braided tensor functor $\cF: \cO_c^0\rightarrow(\cC_{\cW(p)})^{PSL(2,\CC)}$.
\end{lem}
\begin{proof}
 For an object $\cW$ in $\cO_c^0$, recall that $\cF(\cW)=\cW(p)\tens\cW$ as a generalized $V_c$-module, where $\tens$ is the tensor product in $\ind(\cO_c)$. Thus $\cF(\cW)$ admits the $PSL(2,\CC)$-action
 \begin{equation*}
  \varphi_{\cF(\cW)}(g)=g\tens\Id_\cW
 \end{equation*}
for $g\in PSL(2,\CC)$. Just as in \cite[Section 2.3]{McR2}, $\varphi_{\cF(\cW)}(g)$ satisfies \eqref{eqn:G-equiv_compat_cond} because $g$ is an automorphism of $\cW(p)$, but we need to check that $\varphi_{\cF(\cW)}$ is continuous. As in the proof of Lemma \ref{lem:G-equiv_tens_prod}, we have the tensor product intertwining operator
\begin{align*}
 \cY_\tens: \cW(p)\otimes\cW & \rightarrow (\cW(p)\tens\cW)[\log x]\lbrace x\rbrace\nonumber\\
 v\otimes w & \mapsto \cY_\tens(v,x)w=\sum_{h\in\CC}\sum_{k\in\NN} v_{h,k} w\,x^{-h-1}(\log x)^k,
\end{align*}
and for any $h\in\CC$, $k\in\NN$, we have a $G$-module homomorphism
\begin{align*}
 \psi_{h,k}: \cW(p)\otimes\cW & \rightarrow\cW(p)\tens\cW\nonumber\\
 v\otimes w & \mapsto v_{h,k} w
\end{align*}
where $\cW$ is a trivial $G$-module. Since $\cW(p)\otimes\cW$ is a direct sum of finite-dimensional irreducible continuous $G$-modules, the same is true of the image of each $\psi_{h,k}$. Thus because $\cY_\tens$ is a surjective intertwining operator, $\cF(\cW)$ is a continuous $G$-module.

We have now shown that $(\cF(\cW);\varphi_{\cF(\cW)})$ is an object of $(\cC_{\cW(p)})^{PSL(2,\CC)}$, and it is clear that if $f:\cW_1\rightarrow\cW_2$ is a homomorphism in $\cO_c^0$, then $\cF(f)=\Id_{\cW(p)}\tens f$ is also a homomorphism of $G$-modules and thus a morphism in $(\cC_{\cW(p)})^{PSL(2,\CC)}$. Thus induction defines a functor $\cF: \cO_c^0\rightarrow(\cC_{\cW(p)})^{PSL(2,\CC)}$ which is exact because $\cO_c^0$ is rigid, as mentioned previously. See \cite[Theorems 2.11 and 2.12]{McR2} for a proof that $\cF$ is additionally a braided tensor functor. Note that these results in \cite{McR2} do not require the group to be finite and that the braided tensor category structure on $(\cC_{\cW(p)})^{PSL(2,\CC)}$ defined in this subsection indeed agrees with that in \cite[Section 2.3]{McR2}. 
\end{proof}

Now we can prove the main result of this section, that $\cF$ is actually an equivalence of braided tensor categories. The proof largely repeats that of \cite[Theorem 4.17]{McR2}, but some additional care is needed because $PSL(2,\CC)$ is an infinite group:
\begin{thm}\label{thm:G-equiv}
The induction functor $\cF: \cO_c^0\rightarrow(\cC_{\cW(p)})^{PSL(2,\CC)}$ is an equivalence of braided tensor categories.
\end{thm}
\begin{proof} 
For notational simplicity, set $G = PSL(2, \CC)$. Since $\cF$ is a braided tensor functor by Lemma \ref{lem:ind_BTF}, we just need to show it is an equivalence of categories. Thus, we will show there is a $G$-invariants functor $\cI: (\cC_{\cW(p)})^G\rightarrow\cO_c^0$ such that $\cI\circ\cF$ and $\cF\circ\cI$ are naturally isomorphic to the respective identity functors.

For an object $(\cX, Y_\cX; \varphi_\cX)$ in $(\cC_{\cW(p)})^{G}$, define
\[
\cX^{G} = \{b \in \cX\,|\,\varphi_\cX(g)b = b\; \mathrm{for\; all}\; g \in G\}.
\]
By \eqref{eqn:G-equiv_compat_cond}, for any $g\in G$, $v\in V_c=\cW(p)^G$, and $b\in\cX^G$, we have
\[
\varphi_\cX(g)\cdot Y_\cX(v,x)b = Y_\cX(g\cdot v,x)\varphi_\cX(g)b= Y_\cX(v, x)b,
\]
and it follows that $\cX^{G}$ is a $V_{c}$-submodule of $\cX$. 
Then since objects of $\cC_{\cW(p)}$ are modules in $\ind(\cO_c)$ when viewed as $V_c$-modules  (see Proposition 3.1.3 and Remark 3.1.4 of \cite{CMY2}) and since $\ind(\cO_c)$ is closed under submodules, $\cX^G$ is a module in $\ind(\cO_c)$.

For a morphism $f: (\cX_1, Y_{\cX_1}; \varphi_{\cX_1}) \rightarrow (\cX_2, Y_{\cX_2}; \varphi_{\cX_2})$ in $(\cC_{\cW(p)})^{G}$, define $f^G = f|_{\cX_1^{G}}$. Since $f$ intertwines the $G$-actions on $\cX_1$ and $\cX_2$, the image of $f^G$ is contained in $\cX_2^G$ and hence
    \[
    f^G: \cX_1^G \rightarrow \cX_2^G
    \]
    is a morphism in $\ind(\cO_c)$. Thus we have a $G$-invariants functor $\cI: \cC_{\cW(p)}\rightarrow\ind(\cO_c)$, and we will show below that the image of $\cI$ is actually contained in $\cO_c^0$.

Now we show that for a module $\cW \in \cO_c^0$, we have a natural isomorphism $\cF(\cW)^{G} \cong \cW$. Since $\cW(p)$ is a semisimple $G$-module, there is a $V_c$-module projection $\varepsilon_{\cW(p)}: \cW(p) \rightarrow \cW(p)^G$ that is a one-sided inverse to the inclusion $\iota_{\cW(p)}: \cW(p)^G \rightarrow \cW(p)$. Then recalling that $\cF(\cW) = \cW(p) \boxtimes \cW$ and $\varphi_{\cF(\cW)}(g) = g \boxtimes \Id_\cW$ for $g \in G$, we see that
\[
\cF(\cW)^G\hookrightarrow\cW(p)\tens\cW\xrightarrow{\varepsilon_{\cW(p)}\tens\Id_\cW} \cW(p)^G\tens\cW\xrightarrow{l_\cW} \cW
\]
is a natural isomorphism, with inverse $(\iota_{\cW(p)} \boxtimes \Id_W)\circ l_\cW^{-1}$, just as in the proof of \cite[Theorem 4.17]{McR2}.

Next, for $(\cX, Y_\cX; \varphi_\cX)$ in $(\cC_{\cW(p)})^{G}$, recall that $\cF(\cX^G)$ is at first an object of the category $\rep\cW(p)$ of not-necessarily-local $\cW(p)$-modules which are objects of $\ind(\cO_c)$ when viewed as $V_c$-modules. Moreover, as in the proof of Lemma \ref{lem:ind_BTF}, $\cF(\cX^G)$ is a semisimple $G$-module:
\begin{equation*}
 \cF(\cX^G)\cong\bigoplus_{\chi\in\widehat{G}} \cW(p)_\chi\tens\cX^G,
\end{equation*}
where the sum runs over the finite-dimensional irreducible continuous characters of $G$ and $\cW(p)_\chi$ is the isotypical component of $\cW(p)$ corresponding to $\chi$. We have a similar decomposition $\cX=\bigoplus_{\chi\in\widehat{G}} \cX_\chi$ because by assumption, $G$ acts continuously on the finite-dimensional conformal weight spaces of $\cX$.

Now we show that we have a natural isomorphism $\cF(\cX^G)\cong\cX$. Let $\iota_\cX: \cX^G \rightarrow \cX$ denote the inclusion, and let $\mu_\cX: \cW(p)\tens\cX\rightarrow\cX$ denote the unique $V_c$-module homomorphism induced by the intertwining operator $Y_\cX$. Then just as in the proof of \cite[Theorem 4.17]{McR2}, 
\[
\Psi_\cX = \mu_\cX\circ(\Id_{\cW(p)} \boxtimes \iota_\cX): \cW(p) \boxtimes \cX^G \rightarrow \cX
\]
is a $\cW(p)\times G$-module homomorphism, and $\Psi_\cX$ defines a natural transformation from $\cF\circ\cI$ to the inclusion of $(\cC_{\cW(p)})^G$ into the equivariantization $(\rep\cW(p))^G$. Moreover, since $\Psi_\cX$ is a $G$-module homomorphism, it restricts to a map $\cF(\cX^G)^G\rightarrow\cX^G$ which is an isomorphism because it amounts to the left unit isomorphism $\cW(p)^G\tens\cX^G\rightarrow\cX^G$ in $\ind(\cO_c)$. Thus the kernel and cokernel of $\Psi_\cX$ are both $\cW(p)\times G$-modules in $(\rep\cW(p))^G$ with no $G$-invariants, and both are semisimple as $G$-modules because $\cF(\cX^G)$ and $\cX$ are semisimple $G$-modules. Then the argument that concludes the proof of \cite[Theorem 4.17]{McR2} applies to show that the kernel and cokernel of $\Psi_\cX$ are both $0$, so that $\Psi_\cX$ is an isomorphism.

It still remains to show that if $(\cX,Y_\cX;\varphi_\cX)$ is in $(\cC_{\cW(p)})^G$, then $\cX^G$ is in $\cO_c^0$. It is enough to show that $\cX^G$ has finite length, which is equivalent to showing that any decreasing $V_c$-submodule sequence
\[
\cX^G\supseteq\cW_0 \supseteq \cW_1 \supseteq \cdots \supseteq \cW_n \supseteq \cdots
\]
and any increasing sequence
\[
\cW_0 \subseteq \cW_1 \subseteq \cdots \subseteq \cW_n \subseteq \cdots \subseteq \cX^G
\]
are stationary, that is, $\cW_n = \cW_{n+1}$ for $n$ sufficiently large (\cite[Theorem~2.1]{S}; see also \cite[Exercise~8.20]{KS}). Applying the exact functor $\cF$ to the decreasing sequence yields a decreasing sequence of $\cW(p)$-submodules in $\cF(\cX^G)\cong\cX$. Because $\cX$ has finite length, $\cF(\cW_{n}) = \cF(\cW_{n+1})$ for $n$ sufficiently large, which means $\cF(\cW_n/\cW_{n+1}) = 0$ since $\cF$ is exact. Moreover, since $\cW(p)$ is a semisimple $V_c$-module, $\cF(\cW_n/\cW_{n+1})$ contains $V_c\tens(\cW_n/\cW_{n+1})\cong\cW_n/\cW_{n+1}$ as a $V_c$-submodule. So $\cW_n/\cW_{n+1}=0$ for $n$ sufficiently large. Similarly, the increasing series is also stationary. 
\end{proof}

\begin{rem}
 Consider the one-dimensional torus $T^{\vee}\subseteq PSL(2,\CC)$. The fixed-point subalgebra $\cW(p)^{T^\vee}$ is the singlet vertex operator algebra $\cM(p)$, whose tensor categories were studied in \cite{CMY2}. Then similar arguments as above show that induction yields a braided tensor equivalence from the $\cM(p)$-module category $\cC_{\cM(p)}^0$ defined in \cite{CMY2} to the $T^\vee$-equivariantization of $\cC_{\cW(p)}$. In a little more detail:
 \begin{itemize}
  \item The definition of $\cC_{\cM(p)}^0$ in \cite[Definition 3.1.2]{CMY2}, combined with \cite[Proposition 3.2.2]{CMY2} and the argument in the proof of Lemma \ref{lem:ind_BTF}, shows that induction defines an exact braided tensor functor $\cF: \cC_{\cM(p)}^0\rightarrow(\cC_{\cW(p)})^{T^\vee}$.
  
  \item Taking $T^\vee$-invariants yields a functor $\cI$ from $(\cC_{\cW(p)})^{T^\vee}$ to the category $\rep^0\cM(p)$ of generalized $\cM(p)$-modules which are objects of $\ind(\cO_c)$ when viewed as $V_c$-modules. Induction extends to an exact functor from $\rep^0\cM(p)$ to the $T^\vee$-equivariantization of the category $\rep\cW(p)$ of not-necessarily-local $\cW(p)$-modules which are objects of $\rep^0\cM(p)$ when viewed as generalized $\cM(p)$-modules.
  
  \item Because $\cW(p)$ is a semisimple $\cM(p)\times T^\vee$-module, and because objects of $(\cC_{\cW(p)})^{T^\vee}$ are semisimple $T^\vee$-modules, the arguments in the proof of Theorem \ref{thm:G-equiv} show that $\cI\circ\cF$ is naturally isomorphic to the identity on $\cC_{\cM(p)}^0$, and that $\cF\circ\cI$ is naturally isomorphic to the inclusion of $(\cC_{\cW(p)})^{T^\vee}$ into $(\rep\cW(p))^{T^\vee}$.
  
  \item Since for any module $\cX$ in $(\cC_{\cW(p)})^{T^\vee}$, $\cF(\cX^{T^\vee})\cong\cX$ is a $\cW(p)$-module in $\cC_{\cW(p)}$, $\cX^{T^\vee}$ is by definition an object of $\cC_{\cM(p)}^0$. Thus the image of $\cI$ is actually $\cC_{\cM(p)}^0$.
 \end{itemize}
Note that \cite[Conjecture 11.6]{Ne} predicted that taking $T^\vee$-invariants should yield an embedding of $(\cC_{\cW(p)})^{T^\vee}$ into the category of $\cM(p)$-modules. Thus the above argument proves a strong form of this conjecture: $\cI$ in fact yields a braided tensor equivalence with the specific subcategory $\cC_{\cM(p)}^0$ of $\cM(p)$-modules.
\end{rem}


\begin{thebibliography}{CJORY}

\bibitem[Ad]{A}
D. Adamovi\'{c}, Classification of irreducible modules of certain subalgebras of free boson vertex algebra, \textit{J. Algebra} \textbf{270} (2003), no. 1, 115--132.

\bibitem[ACGY]{ACGY}
D. Adamovi\'{c}, T.~Creutzig, N.~Genra and J.~Yang,
The vertex algebras $\mathcal R^{(p)}$ and $\mathcal V^{(p)}$, \textit{Comm. Math. Phys.} \textbf{383} (2021), no. 2, 1207--1241.

\bibitem[ALM]{ALM}
D. Adamovi\'{c}, X. Lin and A. Milas, $ADE$ subalgebras of the triplet vertex algebra $\cW(p)$: $A$-series, \textit{Commun. Contemp. Math.} \textbf{15} (2013), no. 6, 1350028, 30 pp.

\bibitem[AM1]{AM_log_intw}
D. Adamovi\'{c} and A. Milas, Logarithmic intertwining operators and $\cW(2,2p-1)$
algebras, \textit{J. Math. Phys.} \textbf{48} (2007), no. 7, 073503, 20 pp.

\bibitem[AM2]{AM_trip}
D. Adamovi\'{c} and A. Milas, On the triplet vertex algebra $\cW(p)$, {\em Adv. Math.} \textbf{217} (2008), no. 6, 2664--2699.

\bibitem[AM3]{AM_log_mods}
D. Adamovi\'{c} and A. Milas, Lattice construction of logarithmic modules for certain vertex
algebras, \textit{Selecta Math. (N.S.)} \textbf{15} (2009), no. 4, 535--561.

\bibitem[AM4]{AM_doub}
D. Adamovi\'{c} and A. Milas, The doublet vertex operator superalgebras $\cA(p)$ and $\cA_{2,p}$, \textit{Recent Developments in Algebraic and Combinatorial Aspects of Representation Theory}, 23--38, Contemp. Math., \textbf{602}, Amer. Math. Soc., Providence, RI, 2013.

\bibitem[AP]{AnP}
H. Andersen and J. Paradowski, Fusion categories arising from semisimple Lie algebras, {\em Comm. Math. Phy.} {\bf 169} (1995), no. 3, 563--588.

\bibitem[ACKR]{AuCKR}
J. Auger, T. Creutzig, S. Kanade and M. Rupert, Braided tensor categories related to $\mathcal{B}_p$ vertex algebras, \textit{Comm. Math. Phys.} \textbf{378} (2020), no. 1, 219--260.

\bibitem[BK]{BK} B. Bakalov and A. Kirillov, Jr., \textit{Lectures on Tensor Categories and Modular Functors}, University Lecture Series, \textbf{21}, American Mathematical Society, Providence, RI, 2001, x+221 pp.

\bibitem[BPZ]{BPZ}
A. Belavin, A. Polyakov and A. Zamolodchikov, Infinite conformal symmetry in two-dimensional quantum field theory, \textit{Nuclear Phys. B} \textbf{241} (1984), no. 2, 333--380.

\bibitem[BFGT]{BFGT}
P. Bushlanov, B. Feigin, A. Gainutdinov and I. Tipunin, Lusztig limit of quantum $sl(2)$ at root of unity and fusion of $(1,p)$ Virasoro logarithmic minimal models, \textit{Nuclear Phys. B} \textbf{818} (2009), no. 3, 179--195.

\bibitem[BGT]{BGT}
P. Bushlanov, A. Gainutdinov and I. Tipunin, Kazhdan-Lusztig equivalence and fusion of Kac modules in Virasoro logarithmic models, \textit{Nuclear Phys. B} \textbf{862} (2012), no. 1, 232--269.

\bibitem[CF]{CF}
N.~Carqueville and M.~Flohr, Nonmeromorphic operator product expansion and $C_2$-cofiniteness for a family of
$\cW$-algebras, \textit{J. Phys. A} \textbf{39} (2006), no. 4, 951--966.

\bibitem[CGR]{CGR}
T. Creutzig, A. Gainutdinov and I. Runkel, A quasi-Hopf algebra for the triplet vertex operator algebra, {\em Commun. Contemp. Math.} \textbf{22} (2020), no. 3, 1950024, 71 pp.

 \bibitem[CJORY]{CJORY}
T. Creutzig, C. Jiang, F. Orosz Hunziker, D. Ridout and J. Yang, Tensor categories arising from the Virasoro algebra, \textit{Adv. Math.} \textbf{380} (2021), 107601, 35 pp.

\bibitem[CKLR]{CKLR}
T. Creutzig, S. Kanade, A. Linshaw and D. Ridout, Schur-Weyl duality for Heisenberg cosets, {\em Transform. Groups} \textbf{24} (2019), no. 2, 301--354.

\bibitem[CKM]{CKM}
T. Creutzig, S. Kanade and R. McRae, Tensor categories for vertex operator superalgebra extensions, \textit{Mem. Amer. Math. Soc.} \textbf{295} (2024), no. 1472, vi+181 pp.

\bibitem[CMY1]{CMY1} T. Creutzig, R. McRae and J. Yang, Direct limit completions of vertex tensor categories, \textit{Commun. Contemp. Math.} \textbf{24} (2022), no. 2, Paper No. 2150033, 60 pp. 

\bibitem[CMY2]{CMY2}
T. Creutzig, R. McRae and J. Yang, On ribbon categories for singlet vertex algebras, \textit{Comm. Math. Phys.} \textbf{387} (2021),  no. 2, 865--925.

\bibitem[CMY3]{CMY3}
T. Creutzig, R. McRae and J. Yang, Tensor structure on the Kazhdan-Lusztig category for affine $\mathfrak{gl}(1|1)$, \textit{Int. Math. Res. Not. IMRN} 2022, no. 16, 12462--12515.

\bibitem[CMR]{CMR}
T. Creutzig, A. Milas and M. Rupert, Logarithmic link invariants of $\overline{U}^H_q(\mathfrak{sl}_2)$ and asymptotic dimensions of singlet vertex algebras, \textit{J. Pure Appl. Algebra} \textbf{222} (2018), no. 10, 3224--3247.

\bibitem[CRW]{CRW}
T.~Creutzig, D.~Ridout and S.~Wood, Coset constructions of logarithmic $(1,p)$ models, \textit{Lett. Math. Phys.} \textbf{104} (2014), no. 5, 553--583.

\bibitem[DLMF]{DLMF}
\textit{NIST Digital Library of Mathematical Functions}, http://dlmf.nist.gov/, Release 1.0.27 of 2020-06-15, F.  Olver, A. Olde Daalhuis, D. Lozier, B. Schneider, R. Boisvert, C. Clark, B. Miller, B. Saunders, H. Cohl and M. McClain, eds.

\bibitem[DLM]{DLM}
C. Dong, H. Li and G. Mason, Modular-invariance of trace functions in orbifold theory and generalized
Moonshine, {\em Comm. Math. Phys.} {\bf 214} (2000), no. 1, 1--56.

\bibitem[EGNO]{EGNO} P. Etingof, S. Gelaki, D. Nikshych and V. Ostrik, \textit{Tensor Categories},  Mathematical Surveys and Monographs, \textbf{205}, American Mathematical Society, Providence, RI, 2015, xvi+343 pp.

\bibitem[FFr]{FFr}
B. Feigin and E. Frenkel, Affine Kac-Moody algebras at the critical level and Gelfand-Dikii algebras, \textit{Infinite Analysis, Part A, B (Kyoto, 1991)}, 197--215, Adv. Ser. Math. Phys., \textbf{16}, World Sci. Publ., River Edge, NJ, 1992.

\bibitem[FFu]{FF}
B. Feigin and D. Fuchs, Representations of the Virasoro algebra, \textit{Representation of Lie Groups and Related Topics}, 465--554, Adv. Stud. Contemp. Math., \textbf{7}, Gordon and Breach, New York, 1990.

\bibitem[FGST1]{FGST1}
B.~Feigin, A.~Gainutdinov, A.~Semikhatov and I.~Tipunin,
Modular group representations and fusion in logarithmic conformal field theories and in the quantum group center, \textit{Comm. Math. Phys.} \textbf{265} (2006), no. 1, 47--93.

\bibitem[FGST2]{FGST2}
B.~Feigin, A.~Gainutdinov, A.~Semikhatov and I.~Tipunin, Logarithmic extensions of minimal models: characters and modular transformations, \textit{Nuclear Phys. B} \textbf{757} (2006), no. 3, 303--343.

\bibitem[Fi]{F}
M. Finkelberg, An equivalence of fusion categories, {\em Geom. Funct. Anal.} {\bf 6} (1996), no. 2, 249--267; erratum ibid. \textbf{23} (2013), no. 2, 810--811.

\bibitem[dFMS]{Fran}
P. Di Francesco, P. Mathieu and D. Sénéchal, \textit{Conformal Field Theory}, Grad. Texts Contemp. Phys.,
Springer-Verlag, New York, 1997, xxii+890 pp.

\bibitem[FB]{FB}
E. Frenkel and D. Ben-Zvi, \textit{Vertex Algebras and Algebraic Curves}, 2nd ed., Mathematical Surveys and Monographs, \textbf{88}, American Mathematical Society, Providence, RI, 2004. xiv+400 pp.

\bibitem[FHL]{FHL}
I. Frenkel, Y.-Z. Huang and J. Lepowsky, On axiomatic approaches to vertex operator algebras
and modules, \textit{Mem. Amer. Math. Soc.} \textbf{104} (1993), no. 494, viii+64 pp.

\bibitem[FZ1]{FZ1}
I.  Frenkel and Y. Zhu, Vertex operator algebras associated to representations of affine and Virasoro
algebras, {\em Duke Math. J.} {\bf 66} (1992), no. 1, 123--168.

\bibitem[FZ2]{FZ2}
I.~Frenkel and M.~Zhu, Vertex algebras associated to modified regular representations of the Virasoro algebra, {\em Adv. Math.} {\bf 229} (2012), no. 6, 3468--3507.

\bibitem[FHST]{FHST}
J.~Fuchs, S.~Hwang, A.~Semikhatov and I.~Tipunin, Nonsemisimple fusion algebras and the Verlinde formula,
\textit{Comm. Math. Phys.} \textbf{247} (2004), no. 3, 713--742.

\bibitem[GaK]{GaK}
M. Gaberdiel and H. Kausch, Indecomposable fusion products, \textit{Nuclear Phys. B} \textbf{477} (1996), no. 1, 293--318.

\bibitem[GR]{GR}
M. Gaberdiel and I. Runkel, From boundary to bulk in logarithmic CFT, \textit{J. Phys. A} \textbf{41} (2008), no. 7, 075402, 29 pp.

\bibitem[GN]{GN}
T. Gannon and C. Negron, Quantum $SL(2)$ and logarithmic vertex operator algebras at $(p, 1)$-central charge, \textit{J. Eur. Math. Soc. (JEMS)} (2024), DOI: 10.4171/JEMS/1489.

\bibitem[GoK]{GK}
M. Gorelik and V. Kac, On complete reducibility for infinite-dimensional Lie algebras, \textit{Adv. Math.} \textbf{226} (2011), no. 2, 1911--1972.

\bibitem[Hu1]{Hu_Vir_tens}
Y.-Z. Huang, Virasoro vertex operator algebras, the (nonmeromorphic) operator product expansion and the tensor product theory, \textit{J. Algebra} \textbf{182} (1996), no. 1, 201--234.

\bibitem[Hu2]{Hu_rigid}
Y.-Z. Huang, Rigidity and modularity of vertex tensor categories, {\it Commun. Contemp. Math.} {\bf 10} (2008), suppl. 1, 871--911.

\bibitem[Hu3]{Hu_C2}
Y.-Z. Huang, Cofiniteness conditions, projective covers and the logarithmic tensor product theory, {\em J. Pure Appl. Algebra}, {\bf 213}, (2009), no. 4, 458--475.

\bibitem[HKL]{HKL}
Y.-Z. Huang, A. Kirillov, Jr. and J. Lepowsky, Braided tensor categories and extensions of vertex operator algebras, \textit{Comm. Math. Phys.} \textbf{337} (2015), no. 3, 1143--1159.

\bibitem[HLZ1]{HLZ1}
Y.-Z. Huang, J. Lepowsky and L. Zhang, Logarithmic tensor category theory for generalized modules for a
conformal vertex algebra, I: Introduction and strongly graded algebras and their generalized modules, \textit{Conformal Field Theories and Tensor Categories}, 169--248, Math. Lect. Peking Univ., Springer, Heidelberg, 2014.
	
\bibitem[HLZ2]{HLZ2}
Y.-Z. Huang, J. Lepowsky and L. Zhang, Logarithmic tensor category theory for
generalized modules for a conformal vertex algebra, II: Logarithmic formal
calculus and properties of logarithmic intertwining operators, arXiv:1012.4196.
	
\bibitem[HLZ3]{HLZ3}
Y.-Z. Huang, J. Lepowsky and L. Zhang, Logarithmic tensor category theory for
generalized modules for a conformal vertex algebra, III: Intertwining maps and
tensor product bifunctors, arXiv:1012.4197.
	
\bibitem[HLZ4]{HLZ4}
Y.-Z. Huang, J. Lepowsky and L. Zhang, Logarithmic tensor category theory for
generalized modules for a conformal vertex algebra, IV: Constructions of tensor
product bifunctors and the compatibility conditions, arXiv:1012.4198.
	
\bibitem[HLZ5]{HLZ5}
Y.-Z. Huang, J. Lepowsky and L. Zhang, Logarithmic tensor category theory for
generalized modules for a conformal vertex algebra, V: Convergence condition
for intertwining maps and the corresponding compatibility condition,
arXiv:1012.4199.
	
\bibitem[HLZ6]{HLZ6}
Y.-Z. Huang, J. Lepowsky and L. Zhang, Logarithmic tensor category theory for
generalized modules for a conformal vertex algebra, VI: Expansion condition,
associativity of logarithmic intertwining operators, and the associativity
isomorphisms, arXiv:1012.4202.
	
\bibitem[HLZ7]{HLZ7}
Y.-Z. Huang, J. Lepowsky and L. Zhang, Logarithmic tensor category theory for
generalized modules for a conformal vertex algebra, VII: Convergence and
extension properties and applications to expansion for intertwining maps,
arXiv:1110.1929.
	
\bibitem[HLZ8]{HLZ8}
Y.-Z. Huang, J. Lepowsky and L. Zhang, Logarithmic tensor category theory for
generalized modules for a conformal vertex algebra, VIII: Braided tensor
category structure on categories of generalized modules for a conformal vertex
algebra, arXiv:1110.1931.

\bibitem[HY]{HY}
Y.-Z. Huang and J. Yang, Logarithmic intertwining operators and associative algebras, {\em J. Pure Appl. Algebra}
{\bf 216} (2012), no. 6, 1467--1492.

\bibitem[IK]{IK}
K. Iohara and Y. Koga, \textit{Representation Theory of the Virasoro Algebra}, Springer Monographs in Mathematics, Springer-Verlag London, Ltd., London, 2011, xviii+474 pp.

\bibitem[KaR]{KR}
S. Kanade and D. Ridout, NGK and HLZ: Fusion for physicists and mathematicians, \textit{Affine, Vertex and $W$-algebras}, 135--181, INdAM Series, \textbf{37}, Springer, Cham, 2019. 

\bibitem[KS]{KS}
M. Kashiwara and P. Schapira, \textit{Categories and Sheaves}, Grundlehren der Mathematischen Wissenschaften, \textbf{332}, Springer-Verlag, Berlin, 2006, x+497 pp.

\bibitem[Ka]{Ka}
H.~Kausch,
Extended conformal algebras generated by a multiplet of primary fields, \textit{Phys. Lett. B} \textbf{259} (1991), no. 4, 448--455.

\bibitem[KW]{KW}
D. Kazhdan and H. Wenzl, Reconstructing monoidal categories, \textit{I. M. Gel'fand Seminar}, 111--136, Adv. Soviet Math., \textbf{16}, Part 2, Amer. Math. Soc., Providence, RI, 1993.

\bibitem[KO]{KO} A. Kirillov, Jr. and V. Ostrik, On a $q$-analogue of the McKay correspondence and the $ADE$ classification of $\mathfrak{sl}_2$ conformal field theories, \textit{Adv. Math.} \textbf{171} (2002), no. 2, 183--227.

\bibitem[KyR]{KyR} 
K. Kyt\"{o}l\"{a} and D. Ridout, On staggered indecomposable Virasoro modules, \textit{J. Math. Phys.} \textbf{50} (2009), no. 12, 123503, 51 pp.

\bibitem[LL]{LL}
J. Lepowsky and H. Li, \textit{Introduction to Vertex
Operator Algebras and Their Representations}, Progress in Mathematics, \textbf{227}, Birkh\"{a}user Boston, Inc., Boston, MA, 2004. xiv+318 pp.

\bibitem[Li]{Li}
H. Li, Determining fusion rules by $A(V)$-modules and bimodules, \textit{J.
Algebra}  \textbf{212}  (1999), no. 2,  515--556.

\bibitem[Lin]{Lin}
X. Lin, Fusion rules of Virasoro vertex operator algebras, \textit{Proc. Amer. Math. Soc.} \textbf{143} (2015), no. 9, 3765--3776.

\bibitem[McR1]{McR} R. McRae, On the tensor structure of modules for compact orbifold vertex operator algebras, \textit{Math. Z.} \textbf{296} (2020), no. 1-2, 409--452.

\bibitem[McR2]{McR2} R. McRae, Twisted modules and $G$-equivariantization in logarithmic conformal field theory, \textit{Comm. Math. Phys.} \textbf{383} (2021), no. 3, 1939--2019.

\bibitem[Mil]{Mi}
A. Milas, Fusion rings for degenerate minimal models, \textit{J. Algebra} \textbf{254} (2002), no. 2, 300--335.

\bibitem[Miy]{Miy}
M. Miyamoto, $C_1$-cofiniteness and fusion products of vertex operator algebras, \textit{Conformal Field Theories and Tensor Categories}, 271--279, Math. Lect. Peking Univ., Springer, Heidelberg, 2014.

\bibitem[MRR]{MRR}
A. Morin-Duchesne, J. Rasmussen and D. Ridout, Boundary algebras and Kac modules for logarithmic minimal models, \textit{Nuclear Phys. B} \textbf{899} (2015), 677--769.

\bibitem[NT]{NT}
K. Nagatomo and A. Tsuchiya, The triplet vertex operator algebra $W(p)$ and the restricted quantum group $\overline{U}_q(sl_2)$ at $q=e^{\frac{\pi i}{p}}$, {\em Exploring New Structures and Natural Constructions in Mathematical Physics}, 1--49, Adv. Stud. Pure Math., \textbf{61}, Math. Soc. Japan, Tokyo, 2011.

\bibitem[Na]{Na}
W. Nahm, Quasi-rational fusion products, \textit{Internat. J. Modern Phys. B} \textbf{8} (1994), no. 25-26, 3693--3702.

\bibitem[Ne]{Ne}
C. Negron, Log-modular quantum groups at even roots of unity and the quantum Frobenius I, \textit{Comm. Math. Phys.} \textbf{382} (2021), no. 2, 773--814.

\bibitem[OH]{OH}
F. Orosz Hunziker, Fusion rules for the Virasoro algebra of central charge $25$, \textit{Algebr. Represent. Theory} \textbf{23} (2020), no. 5, 2013--2031.

\bibitem[PRZ]{PRZ}
P. Pearce, J. Rasmussen and J.-B. Zuber, Logarithmic minimal models, \textit{J. Stat. Mech. Theory Exp.} 2006, no. 11, P11017, 36 pp.

\bibitem[Ra]{Ra}
J. Rasmussen, Classification of Kac representations in the logarithmic minimal models $\mathcal{LM}(1,p)$, \textit{Nuclear Phys. B} \textbf{853} (2011), no. 2, 404--435.

\bibitem[RP]{RP}
J. Rasmussen and P. Pearce, Fusion algebras of logarithmic minimal models, \textit{J. Phys. A} \textbf{40} (2007), no. 45, 13711--13733.

\bibitem[RS]{RS}
N. Read and H. Saleur, Associative-algebraic approach to logarithmic conformal field theories, \textit{Nuclear Phys. B} \textbf{777} (2007), no. 3, 316--351.

\bibitem[Se]{S} C. Seshadri, Space of unitary vector bundles on a compact Riemann surface, \textit{Ann. of Math. (2)} \textbf{85} (1967), 303--336.

\bibitem[TW]{TW}
A. Tsuchiya and S. Wood, The tensor structure on the representation category of the $\cW_{p}$ triplet algebra, \textit{J. Phys. A} \textbf{46} (2013), no. 44, 445203, 40 pp.

\bibitem[Wa]{Wa}
W. Wang, Rationality of Virasoro vertex operator algebras, \textit{Internat. Math. Res. Notices} 1993, no. 7, 197--211.

\bibitem[Zh]{Zh} Y. Zhu, Modular invariance of characters of vertex operator algebras, \textit{J. Amer.
Math. Soc.} \textbf{9} (1996), no. 1, 237--302.

\end{thebibliography}
\end{document}